\documentclass[reqno,11pt]{amsart}

\usepackage{a4wide}
\usepackage{color}
\usepackage{mathrsfs}
\usepackage{mathtools}
\usepackage{amsmath}
\usepackage{amssymb}
\usepackage{enumitem}
\usepackage{bbm}
\usepackage{esint}
\usepackage{nicefrac}
\numberwithin{equation}{section}
\usepackage[colorlinks,citecolor=green,linkcolor=red]{hyperref}

\usepackage[latin1]{inputenc}

\newtheorem{theorem}{Theorem}[section]
\newtheorem{corollary}[theorem]{Corollary}
\newtheorem{lemma}[theorem]{Lemma}
\newtheorem{proposition}[theorem]{Proposition}
\newtheorem{definition}[theorem]{Definition}

\newtheorem{remark}[theorem]{Remark}

\newcommand{\N}{\mathbb{N}}

\newcommand{\R}{\mathbb{R}}
\renewcommand{\S}{\mathbb{S}}
\newcommand{\Z}{{\rm Z}}
\newcommand{\sfd}{{\sf d}}

\newcommand{\restr}[1]{\lower3pt\hbox{$|_{#1}$}}
\newcommand{\eps}{\varepsilon}  
\newcommand{\nchi}{{\raise.3ex\hbox{$\chi$}}}
\newcommand{\weakto}{\rightharpoonup}

\newcommand{\fr}{\hfill$\blacksquare$}  
\newcommand{\LIP}{\mathrm{LIP}}
\newcommand{\Lip}{\mathrm{Lip}}
\newcommand{\lip}{\mathrm{lip}}

\newcommand{\diam}{\mathrm{diam}}

\newcommand{\RCD}{\mathrm{RCD}}
\newcommand{\CD}{\mathrm{CD}}
\newcommand{\mm}{\mathfrak m}

\renewcommand{\limsup}{\varlimsup}
\renewcommand{\liminf}{\varliminf}
\renewcommand{\d}{{\rm d}}
\newcommand{\X}{{\rm X}}
\newcommand{\Y}{{\rm Y}}
\newcommand{\Xdm}{(\X,\sfd,\mm)}
\newcommand{\rmCh}{{\rm Ch}}
\newcommand{\eucl}{{\sf Eucl}}
\newcommand{\vol}{{\rm Vol}}

\newcommand{\supp}{{\rm supp}}

\newcommand{\Per}{{\rm Per}}
\newcommand{\DDelta}{{\mbox{\boldmath$\Delta$}}}
\newcommand{\PP}{\mathscr{P}}

\newcommand{\la}{{\big\langle}}                  % brackets
\newcommand{\ra}{{\big\rangle}}
\newcommand{\bd}{{\bf \Delta}}

\renewcommand{\phi}{\varphi}
\newcommand{\avr}{{\sf AVR}}
\newcommand{\I}{{\rm I}}
\newcommand{\II}{{\rm II}}
\newcommand{\III}{{\rm III}}

\title[]{Stability of Sobolev inequalities on Riemannian manifolds with  Ricci curvature lower bounds}
\address{University of Jyv\"askyl\"a, Department of Mathematics and Statistics, P.O. Box 35 (MaD), FI-40014 University of Jyv\"askyl\"a, Finland}

\author[]{Francesco Nobili}

\author[]{Ivan Yuri Violo}
\email{francesco.nobili@dm.unipi.it, ivan.violo@sns.it}

\begin{document}
\begin{abstract}
We study the qualitative stability of two classes of Sobolev inequalities on Riemannian manifolds. In the case of positive Ricci curvature, we prove that an almost extremal function for the sharp Sobolev inequality is close to an extremal function of the round sphere. In the setting of non-negative Ricci curvature and Euclidean volume growth, we show an analogous result in comparison with the extremal functions in the  Euclidean Sobolev inequality. As an application, we deduce a stability result for minimizing Yamabe metrics.  The arguments rely on a generalized Lions' concentration compactness on varying spaces and on rigidity results of Sobolev inequalities on singular spaces.
\end{abstract}
\maketitle
\allowdisplaybreaks
\setcounter{tocdepth}{2}
\tableofcontents

\section{Introduction}
The sharp Sobolev inequality on the standard round sphere $\mathbb{S}^n$,   $n>2,$ reads as
\begin{equation}
\|u\|_{L^{2^*}}^2\le \frac{2^*-2}{n} \| \nabla u\|_{L^{2}}^2 + \| u\|_{L^2}^2, \qquad \forall u \in W^{1,2}(\S^n),
\label{eq:Sob sphere}
\end{equation}
where $2^* \coloneqq 2n /(n-2)$ and  the norms are computed with the renormalized volume measure $\frac{\vol_{\S^n}}{\vol_{\S^n}(\S^n)}$. This inequality goes back to the work of Aubin  \cite{Aubin76-3}, who also characterized non-constant  extremizers (see also \cite[Chapter 5]{Hebey99}) having the following expression (denoting by $\sfd$ the distance induced by the metric):
\begin{equation}
 u \coloneqq \frac{a}{(1-b\cos(\sfd(\cdot ,z_0))^{\frac{n-2}{2}}},\qquad \text{with } a\in\R,\, b\in(0,1),\, z_0 \in \S^n.
\label{eq:Aubin extremal intro}
\end{equation}
We will refer to them as \emph{spherical bubbles}. A natural question is  the one of stability:
\begin{itemize}
    \item[$(Q)$] Is a function satisfying almost equality in \eqref{eq:Sob sphere} close to a spherical bubble?
\end{itemize}
Up to a change of coordinates via the stereographic projection (see e.g.\ \cite{Lieb,DE20,EngelsteinNeumayerSpolaor22}), this question is  equivalent to the stability of the  Euclidean Sobolev inequality
\begin{equation}
\|u\|_{L^{2^*}(\R^n)} \le \eucl(n,2)\| \nabla u\|_{L^2(\R^n)}, \qquad \forall u \in \dot W^{1,2}(\R^n),
\label{eq:SobEuclidea}
\end{equation}
where $\dot W^{1,2}(\R^n) \coloneqq \{ u \in L^{2^*}(\R^n) \colon |\nabla u|\in L^2(\R^n)\}$ and $\eucl(n,2)>0$ is the sharp constant, computed by Aubin \cite{Aubin76-2} and Talenti \cite{Talenti76} (see \eqref{eq:eucln2} for its precise value%and also \cite{C-ENV04} for a proof of \eqref{eq:SobEuclidea} via optimal transport
). Extremizers, i.e.\ functions $u$ for which equality occurs in \eqref{eq:SobEuclidea}, are also in this case completely characterized:
\begin{equation}
u(x)\coloneqq \frac{a}{(1+b|x-z_0|^2)^{\frac{n-2}{2}}}, \qquad a \in \R,\, b>0,\, z_0 \in \R^n.
\label{eq:bubble intro}
\end{equation}
We shall refer to these functions as \emph{Euclidean bubbles} (usually called Talenti or Aubin-Talenti bubbles). The first \emph{quantitative} stability result was obtained by Bianchi and Egnell \cite{BianchiEgnell91} who showed that
\begin{equation}
    \inf \frac{\| \nabla( u -w )\|_{L^2(\R^n)}}{\|\nabla u\|_{L^2(\R^n)}} \le  C_n \Big(\frac{\| \nabla u\|_{L^2(\R^n)} }{ \|u\|_{L^{2^*}(\R^n)} }-\eucl(n,2)^{-1} \Big)^{\frac 12}, \qquad \forall u \in \dot W^{1,2}(\R^n), \label{eq:strong stability Eucl}
\end{equation}
for a dimensional constant $C_n>0$ and the infimum taken among all $w$ as in \eqref{eq:bubble intro}. This stability is \emph{strong}, in the sense that the $L^2$-norm of the difference of gradients is the biggest possible norm that can be controlled, and optimal, as the exponent $1/2$ is sharp.  We mention that quantitative stability for the case of the $p$-Sobolev inequality in $\R^n$ has also been obtained in sharp form (see \cite{CianchiFuscoMaggiPratelli09,FigalliNeumayer19,Neumayer19,FigalliZhang22}). The stability of \eqref{eq:SobEuclidea} in \emph{qualitative} form, meaning that if the right-hand side of \eqref{eq:strong stability Eucl} is small then so is the left-hand side (in a non-quantified sense), can be deduced  via concentration compactness  \cite{Lions84,Lions85}.

\medskip

In this note, we address the analogous stability of $(Q)$ for Sobolev inequalities on more general Riemannian manifolds.

Let us consider a closed $n$-dimensional  Riemannian manifold $(M,g)$, $n>2,$ satisfying
\[
{\rm Ric}_g \ge (n-1)g.
\]
Under these assumptions the same Sobolev inequality \eqref{eq:Sob sphere} as in the sphere  holds \cite{Said83}:
\begin{equation}
\|u\|_{L^{2^*}}^2\le \frac{2^*-2}{n} \| \nabla u\|_{L^{2}}^2 + \| u\|_{L^2}^2, \quad \forall u \in W^{1,2}(M),
\label{eq:Sob main intro}
\end{equation}
where the norms are  with the renormalized volume measure. Proofs of this inequality using different methods are also given in \cite{Bakry94,BL96,Fontenas97,Hebey99,BakryGentilLedoux14,DGZ20}. We can ask an analogous stability:

\begin{itemize}
    \item[$(Q')$] Is a function satisfying almost equality in \eqref{eq:Sob main intro}  close to a spherical bubble?
\end{itemize}
Almost equality here means that
\[
   \mathcal{Q}(u)\coloneqq\frac{ \|u\|_{L^{2^*}}^2 - \|u\|^2_{L^2} }{\|\nabla u\|^2_{L^2} } \sim \frac{2^*-2}{n} .
\]
In the previous work \cite{NobiliViolo21}, we proved that if $  |{\mathcal{Q}}(u)-\frac{2^*-2}{n}|$ is small, then $M$ is qualitatively close in the measure Gromov-Hausdorff sense to a spherical suspension, which roughly said is a possibily-singular generalization of the round sphere. In particular, when $\sup \mathcal{Q}(u) = n^{-1}(2^*-2)$, rigidity occurs, i.e.\  $M$ is isometric to $\mathbb{S}^n$. These facts already  suggested an affirmative answer to $(Q')$ and in fact here we will confirm that this is indeed the case. More precisely, for $M$ as above, every $a \in \R$, $b \in [0,1)$ and $z \in M$, set
\begin{equation}\label{eq:bubbles}
    w_{a,b,z}(\cdot)\coloneqq \frac{a}{(1 - b\cos(\sfd(\cdot ,z_0))^{\frac{n-2}{2}}},
\end{equation}
with the convention that $w_{a,0,z}\equiv a.$ Our main result is then the following (as before, all the norms are with respect to the renormalized volume measure):
\begin{theorem}\label{thm:qualitative SobCompact intro}
For every $\eps>0$ and $n >2$ there exists $\delta\coloneqq \delta(\eps,n)>0$ such that the following  holds. Let $(M,g)$ be an $n$-dimensional Riemannian manifold with ${\rm Ric}_g\ge (n-1)g$  and suppose there exists $u\in W^{1,2}(M)$ non-constant satisfying 
\begin{equation}\label{eq:almost extremal intro}
   {\mathcal{Q}}(u) >  \frac{2^*-2}{n}-\delta.
\end{equation}
Then, there exist $a \in \R$, $b \in [0,1)$ and $z \in M$ such that
 \begin{equation}\label{eq:close to bubble}
       \frac{\| \nabla(u- w_{a,b,z} )\|_{L^2} + \|u-w_{a,b,z}\|_{L^{2^*}}}{\|u\|_{L^{2^*}}} \le \eps.
 \end{equation}
Moreover, if $w_{a,b,z}\equiv a $  (i.e.\ $b=0$), then $a\in \R$ can be chosen  so that the reminder
$$R\coloneqq u-a$$ 
satisfies
\begin{equation}
    \| R\cdot \|R\|_{L^2}^{-1}- \sqrt{N+1}\cos(\sfd(\cdot,p))\|_{L^2}\le C_n (\eps^\alpha+\delta)^{\beta},
\end{equation}
for some $p \in M$ and positive constants $\alpha, \beta, C_n$ depending only on $n.$
\end{theorem} 
The above theorem is the first stability result for the Sobolev inequality that covers a wide class of Riemannian manifolds; indeed up to our best knowledge only very special symmetric cases had been studied so far: see \cite{BhaktaGangulyKarmakarMazumdar22} 
 for the hyperbolic space and \cite{Frank21} for $\mathbb{S}^1(1/\sqrt{d-2})\times \mathbb{S}^{n-1}(1)$.

Some comments on the above statement are in order.
\begin{enumerate}[label=\roman*)]
\item The value of $\delta$ depends only on $n$ and $\eps>0$, but not on the manifold $M$. Moreover, up to scaling, an analogous statement holds assuming ${\rm Ric}_g\ge K$ for some $K>0$, with $\delta$ depending also on $K$.
\item Even if Theorem \ref{thm:qualitative SobCompact intro} is stated completely in the smooth-setting, its proof will require the study of the Sobolev inequality also in singular spaces (see below the strategy for more details).
\item  The result \eqref{eq:close to bubble} actually holds under a slightly weaker assumption than \eqref{eq:almost extremal intro}, namely:
\begin{equation}\label{eq:stronger}
    \|u\|_{L^{2^*}(\vol_g)}^2\ge A \| \nabla u\|_{L^{2}(\vol_g)}^2 + B\| u\|_{L^2(\vol_g)}^2,
\end{equation}
with $|A-\frac{2^*-2}{n}|+|B-1|\le \delta$ (see Remark \ref{rmk:weaker assumptino}).
\item The first part of Theorem \ref{thm:qualitative SobCompact intro} holds also restricting to the class of non constant spherical bubbles, that is $w_{a,b,z}$ with $b \neq 0$.
\item The second part of Theorem \ref{thm:qualitative SobCompact intro} should be read as follows: if the almost extremal function $u$ is close to a constant, then (up to changing the constant)  the reminder is close in $L^2$-sense to a cosine of the distance. Thus, since
\[
 1+\eps \cos(\sfd) \sim \frac{1}{(1+\eps \cos(\sfd))^{\frac{n-2}{2}}},
\]
this means that $u$ still retains, at a `second-order' approximation, the shape of a spherical bubble. This extra information essentially  comes from the fact that the linearization of  the Sobolev inequality is the Poincar\'e inequality,  which means that   plugging in \eqref{eq:Sob main intro} functions of the type $1+\eps f$ and sending $\eps \to 0$  gives the sharp Poincar\'e inequality for $f$ (see e.g.\ \cite[Lemma 6.7]{NobiliViolo21}). Therefore if $1+\eps f$ satisfies almost equality in \eqref{eq:Sob main intro}, then $f$ almost satisfies equality in the sharp Poincar\'e inequality and thus should be close to a cosine of the distance (see \cite{CavalettiMondinoSemola19}).

    \item When $M$ is not the round sphere, the existence of an extremizer, that is a function which maximizes  ${\mathcal{Q}}(u)$, is unknown in general. This question is contained in \cite[Question 4B, Pag. 120]{Hebey99} as part of the so-called $AB$-program around Sobolev inequalities on general Riemannian manifolds. In this direction, we mention the Sobolev-alternative statement proved in \cite[Theorem 6.8]{NobiliViolo21}.    
    
    Nevertheless, thanks to the above theorem, we are able to say something about the shape of functions for which this ratio is large, i.e.\  satisfying \eqref{eq:almost extremal intro}.
\end{enumerate}

\begin{remark}
\rm
Note that above we deal only with $p=2$. The reason is that the inequality
\begin{equation}
\|u\|_{L^{p^*}}^p\le A\| \nabla u\|_{L^{p}}^p + \| u\|_{L^p}^p, \qquad \forall u \in W^{1,p}(M),
\end{equation}
is false for any $p>2$, $A>0$ and any $(M,g)$ closed manifold (see \cite[Prop. 4.1]{Hebey99}).
\fr
\end{remark}
As an application of Theorem \ref{thm:qualitative SobCompact intro}, we prove a stability-type result for minimizing  Yamabe metrics. Recall that a solution to the Yamabe problem on a Riemannian manifold $(M,g)$ is a smooth positive function $u$ such that the metric $u^{\frac{4}{n-2}}g$ has constant scalar curvature (see  \cite{Yamabe60} and also the surveys \cite{LP87,BrendleMarques11}). After the works \cite{Trudinger68, Aubin76-3, Schoen84} it is known that a solution exists on every closed Riemannian manifold 
and that can be found as a minimizer of
\begin{equation}
   Y(M,g)\coloneqq \inf_{\substack{u \in W^{1,2}(M) \\  u \neq 0}} {\mathcal{E}}(u)\coloneqq  \inf_{\substack{u \in W^{1,2}(M) \\ u \neq 0}} n(n-1) \frac{ \int \frac{2^*-2}{n}|\nabla u|^2  + \frac{{\rm Scal}_g}{n(n-1)} u^2 \d \vol_g }{\big(\int |u|^{2^*}\, \d \vol_g\big)^{2/2^*}},
\label{eq:Yamabe constant}
\end{equation}
where ${\rm Scal}_g$ is the scalar curvature of $g$ and $\vol_g$ is the (non-renormalized) volume measure.  $ Y(M,g)$ is a called Yamabe constant of $(M,g)$ and it is  a conformal invariant. Note that in the case of $\mathbb{S}^n$, the minimizers of ${\mathcal{E}}(u)$ are precisely the spherical bubbles in \eqref{eq:Aubin extremal intro}.
\begin{corollary}\label{cor:yamabe}
For every $n >2$ and $\eps>0$ there exists $\delta\coloneqq \delta(\eps,n)>0$ such that the following holds. 
Let $(M,g)$ be an $n$-dimensional Riemannian manifold with ${\rm Ric}_g\ge (n-1)g$ and $u \in W^{1,2}(M)$ non-zero such that 
\begin{equation}\label{eq:close to sphere}
    \sfd_{GH}(M,\mathbb{S}^n)\le \delta, \qquad |\mathcal{E}(u)- Y(M,g)|\le \delta.
\end{equation}
Then, there exist $a\in\R,b\in (0,1)$ and $z_0 \in M$ satisfying
\[
   \frac{\| u -w_{a,b,z} \|_{W^{1,2}}}{\|u\|_{W^{1,2}}} \le  \eps,
\] 
where $w_{a,b,z}$ is as in \eqref{eq:bubbles}.
\end{corollary}
Here $\sfd_{GH}$ denotes the Gromov-Hausdorff distance. 
A similar stability for almost minimizers of  $\mathcal{E}(\cdot)$ has been recently proved in \cite{EngelsteinNeumayerSpolaor22} in quantitative form and under no assumptions on the metric. The novelty here is that we have a comparison with an explicit class of functions, while in \cite{EngelsteinNeumayerSpolaor22} no information is known about the shape of the minimizers.

\medskip 

We discuss now a second stability result on non-compact Riemannian manifolds. Our motivations come from the fact that, to prove Theorem \ref{thm:qualitative SobCompact intro}, non-compact setting will naturally arise in our investigation (see below the main strategy of proof).

Let us consider an $n$-dimensional Riemannian manifolds $(M,g)$, $n>2$, satisfying
\begin{equation}
    {\rm Ric}_g \ge 0, \qquad {\sf AVR }(M)\coloneqq \lim_{R\to\infty}\frac{{\rm Vol}(B_{R}(x))}{\omega_n R^n }>0, \label{eq:AVR}
\end{equation}
for $x \in M$. The latter condition is called \emph{Euclidean volume growth} property and ${\sf AVR }(M)$ is the asymptotic volume ratio. Notice that the limit exists and is independent of $x$, by the Bishop-Gromov inequality.

 In \cite{BaloghKristaly21}, the following sharp Euclidean-type Sobolev inequality was derived under the assumptions \eqref{eq:AVR}:
\begin{equation}
\| u\|_{L^{2^*}} \le {\sf AVR }(M)^{-\frac{1}{n}}\eucl(n,2)\| \nabla u\|_{L^{2}},\qquad \forall u \in \dot W^{1,2}(M).
\label{eq:sharp sob intro}
\end{equation}
Moreover, they proved that equality occurs in \eqref{eq:sharp sob intro} for some non-zero function $u\in \dot W^{1,2}(M)$, then $M$ is isometric to $\R^n$ and $u$ is in particular an Euclidean bubble. Actually in \cite{BaloghKristaly21} this rigidity requires also $u \in C^n(M)$ and $u\ge 0$, however these additional assumptions can be removed after the results in \cite{AntonelliPasqualettoPozzettaSemola22} and \cite{CavallettiManini22} (see also Theorem \ref{thm:rigidity sharp Sob}).  

The natural stability question is what happens if a function satisfies almost equality in \eqref{eq:sharp sob intro}. Clearly, differently from \eqref{eq:Sob main intro}, we cannot deduce anything about the geometry of $M$. Indeed the inequality is sharp on every $M$ as in \eqref{eq:AVR}, which means that we can always find functions so that
$\frac{\| u\|_{L^{2^*}}}{\| \nabla u\|_{L^2}}$
is arbitrary close to  ${\sf AVR }(M)^{-\frac{1}{n}}\eucl(n,2)$. We can prove however that a function for which almost equality occurs in \eqref{eq:sharp sob intro} is close to a Euclidean bubble. Set
\[
v_{a,b,z}\coloneqq \frac{a}{(1+b \sfd(\cdot ,z)^2)^{\frac{n-2}{2}}},\qquad \text{for }a \in \R,b >0,z \in M.
\]
\begin{theorem}\label{thm:qualitative SobAVR intro} 
For every $\eps>0,V\in(0,1)$ and $n>2,$ there exists $\delta\coloneqq \delta(\eps,n,V)>0$ such that the following holds.
Let $(M,g)$ be an $n$-dimensional Riemannian manifold as in \eqref{eq:AVR} with ${\sf AVR}(M)\ge V$ and assume there exists $u\in \dot W^{1,2}(M)$ non-zero satisfying
\[
  \frac{\| u\|_{L^{2^*}}}{\| \nabla u\|_{L^2}}> {\sf AVR}(M)^{-\frac 1n}\eucl(n,2)- \delta.
\]
Then, there exist $a \in \R,\, b>0,$ and $z \in M$ so that
\[
    \frac{\| \nabla( u -v_{a,b,z} )\|_{L^2}}{\|\nabla u\|_{L^2}} \le  \eps.
\] 
\end{theorem}
 Notice that the stability  is strong in the sense that we control the gradient norm as in the Euclidean case \eqref{eq:strong stability Eucl}.

 A direct consequence of the above theorem is:
 \begin{corollary}\label{cor:introcorollary}
Let $(M,g)$ be an $n$-dimensional Riemannian manifold as in \eqref{eq:AVR}. Then
\[
 {\sf AVR}(M)^{\frac 1n}\eucl^{-1}(n,2)  =  \inf_{a \in \R,\, b>0,\, z \in M} \frac{\| \nabla v_{a,b,z} \|_{L^2}}{\|  v_{a,b,z} \|_{L^{2^*}}}.
\]
\end{corollary}
\begin{remark}
\rm 
Our main  results in Theorem \ref{thm:qualitative SobCompact intro} and Theorem \ref{thm:qualitative SobAVR intro}, even if stated on smooth Riemannian manifolds, actually hold also in the context of weighted Riemannian manifolds and more generally in the singular setting of metric measure spaces with a synthetic Ricci curvature lower bound. The generalized version of these statements can be found in Theorem \ref{thm:qualitative Optimal Sobolev RCD} and Theorem \ref{thm:qualitative SobAVR RCD}.
\fr
\end{remark}
% To prove Theorem \ref{thm:qualitative SobCompact intro}, we will end up studying Sobolev inequalities of Euclidean type in non-compact settings. This is in analogy with case of the sphere where the stability can be reduced via stereographic projection to that of $\R^n$. Moreover, we will also need to consider not only Euclidean spaces but also singular spaces such as metric cones. See below for a review of our main strategy of proof.
%
%
\noindent{\bf Strategy of proof and non-smooth setting}. We outline the argument for Theorem \ref{thm:qualitative SobCompact intro} (Theorem \ref{thm:qualitative SobAVR intro} is simpler and follows by the same strategy).  The underlying idea is classical, that is  to argue by contradiction and concentration compactness. However, the novelty is that the space is not homogeneous and also not fixed, since we need to deal with a whole class of Riemannian manifolds. Moreover, singular and non-compact limit spaces must also be considered. In particular, the whole analysis will be carried out in the more general setting of $\RCD$ spaces, which are metric measure spaces with a synthetic notion of Ricci curvature bounded below (see Section \ref{sec:prelim} for details and references).

Suppose that Theorem \ref{thm:qualitative SobCompact intro} is false. Then, there  exist $\eps>0,$ a sequence $\{M_k\}_{k \in \N}$ of $n$-dimensional Riemannian manifolds with ${\rm Ric}_k\ge n-1$ and non-constant functions $u_k\colon M_k \to \R$, $\|u_k\|_{L^{2^*}}=1,$ which satisfy \eqref{eq:almost extremal intro} for some  $\delta_k\downarrow 0$, but so that for any $k\in\N$
\begin{equation}
    \inf \| u_k- w \|_{L^{2^*}} + \| \nabla( u_k- w)\|_{L^2} > \eps,
\label{eq:contradiction intro}
\end{equation}
where the inf runs among all spherical bubbles $w = a(1 - b\cos(\sfd_k(\cdot,z))^{\frac{2-n}{2}}$ ($\sfd_k$ being the distance on $M_k$). Similarly to the classical concentration compactness  \cite{Lions84,Lions85} in $\R^n$,  we choose  points $y_k \in M_k$ and constants $\sigma_k>0$ so that, defining
\begin{equation}
(Y_k,\rho_k,\mu_k)\coloneqq (M_k,\sigma_k\sfd_k,{\rm Vol}_k(M_k)^{-1}\sigma_k^n{\rm Vol}_k),\qquad 
 u_{\sigma_k} = \sigma_k^{-n/2^*}u_k,
\label{eq:transformation intro}
\end{equation}
we have
\[
\int_{B_1^{Y_k}(y_k)} |u_{\sigma_k}|^{2^*} \, \d \mu_k=\frac12,
\]
(in the actual proof we choose a suitable constant close to $1$). The spaces $(Y_k,\rho_k,\mu_k)$ are in particular metric measure spaces which are rescalings of the original manifolds $M_k$. Note that it can happen that $\sigma_k \uparrow \infty$,  which  corresponds to a concentrating behavior of the sequence $u_k$. In this case, the diameter of $Y_k$ goes to infinity and we are in a sense performing a blow-up along $M_k.$

Thanks to Gromov's precompactness theorem \cite{Gromov07} it is possible to show that, up to a subsequence, $(Y_k,\rho_k,\mu_k,y_k)$ converges in the pointed-measure-Gromov-Hausdorff sense to a limit $\RCD$ space $(Y,\rho,\mu, \bar y)$ (which might be non-smooth). Using a generalized version of  Lions' concentration compactness for a sequence of $\RCD$ spaces (see Section \ref{sec:gen extremals}), we show that up to a further subsequence, $u_{\sigma_k}$ converges $L^{2^*}$-strongly (on varying spaces, see Definition \ref{def:lpconv} below) to some $u \in L^{2^*}(\mu)$. It also follows that $u$ is extremal for a `limit Sobolev inequality' on $Y$, that might be both as in \eqref{eq:Sob main intro} or of Euclidean-type as in \eqref{eq:sharp sob intro}, depending if there is concentration or not along the original sequence $u_k$. The key point is proving:
\begin{align*}
    &\text{Concentration}  &\Rightarrow \qquad &Y \text{ is a metric-cone and } u \text{ is a Euclidean bubble} \\
    &\text{Non-concentration} &\Rightarrow \qquad &Y \text{ is a spherical suspension and } u \text{ is a spherical bubble}
\end{align*}
We will show these two facts by proving suitable  rigidity theorems for the Sobolev inequalities on $\RCD$ spaces (see Section \ref{sec:rigidity extremal}). The proof will be then completed by carefully bringing back this information from $u$ to the sequence $u_k$ to find a contradiction with \eqref{eq:contradiction intro}. It is worth noticing that, in case of concentration, the scaled functions $u_{\sigma_k}$ tend to a Euclidean bubble but, to reach a contradiction, the original sequence $u_k$ must be close to the family of spherical bubbles. This turns out to be true because a concentrated spherical bubble looks locally, around the point where it is concentrated, like a Euclidean bubble (see Lemma \ref{lem:stability when concentrating}). 

\medskip

We conclude this introduction by mentioning that generalized concentration compactness techniques on varying spaces, in a similar spirit to the present work, have been recently developed in \cite{AntonelliFogagnoloPozzetta21,AntonelliNardulliPozzetta22} and applied to study the problem of existence of isoperimetric regions on non-compact Riemannian manifolds \cite{AntonelliBrueFogagnoloPozzetta22}.

\section{Preliminaries }\label{sec:prelim}

\subsection{Calculus on metric measure spaces}
A metric measure space is a triple $\Xdm$, where $(\X,\sfd)$ is a complete and separable metric space and $ \mm\neq 0$ is a non-negative and boundedly finite Borel measure. Two metric measure spaces are \emph{isomorphic}, provided there exists a measure preserving isometry between them. To avoid technicalities, we will always assume $\supp(\mm)=\X$.  We will denote by $\LIP(\X)$ and $\LIP_{bs}(\X)$ respectively the space of Lipschitz functions and Lipschitz functions with bounded support in $(\X,\sfd)$.
We recall the notion of local lipschitz constant of a Lipschitz function $f\in \LIP(\X)$:
\[ \lip \, f(x) \coloneqq \limsup_{y\to x}\frac{|f(y)-f(x)|}{\sfd(x,y)},
\]
set to $+ \infty$ if $x$ is isolated. The Sobolev space on a metric measure space was  introduced in \cite{Cheeger00} and \cite{Shan00} (inspired by the notion of upper gradient \cite{KoskelaHeinonen96,KoskelaHeinonen98}). Here we follow the axiomatization of \cite{AmbrosioGigliSavare11-3} (equivalent to that of \cite{Shan00,Cheeger00}). 

Let $\Xdm$ be a metric measure space and define the Cheeger energy $\rmCh \colon L^2(\mm)\to [0,\infty]$ 
\[ \rmCh(f) \coloneqq  \inf \Big\{ \liminf_{n\to\infty} \int \lip^2  f_n\, \d \mm \colon (f_n) \subset L^2(\mm)\cap \LIP(\X), f_n\to f \text{ in } L^2(\mm) \Big\}.\]
The Sobolev space is defined as $W^{1,2}(\X)\coloneqq \{f \in L^2(\mm) \colon \rmCh(f) <\infty\}$ and equipped with the norm $\| f\|^2_{W^{1,2}(\X)} \coloneqq \|f\|_{L^2(\mm)}^2 + \rmCh(f)$ turning it into a Banach space. We recall also (see e.g.\ \cite{AmbrosioGigliSavare11-3}) that for every $f \in W^{1,2}(\X)$ there exists a minimal $\mm$-a.e.\ object $| \nabla f| \in L^2(\mm)$ called minimal weak upper gradient so that
\[ \rmCh(f) = \int |\nabla f|^2\, \d \mm.\]
To lighten the notation, we will often write $\|\nabla f\|_{L^{2}(\mm)}$ in place of $\||\nabla f|\|_{L^{2}(\mm)}$. We shall often use the \emph{locality} of minimal weak upper gradients:
\[
|\nabla f| =|\nabla g|,\qquad \mm\text{-a.e.\ in } \{f =g\}.
\]
for every $f,g \in W^{1,2}(\X)$. For $\Omega\subset \X$ open we say that $f \in W^{1,2}_{loc}(\Omega)$, provided $\eta f \in W^{1,2}(\X)$ for every $\eta \in \LIP_{bs}(\X)$ with $\sfd(\supp(\eta),\X \setminus \Omega)>0$. By locality, the object
\[ |\nabla f| \coloneqq  |\nabla (\eta f)|,\qquad \mm\text{-a.e. on }\{ \eta = 1\},\]
is well defined as an $L^2_{loc}(\Omega)$-function and will be called again minimal weak upper gradient.  It can be easily checked that if $f \in W^{1,2}_{loc}(\X)$ with $f,|\nabla f|\in L^2(\mm)$, then  $f \in W^{1,2}(\X)$.

We shall need also the following semicontinuity result: 
\begin{equation}
\begin{array}{l}
     f_n \in W^{1,2}_{loc}(\X),\  f_n \to f\ \mm\text{-a.e.}  \\
     \liminf_n\| \nabla f_n \|_{L^2(\mm)}<\infty  
\end{array}
  \quad \Rightarrow \quad 
  \begin{array}{l}
  f \in W^{1,2}_{loc}(\X), |\nabla f|\in L^2(\mm)\\
  \|  \nabla f \|_{L^2(\mm)}\le\liminf_n\| \nabla f_n\|_{L^2(\mm)}
  \end{array}\label{eq:W12loc lsc}    
\end{equation}
The $W^{1,2}_{loc}$ regularity can be directly proved by appealing to the semicontinuity (see, e.g., \cite[Prop 2.1.13]{GP20}) in the space $W^{1,2}(\X)$ and a cut-off argument. The fact that $|\nabla f|\in L^2(\mm)$ follows by noticing that, for any ball $B\subset \X$, $\int_B |\nabla f|^2\,\d\mm \le \int_B |\nabla (\eta f)|^2\,\d\mm \le \liminf_n \| \nabla f_n \|_{L^2(\mm)}$, where $\eta \in \LIP_{c}(\X)^+$ with $\eta \equiv 1$ on $B$, having used twice the locality of the minimal weak upper gradient  and again \cite[Prop 2.1.13]{GP20}. This proves \eqref{eq:W12loc lsc} by arbitrariness of $B$.

For $\Omega\subseteq \X$ open, we define the Sovolev space of functions vanishing at the boundary $W^{1,2}_0(\Omega)\subset W^{1,2}(\X)$ as the closure or $\LIP_c(\Omega)$ with respect to the $W^{1,2}$ norm.

A metric measure space is called infinitesimally Hilbertian \cite{Gigli12} provided 
\[ |\nabla (f+g)|^2+ |\nabla (f-g)|^2 = 2|\nabla f|^2 + 2|\nabla g|^2, \qquad \mm\text{-a.e.,\ }\forall f,g \in W^{1,2}(\X),\]
or equivalently if $W^{1,2}(\X)$ is Hilbert. This allows defining a formal scalar product between gradients of Sobolev functions by polarization
\begin{equation}
    \la \nabla f, \nabla g\ra \coloneqq |\nabla f|^2+|\nabla g|^2 - |\nabla (f-g)|^2 \in L^1(\mm), \qquad \forall f,g \in W^{1,2}(\X),\label{eq:cosine rule}
\end{equation}
that is bilinear on its entries. By locality, it is possible to consider also a scalar product for functions in $W^{1,2}_{loc}(\Omega)$.

We recall next the \emph{measure-valued Laplacian} as in \cite{Gigli12}, in the case of $\X$ proper and infinitesimally Hilbertian. We say that $f \in W^{1,2}_{loc}(\Omega)$ has a measure-valued Laplacian on $\Omega$, and we write $f \in D(\bd,\Omega)$, provided there exists a (signed) Radon measure $\mu$ such that
\[ \int g \, \d \mu = - \int \la \nabla f,\nabla g\ra \, \d \mm, \qquad \forall g \in \LIP_c(\Omega).\]
Here signed Radon measure means difference of two positive Radon measures (see also \cite{CavallettiMondino20} for a related  discussion). The unique measure $\mu$ satisfying the above is denoted by $\bd f$ and depends linearly on $f$. If $\Omega=\X$ we simply write $f \in D(\bd)$. Moreover, if $\DDelta f \ll \mm$, we  write $\Delta f \coloneqq\frac{ \d  \DDelta f}{\d \mm}\in L^1_{loc}(\Omega)$.

Next, we introduce the sets of finite perimeter following  \cite{ADM14,Miranda03}.
For $E \subset \X$ Borel and $A\subset \X$ open, define
\[  {\rm Per}(E,A)\coloneqq  \inf \Big\{ \liminf_{n\to \infty} \int_A \lip \ f_n\, \d \mm \colon f_n \subset \LIP_{loc}(A), f_n \to \nchi_E \text{ in } L^1_{loc}(A) \Big\}.\]
If ${\rm Per}(E,\X)<\infty$ we say that $E$ has finite perimeter. In this case, the map $A\mapsto {\rm Per}(E,A)$ is the restriction to open sets of a non-negative finite Borel measure called the perimeter measure of $E$ (see \cite{ADM14} and also \cite{Miranda03}). As a convention, when $A=\X$ we simply write ${\rm Per}(E)$ instead of ${\rm Per}(E,\X)$.

\subsection{RCD-spaces}
In this note, we shall work with spaces that encode Ricci lower bounds in a synthetic sense as introduced first and independently in \cite{LV09} and \cite{Sturm06I,Sturm06II}. For $K \in \R, N\in[1,\infty)$, the Curvature Dimension condition $\CD(K,N)$ for a metric measure space is a weak notion of Ricci curvature bounded below by $K$ and dimension bounded above by $N$. We will actually consider here the subclass of spaces satisfying the so-called Riemannian Curvature Dimension condition. The $\RCD$-condition has been defined first in the infinite dimensional setting \cite{AmbrosioGigliSavare11-2} and later in \cite{Gigli12} in finite dimension. We also recall \cite{BacherSturm10,AmbrosioGigliSavare12,AGMR15,AmbrosioMondinoSavare13-2,EKS15,CM16} for key contributions on this theory and for the study of the equivalence of different definitions and approaches.  We refer to \cite{AmbICM} for more details and references. 

\begin{definition}
    A metric measure space $\Xdm$ satisfies the $\RCD(K,N)$ condition for some $K \in \R$, $N \in(1,\infty)$, if it is infinitesimally Hilbertian and satisfies the $\CD(K,N)$ condition
\end{definition}
To keep the exposition shorter will not recall the definition of the $\CD(K,N)$ condition and instead focus  on recalling the key properties of $\RCD$ spaces used in this note.

We start recalling that $\RCD(K,N)$ spaces satisfy the Bishop-Gromov inequality \cite{Sturm06I,Sturm06II}:
\begin{equation}
\frac{\mm(B_R(x))}{v_{K,N}(R)}\le \frac{\mm(B_r(x))}{v_{K,N}(r)}, \quad \text{for any $0<r<R\le \pi \sqrt{\frac{N-1}{K^+}}$ and $x \in \X$},\label{eq:Bishop}
\end{equation}
where $K^+$ is the positive part of $K$ and $v_{K,N}(r)$ is the volume of a ball of radius $r$ in the $(K,N)$-model space, see \cite{Sturm06I,Sturm06II} for the precise definition. We only recall the particular case $v_{0,N}(r)=\omega_Nr^N$. In particular  $\RCD(K,N)$ spaces  are uniformly locally doubling and, since they support a weak local Poincar\'e inequality \cite{Rajala12}, by the work \cite{Cheeger00} we have:
\begin{equation}
|\nabla f| = \lip\, f,\qquad\mm\text{-a.e.},\,\,\forall\, f \in \Lip_{bs}(\X).\label{eq:wug is lip}
\end{equation}

Since $\RCD(K,N)$ spaces are geodesic and uniformly locally doubling, they admit a reverse doubling inequality. We omit the standard argument (see e.g.\cite[Prop. 3.3]{GrigoryanHuLau09}). 
\begin{lemma}\label{prop:reverse doubling}
Let $\Xdm$ be an $\RCD(K,N)$ space for some $N\in(1,\infty),K\in\R$. Then there exists $\gamma=\gamma(N)>0$ and $R_{K^-,N}>0$ (with $R_{0,N}=+\infty$) such that for every  ball $B_R(x)\subsetneq \X$ with $R\le R_{K^-,N}$, it holds
\begin{equation}\label{eq:reverse doubling}
    \frac{\mm(B_{r}(x))}{\mm(B_{R}(x))}\le \left(\frac{r}{R}\right)^\gamma, \qquad \forall\, r \in(0,R/2).
\end{equation}
\end{lemma}
We recall also the following version of the \emph{coarea formula} from \cite[Proposition 4.2]{Miranda03} adapted to $\RCD$-setting  after \cite{GigliHan14}.
\begin{theorem}[Coarea formula]
	Let $(\X,\sfd,\mm)$ be an $\RCD(K,N)$ space, $N<+\infty,$ $\Omega \subset \X$ open and $f \in \LIP_{loc}(\Omega)$. Then  given any Borel function $g:\X \to [0,\infty)$, it holds that
	\begin{equation}\label{eq:coarea}
		\int_{\{s< f<t\}} g \,|\nabla f| \d \mm =\int_s^t \int g\, \d \Per(\{f>r\},\cdot )\, \d r, \qquad \forall s,t \in [0,\infty), \ s<t, \,\{f>s\}\subset \subset \Omega.
	\end{equation}
\end{theorem}
\begin{proof}
Fix $s,t$ as in \eqref{eq:coarea}  and $U \subset \subset \Omega$ open and containing $\{f>s\}$. We can suppose that $s>0$. Let $\eta \in \LIP_c(\Omega)$ with $\eta=1$ in $U$, $0\le \eta \le 1$ and set $\tilde f \coloneqq \eta f \in \LIP_c(\X)$. Then by \cite[Remark 4.3]{Miranda03}  and the results in \cite{GigliHan14} about the identification of total variation and minimal weak upper gradient, \eqref{eq:coarea} holds for $s,t$, any $g$ and with $\tilde f$ in place of $f$. To pass to $f$ simply use the locality of the weak upper gradient  and note that by construction $\{\tilde f>r\}=\{f>r\}$ for every $r>s$.
\end{proof}

We also report a regularity result from \cite{Jiang13}.
\begin{theorem}\label{thm:poisson lip}
Let $\Xdm$ be an $\RCD(K,N)$ space for some $K \in \R$, $N \in (1,\infty)$ and let $u\in D(\Delta)$ with $\Delta u= g u$ for some $g \in L^\infty(\mm)$,  $\|g\|_{L^\infty(\mm)}\le M$. Then for every $x_0 \in \X$ and every $R>0$ it holds 
\[
\||\nabla u|\|_{L^\infty(B_R(x_0))}\le C(K,N,R,M)  \fint_{B_{2R}(x_0)} |u| \d \mm.
\]
In particular $u \in \LIP_{loc}(\X).$
\end{theorem}

We say that an $\RCD(0,N)$ space $\Xdm$ has Euclidean volume growth, if
\begin{equation}
    {\sf AVR}(\X) \coloneqq \lim_{R\to\infty} \frac{\mm(B_R(x))}{\omega_NR^N}>0,
\end{equation}
for one (and thus, any) $x \in \X$. In this setting,  a sharp isoperimetric inequality was proved in \cite{BaloghKristaly21} (previous versions in the smooth-setting already appeared in \cite{Brendle20,AgostinianiFogagnoloMazzieri20,FogagnoloMazzieri22, Johne}). A slightly weaker inequality  holds also in the ${\rm MCP}$ setting (\cite{CavallettiManini22-isoMCP}).
\begin{theorem}
Let $\Xdm$ be an $\RCD(0,N)$ space with $N\in(1,\infty),{\sf AVR}(\X)>0$. Then
\begin{equation}
     \Per(E) \ge N({\sf AVR}(\X)\omega_N)^{1/N}\mm(E)^{\frac{N-1}N}, \qquad \forall E\subset \X \text{ Borel, $\mm(E)<+\infty$}.\label{eq:isoperimetry AVR}
\end{equation}
\end{theorem}
Here $ \omega_N\coloneqq \pi^{N/2} \Gamma^{-1}\left(N/2+1\right)$, where $\Gamma(\cdot)$ is the Gamma-function. We shall need also the rigidity of \eqref{eq:isoperimetry AVR} in the $\RCD$-setting. This has been proved in \cite{AntonelliPasqualettoPozzettaSemola22} under the  noncollapsed assumption, which was recently removed (with a different argument) in \cite{CavallettiManini22}. 
\begin{theorem}\label{thm:rigidity ISOAVR}
Let $\Xdm$ be an $\RCD(0,N)$ space  with $N\in(1,\infty),{\sf AVR}(\X)>0$. Equality holds in \eqref{eq:isoperimetry AVR} for some $E\subset \X$  Borel with $\mm(E)<+\infty$ if and only if $\X$ is a $N$-Euclidean metric measure cone and $E$ is (up to $\mm$-neglible sets) a metric ball centred at one of the tips of $\X$.
\end{theorem}
Theorem \ref{thm:rigidity ISOAVR} is stated in \cite{CavallettiManini22} with the extra assumption that $E$ is bounded, however this assumption can be dropped thanks to the recent \cite{APPV23}.

Recall that for $N \in [1,\infty)$, the $N$-Euclidean  cone  over a metric measure space $(\Z,\mm_\Z,\sfd_\Z)$ is defined  to be the space $\Z\times [0,\infty)/(\Z\times \{0\})$ endowed with the following distance and measure
\begin{align*}
	&\sfd((t,z),(s,z'))\coloneqq \sqrt{t^2+s^2-2st\cos(\sfd_\Z(z,z')\wedge \pi)},\\
	&\mm\coloneqq t^{N-1} \d t\,\otimes\, \mm_\Z.
\end{align*}
The point $\Z \times \{0\}$ is called tip of the cone.
\subsection{Sobolev inequalities}
We next report the main Sobolev inequalities of this note starting in the compact setting. On an $\RCD(N-1,N)$  space $\Xdm$ for some $N\in (2,\infty)$ with $\mm(\X)=1$,  we recall the following Sobolev inequality (\cite{Profeta15,CavallettiMondino17})
\begin{equation}\label{eq:critical sobolev}
	\|u\|_{L^{2^*}(\mm)}^2\le \frac{2^*-2}{N} \|\nabla  u\|_{L^2(\mm)}^2 +   \|u\|_{L^2(\mm)}^2 , \qquad \forall u \in W^{1,2}(\X),
\end{equation}
where $2^* = 2N/(N-2)$. 

Moving to the non-compact setting, we start recalling a  classical one-dimensional inequality by Bliss \cite{Bliss30} (see also \cite{Aubin82,Talenti76,CianchiFuscoMaggiPratelli09}). To state it we introduce some notations. For all $ N \in (2,\infty)$, we define $\sigma_{N-1}\coloneqq N\omega_{N}$ and recall the sharp Euclidean Sobolev constant
\begin{equation}
\eucl(N,2)\coloneqq \Big(\frac{4}{N(N-2)\sigma_N^{2/N}}\Big)^\frac{1}{2}.\label{eq:eucln2}
\end{equation}
\begin{lemma}[Bliss inequality]\label{lem:bliss}
	Let $u\colon [0,\infty)\to \R$ be locally absolutely continuous, $N \in (2,\infty)$ and define $2^*\coloneqq 2N/(N-2)$. Then
	\begin{equation}\label{eq:bliss}
		\Big(\sigma_{N-1}\int_0^\infty |u|^{2^*}(t)\, t^{N-1}\, \d t\Big)^\frac{1}{2^*}\le {\rm Eucl}(N,2)\Big(\sigma_{N-1}\int_0^\infty |u'|^{2}(t)\, t^{N-1}\, \d t\Big)^\frac{1}{2},
	\end{equation}
	 whenever one side is finite. Moreover, equality holds if and only if $u$ is of the type:
	 \begin{equation}
	     v_{a,b}(r)\coloneqq a(1+b r^2)^\frac{2-N}{2},\qquad a\in\R, b>0.\label{eq:Bliss extremal}
	 \end{equation}
\end{lemma}
We recall the sharp Sobolev Euclidean-type inequality \cite{NobiliViolo21} (first appeared in \cite{BaloghKristaly21} for manifolds).

\begin{theorem}\label{thm:sharpSobgrowth}
	Let $\Xdm$ be an $\RCD(0,N)$ space, $N \in (2,\infty)$, with Euclidean volume growth. Then, for every $u \in W^{1,2}_{loc}(\X)$ with $\mm(\{|u|>t\})<+\infty$ for all $t>0$, it holds
	\begin{equation}\label{eq:sharp sobolev growth}
		\|u\|_{L^{2^*}(\mm)}\le \eucl(N,2){\sf AVR(\X)}^{-\frac{1}{N}}\| \nabla  u\|_{L^2(\mm)}.
	\end{equation}
Moreover, \eqref{eq:sharp sobolev growth} is sharp.
\end{theorem}
\begin{proof}
Combine \cite[Theorem 1.13]{NobiliViolo21} and Lemma \ref{lem:sobolev loc}.
\end{proof}

For convenience in the rest of this note, we adopt the following notation.
\begin{quote}
      {\bf Convention:} We say that an $\RCD(K,N)$ space $\Xdm$, with $N\in(2,\infty)$,  supports a Sobolev inequality with constants $A>0,B\ge 0$, if, setting $2^*\coloneqq 2N/(N-2)$,
    \begin{equation}
         \| u\|_{L^{2^*}(\mm)}^2 \le A \|\nabla u\|_{L^2(\mm)}^2 + B\| u\|_{L^2(\mm)}^2, \qquad\forall u \in W^{1,2}(\X).\label{eq:convention}\tag{{\bf S}}
    \end{equation} 
\end{quote}  
Inequality \eqref{eq:convention}, if true,  actually holds for all $u \in W^{1,2}_{loc}(\X)$ satisfying $\mm(\{|u|>t\})<+\infty$ for all $t>0$  (recall Lemma \ref{lem:sobolev loc}).

\subsection{Convergence and stability under pmGH-convergence}\label{sec:conv}
We start recalling the notion of \emph{pointed-measure Gromov Hausdorff convergence} (pmGH convergence for short) following \cite{GigliMondinoSavare13}. This presentation is not standard (see e.g. \cite{BBI01,Gromov07}), but it is equivalent in the case of a sequence of uniformly locally doubling metric measure spaces (\cite{GigliMondinoSavare13}).

Set $\bar \N \coloneqq  \N\cup \{\infty\}$ and consider a sequence of pointed metric measure spaces $(\X_n,\sfd_n,\mm_n,x_n)$, with $x_n \in \X_n$. We say that $\X_n$ pmGH-converge to $\X_\infty$ if 
	there exist isometric embeddings $\iota_n: \X_n \to (\Z,\sfd)$, $n \in \bar \N$,  into a common metric space $(\Z,\sfd)$ such that
	\[
	({\iota_n})_\sharp\mm_n \rightharpoonup 	({\iota_\infty})_\sharp\mm_\infty \text{ in duality with $C_{bs}(\Z)$ and } \iota_n(x_n)\to \iota_\infty(x_\infty) \text{ in }\Z.\]
In the case of a  sequence of uniformly locally doubling spaces (as in the case of $\RCD(K,N)$-spaces for fixed $K\in R,N<\infty$) we can also take $(\Z,\sfd)$ to be proper.

It will be also convenient to adopt the so-called \emph{extrinsic approach} and identify $\X_n$ with their isomorphic copies in $(\Z,\sfd)$.  This allows writing $\mm_n\weakto \mm_\infty$ in duality with $C_{bs}(\Z)$.  A choice of space $(\Z,\sfd)$ together with isomorphic copies of the spaces $\X_n$ will be often called  a \emph{realization of the convergence}. 

For the scope of this note, it is important to recall the notion of convergence of functions along pmGH-convergence \cite{Honda15,GigliMondinoSavare13,AmbrosioHonda17} and their  properties. We fix in what follows a pmGH-convergent sequence of pointed metric measure spaces as discussed above.
\begin{definition}\label{def:lpconv}
 Let $p \in (1,\infty)$ and fix a realization  of the convergence in $(\Z,\sfd)$. We say:
\begin{itemize}
\item[{\rm i)}] $f_n\in L^p(\mm_n)$ \emph{converges $L^p$-weak} to $f_\infty\in L^p(\mm_\infty)$, provided $\sup_{n \in \N}\|f_n\|_{L^p(\mm_n)}<\infty$ and $f_n\mm_n \weakto f_\infty\mm_\infty$ in $C_{bs}(\Z)$;
\item[{\rm ii)}] $f_n\in L^p(\mm_n)$ \emph{converges $L^p$-strong} to $f_\infty\in L^p(\mm_\infty)$, provided it converges $L^p$-weak and $\limsup_n \|f_n\|_{L^p(\mm_n)} \le  \|f_\infty\|_{L^p(\mm_\infty)}$;
\item[{\rm iii)}] $f_n \in W^{1,2}(\X_n)$ \emph{converges $W^{1,2}$-weak} to $f_\infty \in W^{1,2}(\X)$ provided it converges $L^2$-weak and $\sup_{n \in \N} \| \nabla f_n \|_{L^2(\mm_n)}<\infty$;
    \item[{\rm iv)}] $f_n \in W^{1,2}(\X_n)$ \emph{converges $W^{1,2}$-strong} to $f_\infty \in W^{1,2}(\X)$ provided it converges $L^2$-strong and $\| \nabla f_n \|_{L^2(\mm_n)} \to \|\nabla f_\infty \|_{L^2(\mm_\infty)}$;
\item[{\rm v)}] $f_n\in L^p(\mm_n)$ \emph{converges $L^p_{loc}$-strong} to $f_\infty\in L^p(\mm_\infty)$, provided $\eta f_n$ converges $L^p$-strong to $\eta f_\infty$ for every $\eta \in C_{bs}(\Z)$.
\end{itemize}
\end{definition}
 Recall from  \cite{Honda15,GigliMondinoSavare13,AmbrosioHonda17} the linearity of convergence: if $f_n,g_n$ converge $L^p$-strong to $f_\infty,g_\infty$, respectively, then
\begin{equation}
    f_n + g_n \text{ converges } L^p\text{-strong to } f_\infty+g_\infty.\label{eq:linearity Lpconvergence}
\end{equation}
We point out the following simple fact: for any $p \in (1,\infty)$ it holds
\begin{equation}
     f_n \ L^p\text{-weak converges to } f_\infty\qquad \Rightarrow\qquad 
     \|f_\infty\|_{L^2(\mm_\infty)}\le \liminf_{n\to\infty} \|f_n\|_{L^2(\mm_n)}.
     \label{eq:L2norm is Lp-lsc}
\end{equation}
Indeed, if the above liminf above is $+ \infty$, then there is nothing to prove. So let us assume it to be finite and also to be a limit, hence $f_n$ is $L^2$-bounded. Then there exists an $L^2$-weak convergent subsequence (see \cite{GigliMondinoSavare13}) to some $h \in L^2(\mm_\infty)$ and in particular  $\|h\|_{L^2(\mm_\infty)}\le \liminf_n \|f_n\|_{L^2(\mm_n)} $.  By uniqueness of limits we have $h = f_\infty$, which shows \eqref{eq:L2norm is Lp-lsc}. 

After the works in \cite{Sturm06I,Sturm06II,LV09,Gigli10,AmbrosioGigliSavare11-2,GigliMondinoSavare13} and thanks to Gromov's precompactness theorem \cite{Gromov07} we have the following precompactness result.
\begin{theorem}\label{thm:pmGHstableRCD}
    Let $(\X_n,\sfd_n,\mm_n,x_n)$ be a sequence of  pointed  $\RCD(K_n,N_n)$ spaces, $n \in \N$, with $\mm_n(B_1(x_n))\in [v^{-1},v]$, for $v >1$ and $K_n\to K\in \R, N_n\to N\in [1,\infty)$. Then, there exists a subsequence $(\X_{n_k},\sfd_{n_k},\mm_{n_k},x_{n_k})$ pmGH-converging to a pointed   $\RCD(K,N)$ space $(\X_\infty,\sfd_\infty,\mm_\infty,x_\infty)$.
\end{theorem}
We report from \cite{GigliMondinoSavare13} the  Mosco-convergence of the Cheeger energies for pmGH-converging $\RCD$-spaces: if $f_n$ is $L^2$-weak convergent to $f_\infty$, then
\begin{equation}
   \rmCh(f_\infty) \le \liminf_{n\to\infty} \rmCh(f_n). \label{eq:Gamma liminf}
\end{equation}
Moreover, for any $f_\infty \in L^2(\mm_\infty)$, there exists $f_n \in L^2(\mm_n)$ converging $L^2$-strong to $f_\infty$ and
\[  \limsup_{n\to\infty}  \rmCh(f_n) \le \rmCh(f_\infty).
\]
In particular, the above is a limit.

\section{P\'olya-Szeg\H{o} inequality}\label{sec:Polyarigidity}

\subsection{Non-compact case}
 In this part we extend to the non-compact case the P\'olya-Szeg\H{o} inequality of Euclidean-type obtained in \cite{NobiliViolo21}.

We need first to recall basic notations and facts about monotone decreasing rearrangements for functions in a m.m.s.\ $\Xdm$ (for more details we refer to \cite{MondinoSemola20}). Let $\Omega\subseteq \X$ be an open set (possibly unbounded) and $u:\Omega\to[0,+\infty)$ be a Borel function such that $\mm( \{u>t\})<\infty$ for any $t>0$. We define  $\mu:[0,+\infty)\to[0,\infty)$, the distribution function of $u$ as  $\mu(t)\coloneqq \mm(\{u> t\})$. For $u$ and $\mu$ as above, let us consider the generalized inverse $u^{\#}$ of $\mu$:
\begin{equation*}
	u^{\#}(s)\coloneqq 
	\begin{cases*}
		{\rm ess}\sup u & \text{if $s=0$},\\
		\inf\left\lbrace t:\mu(t)<s \right\rbrace &\text{if $s>0$}. 
	\end{cases*}
\end{equation*}
Note that $u^{\#}$ is non-increasing.
In this note, we will perform rearrangements into the \emph{Euclidean model space} $I_N\coloneqq ([0,\infty),|.|,\mm_{N})$, equipped with the standard Euclidean distance and weighted measure $\mm_{N}\coloneqq  \sigma_{N-1} t^{N-1}\mathcal L^1$, for $N \in(1,\infty)$. For any open set $\Omega \subset \X$ we set  $\Omega^*\coloneqq  [0,r]$ with $\mm_{N}([0,r])= \mm(\Omega)$ (i.e.\ $r^N=\omega_N^{-1}\mm(\Omega)$), with the convention $\Omega^*=[0,\infty)$ if $\mm(\Omega)=+\infty.$ The Euclidean monotone rearrangement $u^*_{N} : \Omega^* \rightarrow \R^+$ is then defined by
    \[ u^*_{N}(x) \coloneqq  u^\#(\mm_{N}([0,x]))=u^\#(\omega_Nx^N ), \qquad \forall x \in \Omega^*.\]
   Note that $u^*_{N}$ is always a non-increasing function, since so is $u^{\#}$. To lighten the notation, we shall often drop the subscript and just write $u^*$. We collect basic facts about rearrangements, that can be proved by standard arguments as in the Euclidean case (see, e.g.\ \cite{ciao}):
\begin{align}
    &u \le v \Rightarrow u^*\le v^*, & &\label{eq:order preserving} \\
& (\varphi(u))^*=\varphi(u^*), & & \forall \varphi: [0,\infty) \to [0,\infty) \text{ non-decreasing.}\label{eq:composition rearr} \\
 &    \| u\|_{L^p(\mm)} = \| u^*\|_{L^p(\mm_N)}, & & \forall u \in L^p(\Omega). \label{eq:equaleuclLpnorm} 
\end{align}
\begin{lemma}\label{lem:rearrmonotoneconvergence}
    Let $\Xdm$ be a metric measure space and $N \in (1,\infty)$. Let $(u_n) \colon \X \to \R^+$ be an non-decreasing sequence of Borel functions.  Denote $u \coloneqq  \sup_n u_n$ and suppose that $\mm(\{u>t\})<+\infty$ for every $t>0$. Then, $u_n^*:I_{N} \to \R^+$ (which exists by the assumptions) is a monotone non-decreasing sequence and $\lim_n u_n^*=u^*$ a.e. in $[0,\infty).$
\end{lemma}
\begin{proof}
The fact that $(u_n^*)$ is monotone non-decreasing follows by the order preserving property of the rearrangement \eqref{eq:order preserving}.  Set $g \coloneqq  \sup_n u_n^*=\lim_n u_n^*$ pointwise on $[0,\infty)$. In particular $\{u_n^* > t\}\uparrow \{g>t\}$ and $\{u_n > t\}\uparrow \{u>t\}$ for any $t >0$. Therefore
\[
\mm_N( \{ g>t \} ) = \lim_n \mm_N( \{ u_n^* >t \} ) = \lim_n \mm ( \{ u_n >t \} ) = \mm ( \{ u >t \} )=\mm_N(\{u^*>t\}).
\]
So $g,u^*:[0,\infty) \to [0,+\infty]$ are equimeasurable and non-increasing (indeed $g$ is the supremum of non-increasing functions), therefore they coincide a.e.\ (see e.g.\ the proof \cite[Prop. 1.1.4]{ciao}).
\end{proof}
We will need the following approximation result to pass from the bounded to the unbounded case in the Euclidean P\'olya-Szeg\H{o} inequality. It will be needed also in other parts of this note.
\begin{lemma}\label{lem:el trick}
Let $\Xdm$ be a metric measure space and $u \in W^{1,2}_{loc}(\X)$ such that $\mm(\{|u|>t\})<+\infty$ for all $t>0$ and $|\nabla u|\in L^2(\mm).$ Then there exists a sequence $u_n \in W^{1,2}(\X)$ of functions with bounded support, such that $u_n \to u$ $\mm$-a.e.\ and $|\nabla (u_n- u)|\to 0$ in $L^2(\mm).$

Moreover if $u \ge 0$ (resp.\ $u \in L^p(\mm)$, $p\in[1,\infty)$) we can take $(u_n)$ non-decreasing (resp.\ so that $u_n\to u$ in $L^p(\mm)$).
\end{lemma}
\begin{proof}
We first deal with the case $u\ge 0$ and $u \in L^\infty(\mm)$ with $\mm(\supp (u))<+\infty.$ Fix $x\in \X$ and consider the sequence $(\eta_n) \subset \LIP(\X)$ given by  $\eta_n(.)\coloneqq (2-\frac{\sfd(.,x)}{n})^+\wedge1$. Note that $(\eta_n)$ is non-decreasing with $\LIP(\eta_n)\le n^{-1}$, $\eta_n=1$ in $B_n(x)$ and $\supp(\eta_n)\subset B_{2n}(x)$. Take $u_n \coloneqq u \eta_n \in W^{1,2}(\X)$ with bounded support. Clearly $u_n \uparrow u$ pointwise and if $u \in L^p(\mm)$ also $u_n \to u $ in $L^p(\mm)$ by dominated convergence. Moreover, by locality
\[
\int |\nabla (u-u\eta_n)|^2\d \mm\le 2\int_{B_n^c(x)}|\nabla u|^2+|\nabla (\eta_n u)|^2 \d \mm,
\]
and by the Leibniz rule 
 \begin{equation*}
     \begin{aligned}
       \| \nabla (\eta_n u)\|_{L^2(B_n(x)^c)}&\le 2n^{-1} \|u\|_{L^\infty(\mm)} \mm(\supp(u))^{\frac12} + \|\eta_n |\nabla u|\|_{L^2(B_n^c(x))} \\ 
      &\le  2n^{-1} \|u\|_{L^\infty(\mm)} \mm(\supp(u))^{\frac12}  + \| \nabla u \|_{L^2 (B_n^c(x))}\to 0.
     \end{aligned}
 \end{equation*}
This proves that $|\nabla u-\nabla(u_n\eta_n)|\to 0$ in $L^2(\mm).$

If  $u\ge 0$,  take  $u_k\coloneqq ((u-1/k)^+)\wedge k$, $k \in \N$, which is a non-decreasing sequence of functions. Clearly 
\[
\int |\nabla (u- u_k)|^2\d \mm\le \int_{\{0<u<1/k\}} |\nabla u|^2\d \mm \to 0,
\]
by dominated convergence. Moreover, since $u_k \in L^\infty(\mm)$ and $\mm(\supp(u_k))<+\infty,$ the conclusion in this case follows from  the previous one and a diagonal argument (multiplying by the functions $\eta_n$). Monotonicity of the sequence is preserved because $\eta_n f\le \eta_{\bar n}g$ \  $\mm$-a.e.\ for every $\bar n>n$ and assuming $0\le f\le g$ $\mm$-a.e.. The pointwise $\mm$-a.e.\ convergence  is also kept, since it remains true on every ball, recalling that $\eta_n=1$ in $B_n(x)$. 

Finally for a general $u$ we approximate first $u^+$ and then $u^-$ by functions $u_n$ and $v_n$ respectively as we did in the above steps. Clearly if $u \in L^p(\mm)$ then $u_n-v_n \to u$ in $L^p(\mm)$. Moreover by construction we have that $u_n-v_n=\nchi_{\{u>0\}}u_n-\nchi_{\{u<0\}}v_n$. Therefore 
$|\nabla (u-(u_n-v_n))|=|\nabla (u^+-u_n)|+|\nabla (u^--v_n)|\to 0$ in $L^2(\mm)$. This concludes the proof also in this case.
\end{proof}
We can now prove the P\'olya-Szeg\H{o} inequality in the non compact case.
    \begin{proposition}\label{prop:unbounded euclpolyaszego}
Let $\Xdm$ be an $\RCD(0,N)$ space for some $N \in (1,\infty)$ with ${\sf AVR}(\X)>0$. Let $u \in W^{1,2}_{loc}(\X)$ be non-negative and such that $\mm( \{u>t\})<\infty$ for any $t>0$. Then, 
\begin{equation}\int|\nabla  u|^2\d \mm\ge {\sf AVR}(\X)^{2/N}\int_0^\infty|\nabla u^*|^2\d \mm_N, \label{eq:unbounded euclpolyaszego}
\end{equation}
meaning that, if the left hand side is finite, then $u^* \in W^{1,2}_{loc}(I_N)$ and \eqref{eq:unbounded euclpolyaszego} holds.
\end{proposition}
\begin{proof}
First, if $\|\nabla u\|_{L^2(\mm)}=\infty$, there is nothing to prove. So, suppose $|\nabla u|\in L^2(\mm)$. By Lemma \ref{lem:el trick} there exists a non-decreasing sequence $u_n \in W^{1,2}(\X)$ of functions with bounded support, such $u_n \to u$ $\mm$-a.e.\ and $\|\nabla u_n\|_{L^2(\mm)}\to \|\nabla u\|_{L^2(\mm)}$. Applying the P\'olya-Szeg\H{o} inequality for bounded domains in \cite[Theorem 3.6]{NobiliViolo21}, we have $u_n^* \in  W^{1,2}(I_N)$ and
 \[ \int|\nabla   u_n|^2\d \mm\ge {\sf AVR}(\X)^{2/N}\int|{\nabla  u_n^*}|^2\d \mm_N.\]
Moreover by  Lemma \ref{lem:rearrmonotoneconvergence}  the sequence $u_n^*$ is non-increasing and $\sup_n u_n^* = u^*$ pointwise. The proof is now concluded since we have that $u^* \in W^{1,2}_{loc}(I_N)$ and $\liminf_n \int |\nabla u_n^*|^2 \, \d \mm_N \ge \int |\nabla  u^*|^2\,\d \mm_N$ by semicontinuity (recall \eqref{eq:W12loc lsc}).
 \end{proof}

\subsection{Rigidity}
In this section, we prove the rigidity in the P\'olya-Szeg\H{o} inequality of Proposition \ref{prop:unbounded euclpolyaszego}. The idea is that if equality  in \eqref{eq:unbounded euclpolyaszego} is attained, the superlevel sets are isoperimetric sets, so Theorem \ref{thm:rigidity ISOAVR} implies that the space is a cone.  This line of thoughts follow classical arguments that date back to the work of \cite{PolSzeg51} in Euclidean contexts and \cite{BerMey82} for manifolds with Ricci curvature lower bounds.

Moreover, under additional regularity, the function can also be proven to be radial. A similar rigidity result was proved in \cite{MondinoSemola20} in the compact case for a different  P\'olya-Szeg\H{o} inequality.

\begin{theorem}[Rigidity of the Euclidean P\'olya-Szeg\H{o} inequality]\label{thm:rigidity PolyaAVR}
Let $(\X,\sfd,\mm)$ be an $\RCD(0,N)$ space for some $N \in (1,\infty)$ with ${\sf AVR}(\X)>0$. Suppose equality holds in \eqref{eq:unbounded euclpolyaszego} (with both sides finite) for $u \in \LIP_{loc}(\X)$  non-negative satisfying $u(x)\to 0$  as $\sfd(x,z)\to \infty,$ for $ z \in \X$ 
and with  $ (u^*)'\neq 0$ a.e. in $\{u^*>0\}$. Then, $\X$ is isomorphic to an $N$-Euclidean metric measure cone.

Moreover, if $|\nabla u|\neq 0$ $\mm$-a.e.\ on $\{u>0\}$, then $u$ is radial, i.e.\  
$$u(x) = u^*\circ  \avr(\X)^{\frac1N} \sfd(x,x_0)$$ for a suitable tip $x_0$ of $\X$.
\end{theorem}
\begin{proof}
We divide the proof into different steps.

\noindent{{\color{blue} Step 1.}} We establish an \email{improved} version of \eqref{eq:unbounded euclpolyaszego} for a function $u$ as in the statement. Fix such $u$. By Theorem \ref{thm:sharpSobgrowth} we know that $u \in L^{2^*}(\mm)$.  For every $n \in \N$ set $v_n\coloneqq (u-1/n)^+$ and notice that they are supported in the open set $\Omega_n\coloneqq\{u>1/(2n)\}$, which is bounded. Therefore $v_n \in \LIP_c(\X).$ In particular by the Lipschitz-to-Lipschitz property of the rearrangement in the compact case (see \cite[Prop. 3.4]{NobiliViolo21}) we have $v_n^*\in \LIP_c([0,R_n))$ for suitable $R_n>0$. From \eqref{eq:composition rearr} we also have $v_n^*=(u^*-1/n)^+$, which is non-increasing and $(v_n^*)'\neq 0$ a.e.\ in $\{v_n^*>0\}$. In particular $u^* \in \LIP_{loc}(0,\infty).$

Define the functions $\phi_n,\psi_n,\mu_n : [0,\sup v_n) \to [0,+\infty)$ as
\[
\phi_n(t)\coloneqq \int_{\{v_n>t\}} |\nabla v_n|^2\, \d \mm, \quad \psi_n(t)\coloneqq \int_{\{v_n>t\}} |\nabla v_n|\, \d \mm, \quad \mu_n(t)\coloneqq \mm(\{v_n>t\})
\]
 and  analogously $\phi,\psi,\mu: [0,\sup u) \to [0,+\infty]$ replacing everywhere $v_n$ with $u$. Note that, thanks to the locality of the gradient, $\phi(t)=\phi_n(t-1/n)$ for all $t \in (1/n,\infty)$ and the same holds for $\psi$ and $\mu.$ We claim that
 \begin{itemize}
 \item[a)]$\mu_n$  is absolutely continuous with 
     \begin{equation}\label{eq:mun derivative} 
         -\mu_n'(t) = \frac{\Per(\{v_n^* >t\})}{|(v_n^*)'|\big( (v_n^*)^{-1}(t)\big)}, \qquad \text{ a.e. $t\in (0,\sup v_n)$}.
     \end{equation} 
      If moreover $|\nabla u|\neq 0$ $\mm$-a.e.\ in $\{u>0\}$ then also
      \begin{equation}\label{eq:mun derivative 2} 
         -\mu_n'(t) = \int |\nabla v_n|^{-1} \d \Per(\{v_n>t\})\qquad \text{ a.e. $t\in (0,\sup v_n)$} ;
     \end{equation} 
     \item[b)] $\phi_n,\psi_n$ are  are absolutely continuous with
     \begin{equation}
          \phi'_n(t)=-\int |\nabla v_n|\, \d \Per(\{v_n>t\}),\,\, \psi'_n(t)=-\Per(\{v_n>t\}), \quad \text{for  a.e. $t\in (0,\sup v_n)$}. 
     \end{equation}
 \end{itemize}
Claim \eqref{eq:mun derivative} in a) follows from \cite[Lemma 3.10-3.11]{MondinoSemola20}, since $\mu_n(t) = \mm_N(\{ v_n^*>t\})$ and $\Per(\{v_n^* >t\})$ is concentrated on the point $(v_n^*)^{-1}(t)$.
Claim b) is instead just a direct verification using the coarea formula (see \eqref{eq:coarea}), since $v_n \in \LIP_c(\X)$.
Under the assumption  $|\nabla u|\neq 0$ $\mm$-a.e.\ in $\{u>0\}$, 
by the H\"older inequality (using \eqref{eq:mun derivative 2}) we have 
\begin{equation}\label{eq:holder polya}
    -\phi_n'(t)\ge -\psi_n'(t)^2(-\mu_n'(t))^{-1},
\end{equation}
at a.e.\ $t\in (0,\sup v_n)$ which is a differentiability point for $\mu_n,\psi_n,\phi_n.$ If instead we only know that $(u^*)'\neq0$ a.e.\ in $\{u^*>0\}$, we can still deduce \eqref{eq:holder polya} applying first H\"older inequality and then differentiating (see the argument in \cite[Prop. 3.12]{MondinoSemola20}). Integrating the above inequality, recalling that $\Per(\{v_n^*>t\})= N\omega_N^{\frac1N}\mu_n(t)^{\frac{N-1}{N}},$ we get for every $r,s \in [0,\sup v_n]$ with $s<r$:
\begin{equation}
		\int_{\{s<v_n\le r\}}|{\nabla v_n}|^2\d \mm\ge \int_s^{r}\Big(\frac{\Per (\{v_n>t\})}{N\omega_N^{\frac1N}\mu_n(t)^{\frac{N-1}{N}}} \Big)^2\int  |\nabla  v_n^*| \d \Per(\{v_n^*>t\})  \, \d t. \label{eq:partial improved polya}
	\end{equation}
Hence, the isoperimetric inequality \eqref{eq:isoperimetry AVR} gives directly
\begin{equation}
   	\int_{\{s<v_n\le r\}}|{\nabla v_n}|^2\d \mm\ge {\sf AVR}(\X)^{2/N} \int_{\{s<v_n^*\le r\}} |\nabla v_n^*|^2\,\d\mm_N, \quad \forall 0\le s<r\le \sup v_n,    \label{eq:partial Polya intermediate}
\end{equation}
having also used coarea formula for the function $v_n^*$ since it is $\LIP([0,R_n])$ as recalled before.

 Since $v_n=(u-1/n)^+$ and $v_n^*=(u^*-1/n)^+,$ from the locality of the gradient we can rewrite \eqref{eq:partial Polya intermediate} (after a change of variable) as
\begin{equation}\label{eq:partial improved polya v2}
\int_{\{s+1/n<u\le r+1/n\}}|{\nabla u}|^2\d \mm\ge  {\sf AVR}(\X)^{2/N}  \int_{\{s+1/n<u^*\le r+1/n\}} |\nabla u^*|^2\,\d\mm_N,
\end{equation}
for every $s<r$ with $s,r \in (0,\sup u-1/n].$
Taking the limit as $n \to +\infty$ we obtain 
\begin{equation}
		\int_{\{s<u\le r\}}|{\nabla u}|^2\d \mm\ge  {\sf AVR}(\X)^{2/N}  \int_{\{s<u^*\le r\}} |\nabla u^*|^2\,\d\mm_N,  \quad \forall 0\le s<r\le \sup u. \label{eq:improved polya}
\end{equation}

\noindent{{\color{blue} Step 2.}} We pass to the proof that $\X$ is a cone. We claim that if equality occurs in \eqref{eq:partial Polya intermediate} for some $n\in \N$ and some $r,s \in [0,\sup v_n]$ with  $r<s$, then
\begin{enumerate}[label=\roman*)]
    \item\label{it:pern}$\Per (\{v_n>t\})=N(\omega_N\avr(\X))^{\frac1N}\mu_n(t)^{\frac{N-1}{N}}$, for a.e. $t \in (s,r)$
    \item\label{it:constn} If $|\nabla u|\neq 0$ $\mm$-a.e.\ in $\{ u>0\}$, then $|\nabla v_n|$ is constant $\Per(\{v_n>t\})$-a.e.\ for a.e. $t\in(s,r)$.
\end{enumerate}
Claim \ref{it:pern} follows directly from the way we deduced \eqref{eq:partial Polya intermediate} from \eqref{eq:partial improved polya} using the isoperimetric inequality \eqref{eq:isoperimetry AVR}. Claim \ref{it:constn} instead follows by the equality case in the H\"older inequality \eqref{eq:holder polya}.

We now suppose, as in the hypotheses, that $u$ attains equality in \eqref{eq:unbounded euclpolyaszego}, which means that equality holds in \eqref{eq:improved polya} with $(s,r)=(0,\sup u)$.  We claim that equality must hold in \eqref{eq:improved polya} also for all $s<r$ with $s,r\in (0,\sup u)$. Suppose it fails for some $s<r.$ Then, calling $L(s',r')$ and $R(s',r')$ respectively the left and right hand sides of \eqref{eq:improved polya}, we have
\[
L(0,\sup u)=L(0,s)+L(s,r)+L(r,\sup u)>R(0,s)+R(s,r)+R(r,\sup u)\ge R(0,\sup u),
\]
which contradicts the equality for $(0,\sup u).$ This proves the claim. Thus, equality holds in \eqref{eq:partial improved polya v2} for every $s<r$, with $s,r\in(0,\sup u-1/n]$ which is equivalent to equality in \eqref{eq:partial Polya intermediate} for every $s<r$ with $r,s \in [0,\sup v_n]$. Therefore  \ref{it:pern} holds and, provided $|\nabla u|\neq 0$ at $\mm$-a.e.\ point in $\{u>0\}$, also \ref{it:constn} holds for every $s<r$ with $r,s \in [0,\sup v_n]$ and $n \in \N$. Putting these together and by arbitrariness of $n,$ implies that 
\begin{align}
&\Per(\{u>t\}) = N({\sf AVR}(\X)\omega_N)^{1/N}\mm ( \mu(t) )^{\frac{N-1}N},& &\text{a.e.\ }t \in (0,\sup(u)),\label{eq: PolSzego 1}
\end{align}
and, if $|\nabla u|\neq 0$ $\mm$-a.e.\ in $\{ u>0\}$, we get
\begin{equation}
    |\nabla u| \equiv c_t\qquad \Per(\{u>t\})\text{-a.e. for some constant }c_t\ge 0\label{eq:lipu ct}
\end{equation}
for a.e. $t\in(0,\sup u)$. 
Therefore, there exists $t$ with $\mu(t)>0$ so that equality occurs in \eqref{eq: PolSzego 1}, and recalling the rigidity in Theorem \ref{thm:rigidity ISOAVR}, we get that $\X$ is isomorphic to an $N$-Euclidean metric measure cone.

\noindent{{\color{blue} Step 3.}} Here we prove the functional rigidity of $u$, i.e.\ we prove that $u$ is radial under the additional assumption: $|\nabla u|\neq 0$ $\mm$-a.e.\ on $\{u>0\}$. 

We first claim that \eqref{eq: PolSzego 1} actually holds for every $t \in (0,\sup u)$. Let $t \in (0,\sup u)$  and consider a sequence $t_n\downarrow t$ for which \eqref{eq: PolSzego 1}  holds in every $t_n$.  Then, by lower-semicontinuity of the perimeter (see, e.g., \cite[Proposition 3.6]{Miranda03}) and continuity  of $\mu$, we get
\[ \Per(\{u>t\}) \le \liminf_{n\to\infty}\Per(\{u>t_n\}) \overset{\eqref{eq: PolSzego 1}}{=}  N({\sf AVR}(\X)\omega_N)^{1/N} \mu(t) ^{\frac{N-1}N}.\]
Being the converse inequality always true (from  \eqref{eq:isoperimetry AVR}), the claim follows. Since $\{ u>t\}$ are bounded (recall that $u$ tends to zero at infinity), we can apply the rigidity  Theorem \ref{thm:rigidity ISOAVR} to deduce that for every $t \in (0,\sup u)$ there exists a radius $R_t> 0$ and $x_t \in \X$  a tip for $\X$ (recall that $\X$ is a cone from Step 2) so that $\mm(  \{u>t\} \triangle B_{R_t}(x_t) )=0$, where $ \triangle$ denotes the symmetric difference. However $\{u>t\}$ is open. Thus
\begin{equation}\label{eq:level=ball}
     \{u>t\} =B_{R_t}(x_t).
\end{equation}
 We stress that the notation $x_t$ is chosen because  the cone structure may depend \emph{a priori} on the  isoperimetric superlevel set $\{u>t\}$.  From here, the rest of the proof is devoted to show that $x_t$ is in fact independent of $t$ and $u$ is radial. To do so we will follow the lines of 
  the argument used in \cite[Theorem 5.1]{MondinoSemola20}, for the compact case.

Using \eqref{eq:lipu ct} and  \eqref{eq:mun derivative} (recall that $\mu(t)=\mu_n(t-1/n)$) we get
\begin{align*}
      Nc_t^{-1}(\avr(\X)\omega_N)^{\frac1N}\mu(t)^{\frac{N}{N-1}}&=\int |\nabla u|^{-1} \d \Per(\{u>t\})=-\mu'(t)= \frac{N \omega_N^{\frac1N}\mu(t)^{\frac{N}{N-1}}}{|(u^*)'((u^*)^{-1}(t))|},
\end{align*}
for a.e. $t \in (0,\sup u).$
In particular, 
\begin{align}
& |\nabla u| =\avr(\X)^{\frac1N}|(u^*)'((u^*)^{-1}(t))| &  & \Per(\{u>t\})\text{-a.e.\ and a.e.\ } t \in (0,\sup u).\label{eq: PolSzego 2} 
\end{align}
Let $M\coloneqq\|u\|_{L^\infty(\mm)}\in [0,+\infty).$
From the hypotheses $u^*$ is non-negative, strictly decreasing and locally absolutely continuous (in fact locally Lipschitz)  in $\{u^*>0\}=[0,A)$ for some $A\in (0,+\infty]$ (in fact $A=\mm(\{u>0\})$). Hence it admits a strictly decreasing continuous inverse  $ (u^*)^{-1}:(0,M]\to [0,A)$, locally absolutely continuous in $(0,M)$. Since $(u^*)^{-1}(M)=0,$ we can extend it by zero in $[M,\infty)$ and call $H:(0,\infty)\to [0,A)$ this extension. In particular $H\in {\sf AC}_{loc}(0,\infty)$. Observe that $H$  might blow up at zero.
Note also that, since $u^*$ is locally Lipschitz in $(0,A)$, it preserves $\mathcal L^1$-null sets. Hence pre-images of $\mathcal L^1$-null subsets  of $(0,M)$ via $H=(u^*)^{-1}$ are also $\mathcal L^1$-null. Therefore for a.e.\ $t\in (0,A)$ the function $u^*$ is differentiable at $(u^*)^{-1}(t)$, the function $H$ is differentiable at $t$  and
\begin{equation}\label{eq:inverse derivative}
   (u^*)'((u^*)^{-1}(t))H'(t)=(u^*((u^*)^{-1}(t)))'=1.
\end{equation}
To conclude the proof, we need to show that $f\coloneqq  \avr(\X)^{-\frac1N}H \circ u: \{u>0\}\to [0,\infty)$ satisfies
\begin{equation}\label{eq:f=d}
    f(.)=\sfd(x_0,.),
\end{equation}
for some point $x_0 \in \{u>0\}.$ Observe that $f$ is continuous. We start proving that:
\begin{equation}\label{eq:regularity of f}
    f \in \LIP_{loc}(\{u>0\}) \text{ and } |\nabla f|=1\,\, \mm \text{-a.e.. in $\{u>0\}$}.
\end{equation}
To show this we will use the chain rule in Lemma \ref{lem:chain rule} with $u$, $\Omega\coloneqq \{u>0\}$, $\phi\coloneqq H$ and $I\coloneqq(0,\infty)$.  To check the hypotheses we observe that by continuity $u(\Omega')\subset \subset (0,\infty)$ for all $\Omega'\subset \subset \Omega$. Moreover  by \eqref{eq: PolSzego 2} and \eqref{eq:inverse derivative} we have that for a.e.\  $t \in (0,M)$ it holds
\[
|H'(u)||\nabla u| =|H'(t)||(u^*)'((u^*)^{-1}(t))|\avr(\X)^{\frac1N}=\avr(\X)^{\frac1N}, \quad  \Per(\{u>t\})\text{-a.e..} 
\]
Therefore by coarea (recall \eqref{eq:coarea}) and the fact that $\mm(\{|\nabla u|=0\}\cap \Omega)=0$, we easily deduce that $|H'(u)||\nabla u|=\avr(\X)^{\frac1N}$ $\mm$-a.e. in $\Omega.$ 
In particular $|H'(u)||\nabla u| \in L^2_{loc}(\mm)$ and we can apply Lemma \ref{lem:chain rule} to deduce that $f \in W^{1,2}_{loc}(\{u>0\})$ with $|\nabla f|= 1$, $\mm$-a.e.\ in $\{u>0\}$. Moreover from the local Sobolev-to-Lipschitz property (see \cite[Prop. 1.10]{GV21}) we deduce that $f\in \LIP_{loc}(\{u>0\})$ and
\begin{equation}\label{eq:improved sobolev to lip}
    |f(x)-f(y)|\le \sfd(x,y), \quad \forall x,y \in \{u>0\}, \text{ with } \sfd(x,y)\le \sfd(x,\{u=0\}).
\end{equation}
This proves \eqref{eq:regularity of f}. Next, we claim that
\begin{equation}\label{eq:quasidistance}
    \{f<t\}=B_{t}(x_t), \quad \forall t \in (0,A),
\end{equation}
with $x_t \in \{u>0\}.$ We already know  by \eqref{eq:level=ball} and since $H$ is strictly decreasing, that for every $t \in(0,A)$ the set $\{f<t\}$ is a ball $B_{r_t}(x_t)$ for some $r_t\ge 0$ and $x_t$ tip of $\X$.  In particular $\mm(\{f<t\})=\omega_N \theta (r_t)^N$ and $\Per(\{f<t\})=(\omega_N \theta)^{\frac 1N}N\theta  (r_t)^{N-1},$  where $\theta\coloneqq \avr(\X)$.
 Moreover by coarea formula \eqref{eq:coarea} applied to $-f$ and using \eqref{eq:regularity of f}
\[
\omega_N \theta [(r_t)^N-(r_s)^N]=\mm(\{f<t\})-\mm(\{f<s\})=\int_{\{s\le f<t\}}|\nabla f|\, \d \mm=\int_s^t \Per (\{f<r\}) \, \d r.
\]
Therefore the function $(r_t)^N$ is absolutely continuous with 
$$\frac{\d}{\d t}(r_t)^N=(\omega_N \theta)^{-1}\Per (\{f<t\})=N(r_t)^{N-1},\quad \text{a.e. }t\in(0,A),$$
from which follows that $r_t=a+t$, for all $t \in(0,A)$, for some constant $a\ge0$. We claim that $a=0$. Indeed by continuity and Bishop-Gromov inequality we have
$$a^N\omega_N\avr(\X)\le \mm(\cap_{t>0} B_{a+t}(x_t))=\mm(\cap_{t>0} \{f<t\})=\mm(\{f=0\})=\mm(\{u=M\})=0,$$
where in the last equality we used that $|\nabla u|\neq0$ $\mm$-a.e. in $\{u>0\}$. This proves \eqref{eq:quasidistance}.

It remains to prove that $x_t\equiv x_0$ for all $t \in (0,A)$. This would show \eqref{eq:f=d} and conclude the proof. We argue by contradiction and suppose that $x_t \neq x_{\bar t}$ for some $\bar t<t<A.$ Set $\delta\coloneqq \sfd(x_t,x_{\bar t})>0.$ Recall that $x_t$ is a tip of $\X$, hence there is a ray emanating from it and containing $\bar x_t$, i.e.\ an isometry $\gamma: [0,\infty)\to \X$ with $\gamma_0=x_t$ and $\gamma_{\delta}=x_{\bar t}$. Consider the points $x\coloneqq \gamma_t \in \partial B_{t}(x_t)=\{f=t\}$ and $y\coloneqq \gamma_{\delta+\bar t}\in \partial B_{\bar t}(x_{\bar t})=\{f =\bar t\}$. Since $\gamma_{\delta+\bar t}\in B_{t}(x_t)$ and $\gamma$ is an isometry, $\delta+\bar t<t$. Therefore applying \eqref{eq:improved sobolev to lip}, since $\sfd(y,\{u=0\})\ge \sfd(y,\partial B_t(x_t))=\sfd(x,y)$, we finally find a contradiction:
\[
t-\bar t=f(x)-f(y)\le \sfd(x,y)=t-(\bar t+\delta).
\]
\end{proof}
From Step 1 of the above proof, we deduce the following that has its own interest.
\begin{proposition}[Improved P\'olya-Szeg\H{o} inequality]\label{prop:improved polya}
Let $\Xdm$ be an $\RCD(0,N)$ space with $N\in(1,\infty)$ and ${\sf AVR}(\X)>0$. Then for every $u\in \LIP_{loc}(\X)$, non-negative, $u(x)\to 0$ as $\sfd(x,z)\to +\infty$ for some $z \in\X$, and with $(u^*)'\neq 0$ -a.e. in $\{u^*>0\}$,   it holds
\begin{equation}
		\int_{\{s<u<r\}}|\nabla u|^2\d \mm\ge \int_s^{r}\Big(\frac{\Per (\{u>t\})}{N\omega_N^{\frac1N}\mu(t)^{\frac{N-1}{N}}} \Big)^2\int  |\nabla  u_{N}^*| \d \Per(\{u^*_N>t\})  \, \d t \label{eq:improved polya RCD}, \quad \forall 0\le s<r\le \sup u.
	\end{equation}
\end{proposition}
\begin{remark}
\rm
Even if we shall not need it, we observe that  Proposition \ref{prop:unbounded euclpolyaszego}, Proposition \ref{prop:improved polya} and Theorem \ref{thm:rigidity PolyaAVR}  hold replacing $p=2$ with any $p \in (1,\infty)$, the proof is the same.

We point out that the improved rearrangement inequality \eqref{eq:improved polya RCD} appeared also in \cite[Eq. (3.46)]{AntonelliPasqualettoPozzettaSemola22} for non-collapsed spaces and for functions defined on open sets (with finite volume) and with zero-Dirichlet boundary conditions.
\fr
\end{remark}

\begin{remark}[On the necessity of $(u^*)'\neq 0$ and $|\nabla u|\neq 0$]
    \rm 
    We point out that, the hypothesis $(u^*)'\neq 0$ in Theorem \ref{thm:rigidity PolyaAVR} is \emph{necessary} to prove that $u$ is radial. This is well-known, see e.g.\ \cite[Example 4.6]{BrothersZiemer88} for an easy counterexample (in $\R^n$) of a Lipschitz function saturating the P\'olya-Szeg\H{o} inequality with $(u^*)'=0$ occurring on a set of positive measure.

    In Theorem \ref{thm:rigidity PolyaAVR} we also assumed $|\nabla u|\neq 0$ at $\mm$-a.e.\ point of $\{u>0\}$. This was needed to carry out key computations by differentiating the distribution functions (see, e.g., \eqref{eq:mun derivative 2} above), as also done in \cite{MondinoSemola20}. It is not clear to us at the moment if this assumption can be removed.
    \fr 
\end{remark}

\section{Regularity of extremal functions}\label{sec:regularity}

We discuss here the general regularity properties of extremal functions for the Sobolev inequalities \eqref{eq:convention} considered in this note.

\begin{theorem}[Regularity of extremal functions]\label{thm:regularity of extremals}
Fix $N \in (2,\infty)$ and set $2^*\coloneqq 2N/(N-2).$  Let $\Xdm$ be an $\RCD(K,N)$ space, for some $K\in\R,N\in(2,\infty)$ supporting a Sobolev inequality \eqref{eq:convention} with constant $A>0,B\ge 0$.     Suppose  that equality occurs in \eqref{eq:convention}  for some $u \in W^{1,2}_{loc}(\X)$  satisfying  $\|u\|_{L^{2^*}(\mm)}=1$. Then  $u \in D(\bd)$ and
      \begin{equation}\label{eq:el}
            -A\Delta u=(|u|^{2^*-2}u-Bu).
    \end{equation}
    Moreover if $u \in L^\infty(\mm)$, then $u \in \LIP_{loc}(\X)$, $|u|>0$ on $\X$ and if $B=0$ then $|\nabla u|\neq 0$ $\mm$-a.e..
\end{theorem}
For the proof, we need two additional results.
\begin{proposition}[Hopf strong maximum principle]\label{prop:max prin}
Let $\Xdm$ be an $\RCD(K,N)$ space for $K\in\R, N<\infty$. Let $\Omega\subset \X$ be open and connected and $u \in D(\bd,\Omega)\cap C(\Omega)$ satisfying $\bd u-cu\mm \ge 0$ for some constant $c\ge 0$ and $u(x_0)=\sup_{\Omega} u\ge 0$, with $x_0 \in \Omega$. Then $u$ is constant.
\end{proposition}
\begin{proof}

We first prove the following weaker maximum principle: 
\begin{quote}
    let $U\subset \X$ be open and bounded, and suppose that  $v \in D(\bd,U)\cap C(\bar U)$ satisfies $\bd v-cv\mm \ge \delta \mm$
with $\delta>0$, and $m\coloneqq\max_{\bar U} v\ge 0$, then
\begin{equation}\label{eq:weaker maximum principle}
\max_{\bar U} v\le \sup_{\partial U} v.
\end{equation}
\end{quote}
Let $v$ and $U$ be as above. Set $C\coloneqq\{x \in \bar U \ : \ v(x)=m\}.$ If $C\cap \partial U\neq \emptyset$ we are done, hence we can assume that $C\subset U.$ Since $C$ is closed $\emptyset \neq \partial C\subset C\subset U$. Let $z_0 \in \partial C$. By continuity  there exists $r$  small enough so that $B_r(z_0)\subset U$ and $v\ge -\delta/(2c)$ in $B_r(z_0).$ Then $\bd v\ge cv \mm+\delta\mm \ge \delta/2\mm$ in $B_r(x_0)$ and in particular $v$ is subharmonic. Then from the strong maximum principle for subharmonic functions \cite{GR19} (see also \cite{Bjorn-Bjorn11}) (recall that balls in $\X$ are connected) we deduce that $v\equiv m$ in $B_r(z_0),$ which contradicts the fact that $z_0 \in \partial C\subset U.$

We now go back to the proof. The argument is essentially the same in  \cite{GR19}, only that we will use the above weak maximum principle instead of the weak maximum principle for subharmonic functions.

Define the set $C\coloneqq \{u=u(x_0)\}\subset \Omega$. If $C=\Omega$ we are done. Otherwise there exists $x\in \Omega \setminus C$ such that exists a unique $y \in C$ satisfying $r\coloneqq \sfd(x,y)=\sfd(x,C)<\sfd(x,\Omega^c)$ (see \cite{GR19}). Define the function $h(z)\coloneqq e^{-A\sfd(x,z)^2}-e^{-A r^2},$ with $A\gg1$ to be chosen. Let $r'<r/2$ be such that $B_{r'}(y)\subset \Omega.$ To finish the proof it is sufficient to show that 
\begin{equation}\label{eq:sup nuovo}
u(y)=u(y)+\eps h(y)\le \sup_{\partial B_{r'}(y)} u+\eps h, \quad \forall \eps>0,
\end{equation}
indeed the conclusion then follows arguing exactly as at the end of \cite{GR19}.

By Laplacian comparison \cite{Gigli12} (with computations similar to \cite{GR19}) we can show that, provided $A$ is chosen large enough depending on $r$ and $c$, $\bd\restr{B_{r/2}(y)} h\ge 2ce^{-A\sfd(x,\cdot)^2} \mm\restr{B_{r/2}(y)}.$ Therefore
\[
(\bd h-ch\mm)\restr{B_{r/2}(y)} \ge ce^{-A\sfd(x,\cdot )^2}\mm\restr{B_{r/2}(y)}\ge c e^{-4Ar^2} \mm\restr{B_{r/2}(y)}.
\]
In particular for every $\eps>0$
$$(\bd ( u+\eps h)-c(u+\eps h)\mm)\restr{B_{r'}(y)}\ge \eps c e^{-4Ar^2} \mm \restr{B_{r'}(y)}, $$
from which \eqref{eq:sup nuovo} follows from \eqref{eq:weaker maximum principle} with $v\coloneqq u+\eps h$, $U\coloneqq B_{r'}(y),$ noticing that $\sup_{B_{r'}(y)} v\ge u(y)+\eps h(y)=u(x_0)\ge0.$
\end{proof}
\begin{proposition}\label{prop:laplcian locality}
Let $\Xdm$ be an $\RCD(K,N)$ space for some $K\in\R$, $N<+\infty$. Consider $\Omega \subset \X$ open and $u \in D(\DDelta,\Omega)$ with $\Delta u \in L^2_{loc}(\Omega).$ Then
\begin{equation}\label{eq:laplacian locality}
    \Delta u=0 \qquad \text{$\mm$-a.e. in $\{|\nabla u|=0\}$}.
\end{equation}
\end{proposition}
\begin{proof}
We adapt an argument present in \cite{Lou} in the Euclidean setting. 

It is enough to consider $\Omega=\X$ and $\Delta u \in L^2(\mm)$ with $u \in W^{1,2}(\X)$, the general case follows multiplying by Lipschitz cut-off functions with bounded Laplacian (see \cite{Mondino-Naber14}). We have $|\nabla u|\in W^{1,2}(\X)$ (see e.g.\ \cite[Lemma 3.5]{DGP21}) and in particular for every $\eps>0$, $\frac{|\nabla u|}{|\nabla u|+\eps}\in W^{1,2}(\X)$ with
\[
\nabla \left(\frac{|\nabla u|}{|\nabla u|+\eps}\right)= \nabla |\nabla u| \frac{\eps}{(|\nabla u|+\eps)^2}
\]
(see \cite{Gigli14} for the notion of gradient of a Sobolev function).
Fix $\phi \in \LIP(\X)$ with $\supp (\phi) \subset \Omega$. Then integrating by parts
\begin{align*}
    \int \phi \Delta u \frac{|\nabla u|}{|\nabla u|+\eps}\,  \d \mm &= - \int \la \nabla \phi, \nabla u\ra \frac{|\nabla u|}{|\nabla u|+\eps} \, \d \mm  + \int \varphi\la \nabla |\nabla u|, \nabla u\ra \frac{\eps}{(|\nabla u|+\eps)^2}\, \d \mm.
\end{align*}
Since $\left| \frac{\eps|\nabla u|}{(|\nabla u|+\eps)^2} \right|\le 1,$  sending $\eps \to 0^+$ and applying dominated convergence we obtain 
\[
   \int_{\{ |\nabla u| \neq 0\}} \phi \Delta u  \, \d \mm = - \int \la \nabla \phi, \nabla u\ra \, \d \mm  =  \int \Delta u \phi\,  \d \mm.
\]
From the arbitrariness of $\phi$ the conclusion follows.
\end{proof}
\begin{remark}
\rm
Even if not needed here, we observe that  Proposition \ref{prop:laplcian locality} actually holds in the more general setting of $\RCD(K,\infty)$ spaces (with the same proof).
\fr
\end{remark}
We can now prove the regularity result for Sobolev extremals.
\begin{proof}[Proof of Theorem \ref{thm:regularity of extremals}]
The fact that $u \in D(\bd)$ and that \eqref{eq:el} holds follows from a straight-forward computation exploiting the fact that $u$ is a minimizer of
\[
\inf \frac{\|\nabla v\|_{L^2(\mm)}^2+B/A\|v\|_{L^2(\mm)}^2}{\|v\|_{L^{2^*}(\mm)}^2}=\frac{1}{A},
\]
where the infimum is among all $v \in W^{1,2}_{loc}(\X)$ such that $\mm(\{|v|>t\})<+\infty$ for every $t>0$ and taking variations of the form $u+\eps v$, $v \in \LIP_c(\X)$ as $\eps \to 0.$ See e.g. \cite[Prop. 8.3]{NobiliViolo21} for the details in the compact case.

We pass to the second part, assuming that $u$ is in $L^\infty(\mm).$ 
From \eqref{eq:el} we have that $\Delta u\in L^\infty(\mm),$ therefore Theorem \ref{thm:poisson lip} shows that $u \in \LIP_{loc}(\X).$

From now on we will identify $u$ with its continuous representative. We need to show that $|u|>0$ in $\X$. Suppose this is not the case, i.e.\ $|u|(x_0)=0$ for some $x_0 \in \X.$  Note that $|u|$ also satisfies the hypotheses of the theorem, hence $-\Delta |u|=  |u| A^{-1}(|u|^{2^*-2}-B)$. Consider the function $v\coloneqq -|u|\le 0$. Then, since $u \in L^\infty(\mm)$, 
\[
\Delta v-Cv=|u| (A^{-1}|u|^{2^*-2}-A^{-1}B+C)\ge 0,
\]
provided we choose the constant $C>0$ big enough. In particular, $v$ satisfies the assumption of the maximum principle of Proposition \ref{prop:max prin} with $v(x_0)=0=\max v$. Hence $v\equiv 0$ in $\X$, which is a contradiction because $u$ is assumed non-zero. Finally, if $B=0$, since $u$ never vanishes, we have that also $\Delta u$ never vanishes, hence $|\nabla u|\neq 0$ $\mm$-a.e.\ thanks to \eqref{eq:laplacian locality}.
\end{proof}

\section{Rigidity of extremal functions in the Sobolev inequality}\label{sec:rigidity extremal} 

\subsection{Compact case}
We study here the equality case for the Sobolev inequality as in \eqref{eq:critical sobolev}. 

As a technical tool  we will need the following result that is a standard application of the Moser iteration scheme (see e.g.\ \cite[Theorem 4.4]{HLbook}). This is known to be still valid in our setting, relying only on the Sobolev inequality (see also the discussion after \cite[Theorem 5.7]{gigli22}).
\begin{lemma}\label{lem:extra int}
    Let $\Xdm$ be a compact $\RCD(K,N)$ space, $N<+\infty$, and $u\in D(\bd)$ satisfying for some $g \in L^{N/2}(\mm)$
    $${\bf \Delta} u=g u\mm.$$
    Then $u \in L^q(\mm)$ for every $q<+\infty$.
\end{lemma}

 We can now state and prove the main result of this section. Note that the fact that $\X$ is spherical suspension already follows from \cite[Theorem 1.9]{NobiliViolo21}. Here, we are mainly interested in the explicit expression of extremal functions.
\begin{theorem}\label{thm:Aubin extremals}
Let $\Xdm$ be an $\RCD(N-1,N)$ space, $\mm(\X)=1,$ $N\in(2,\infty)$ and set $2^*=2N/(N-2)$. Let $u\in W^{1,2}(\X)$ be non-constant with $\|u\|_{L^{2^*}}=1$ satisfying 
$$\|u\|_{L^{2^*}(\mm)}^2=\frac{2^*-2}{N} \|\nabla u\|_{L^{2}(\mm)}^2+\|u\|_{L^{2}(\mm)}^2.$$ 
Then, $\X$ is isomorphic to a spherical suspension and, for some $a\in \R,b\in(0,1)$ and $z_0 \in \X$:
\[
u= a (1 - b \cos \sfd(\cdot,z_0) )^{\frac{2-N}{2}}.
\]
\end{theorem}
\begin{proof}
The argument is inspired by the computations in \cite[Section 2.1]{DGZ20}.

First, we need to deduce some regularity on the extremal function $u$.
From Theorem \ref{thm:regularity of extremals} we know that $u \in D(\bd)$ and that
\begin{equation}\label{eq:pde u}
\frac{2^*-2}{N} \Delta u =u-|u|^{2^*-2}u.
\end{equation}
 Since $u^{2^*-2} \in L^{N/2}(\mm)$, by Lemma \ref{lem:extra int} below we deduce that $u \in L^q(\mm)$ for all $q<+\infty.$ In particular  $\Delta u \in L^q(\mm)$ for all $q<+\infty.$ Therefore by \cite[Corollary 6]{Kell13} we have $u \in \LIP(\X)$ and so $u \in L^\infty(\mm)$ (alternatively we could have showed $u\in L^\infty(\mm)$ applying \cite[Lemma 4.1]{Profeta15} and then deduced the Lipschitzianity from Theorem \ref{thm:regularity of extremals}). Then we can apply the second part of Theorem \ref{thm:regularity of extremals} to deduce that either $u>0$ or $u<0$ in $\X.$ Note also that $\Delta u \in W^{1,2}(\X).$
 
Without loss of generality, we can assume that $u>0.$ Set $v\coloneqq u^{\frac{-2}{N-2}}$. By the chain rule for the Laplacian (see e.g.\ \cite[Prop. 5.2.3]{GP20}) $v \in D(\bd)$ with 
$$\Delta v=u^{\frac{-2}{N-2}} \left({\frac{-2}{N-2}}u^{-1}\Delta u +\frac{ 2N}{(N-2)^2}u^{-2}|\nabla u|^2\right)\in W^{1,2}(\X)\cap L^\infty(\mm),$$
indeed $|\nabla u|^2 \in W^{1,2}(\X)$ by \cite[Prop. 3.1.3]{Gigli14}. Noting that $|\nabla v|^2=\frac{ 4}{(N-2)^2}|\nabla u|^2 u^{-2} u^{\frac{-4}{N-2}}$, an easy computation using \eqref{eq:pde u} shows
\begin{equation}\label{eq:pde v}
v\Delta v=-\frac{N}{2}(v^2-1)+\frac{N}{2}|\nabla v|^2.
\end{equation}
Since $v$ is bounded above and away from zero, by the chain rule for the Laplacian we also have that $v^{1-N}\in D(\Delta)$ with $\Delta v^{1-N} \in L^\infty(\mm)$. We can then multiply \eqref{eq:pde v} by $\Delta v^{1-N}$ and integrate 
\[
-\frac N2 \int \Delta v^{1-N}v^2\d \mm=\int \Delta v^{1-N} \left(v \Delta v-\frac N2 |\nabla v|^2\right) \d \mm.
\]
We now proceed to integrate by parts. To do this note that $v \Delta v \in W^{1,2}(\X)$ and $|\nabla v|^2 \in D(\bd)$ (see \cite[Prop. 3.1.3]{Gigli14}). Moreover by the Leibniz rule for the divergence  ${\rm div}(\nabla v \Delta v)=\la\nabla v,\nabla \Delta v\ra+ (\Delta v)^2 \in L^1(\mm)$ by the Leibniz rule (see \cite{Gigli14,GP21} for the notion of divergence and e.g.\ \cite[Prop. 3.2]{GV21} for a version of the Leibniz rule that applies here). Hence
\begin{align*}
    N(1-N) \int|\nabla v|^2 v^{1-N}\d \mm&=-\frac N2\int \Delta v^{1-N} |\nabla v|^2 - \int \la \nabla v^{1-N},\nabla v \Delta v+v \nabla \Delta v\ra \d \mm\\
    &=-\frac N2\int v^{1-N} \bd|\nabla v|^2 + \int v^{1-N} (\Delta v)^2+N\la \nabla v,  \nabla \Delta v\ra v^{1-N} \d \mm.
\end{align*}
Combining the above with the dimensional Bochner inequality (\cite{EKS15,Han18}) and with $v>0$, we get 
\[
\frac12\bd |\nabla v|^2-\la \nabla \Delta v,\nabla v\ra\mm=\frac{(\Delta v)^2}{N}\mm+(N-1)|\nabla v|^2\mm.
\]
Integrating and using that $\int\,\d\bd |\nabla v|^2=0$ gives
\[
\int (\Delta v)^2 \d \mm=-\int\la\nabla \Delta v,\nabla v\ra = \int \frac{(\Delta v)^2}{N}\d \mm+(N-1) \int |\nabla v|^2\d \mm,
\]
from which $\int  (\Delta v)^2 \d \mm=N \int |\nabla v|^2\d \mm .$
In particular $\int \tilde v ^2\d \mm =N\int |\nabla \tilde v|^2$, where $\tilde v\coloneqq (v-\int v\d \mm)$. Then by \cite{ket15} we deduce that $\X$ is a spherical suspension and
$$\tilde v(x) =c \cos \sfd(x,z_0)=-c \cos \sfd(x,\bar z_0),\qquad \forall x \in \X,  $$  for some constant $c>0$ and $z_0,\bar z_0\in \X$ tips of the spherical suspension with $\sfd(z,\bar z_0)=\pi$.
Recalling that $v=u^{\frac{2}{2-N}}$ concludes the proof.
\end{proof}

\subsection{Non-compact case}
Here we investigate the equality case in the Euclidean-type Sobolev inequality \eqref{eq:sharp sobolev growth}.
\begin{theorem}\label{thm:rigidity sharp Sob}
Let $(\X,\sfd,\mm)$ be an $\RCD(0,N)$ space with $N \in (2,\infty)$, ${\sf AVR}(\X)>0$ and set $2^*=2N/(N-2)$. Suppose that for some non-zero $u \in W^{1,2}_{loc}(\X)$ with $\mm(\{|u|>t\})<\infty$ for all $t>0$, it holds
\begin{equation}\label{eq:equality in sobolev}
    \| u\|_{L^{2^*}(\mm)} = \eucl(N,2){\sf AVR(\X)}^{-\frac{1}{N}}\| \nabla u\|_{L^2(\mm)}
\end{equation}
(both being finite). Then, $\X$ is isomorphic to a $N$-Euclidean metric measure cone and 
\begin{equation}
     u = a(1+ b \sfd^2(\cdot,z_0))^{\frac{2-N}{2}},\label{eq:rigidity Sob radial}
\end{equation}
for some $a \in \R,b>0$ and  $z_0$ one of the tips of $\X$.
\end{theorem}
\begin{proof}
We will apply Theorem \ref{thm:rigidity PolyaAVR}. First we need to prove the required regularity of $u.$

Notice that we can equivalently suppose that $\|u\|_{L^{2^*}(\mm)}=1$, by scaling invariance.  Moreover also $|u|$ satisfies the equality in \eqref{eq:equality in sobolev}.
By assumptions, it is possible to perform a Euclidean rearrangement $|u|^*$ of $|u|$. By the P\'olya-Szeg\H{o} inequality and the one-dimensional Bliss inequality we get
\[
\| u\|_{L^{2^*}(\mm)} = \eucl(N,2){\sf AVR(\X)}^{-\frac{1}{N}}\| \nabla u \|_{L^2(\mm)}\overset{\eqref{eq:unbounded euclpolyaszego}}{\ge} \eucl(N,2)\|\nabla |u|^*\|_{L^2(\mm_N)} \overset{\eqref{eq:bliss}}{\ge} \| u^*\|_{L^{2^*}(\mm_N)}.
\]
Note that we can apply \eqref{eq:bliss} since by Proposition \ref{prop:unbounded euclpolyaszego}, $u^*\in W^{1,2}_{loc}(I_N)$ and thus $u$ is locally absolutely continuous in $(0,\infty)$ (see e.g.\  \cite[Section 2.2]{NobiliViolo21}).
By \eqref{eq:equaleuclLpnorm} we see that the inequalities in the above are all equalities, and therefore equality holds in the Bliss inequality. Therefore $|u|^*(t) = a(1+bt^2)^{\frac{2-N}{2}}$ for some $a\in\R,b>0$. In particular, since $\|u\|_{L^\infty} = \|u^*\|_{L^\infty} <\infty$ by equimeasurability, we have $u \in W^{1,2}_{loc}(\X)\cap L^\infty(\mm)$ and we can invoke Theorem \ref{thm:regularity of extremals} (with $B=0$) to deduce $u \in \Lip_{loc}(\X)\cap D(\bd)$,   $\mm(\{|\nabla u|=0\})=0$,  $u>0$ or $u<0$,  and (assuming $u>0$):
\[ 
\eucl^2(N,2){\sf AVR(\X)}^{-\frac{2}{N}} \Delta u = -u^{2^*-1}.
\]
Recalling Theorem \ref{thm:poisson lip}, since $u \in L^\infty(\mm)$, we  get that $|\nabla u| \in L^\infty(\mm)$. By the Sobolev-to-Lipschitz property (see \cite{Gigli13,AmbrosioGigliSavare11-2}), $u$ has a Lipschitz representative, still denoted by $u$ in what follows. It remains to show that $u(x)\to 0$ as $\sfd(z,x)\to\infty$, for $z\in\X$. Suppose, by contradiction, that there is a sequence $(x_n)\subset \X$ satisfying  $\sfd(x_n,z)\to \infty$ as $n\uparrow \infty$ and  with the property that $u(x_n)\ge c>0$ for all $n \in \N$. Since $u \in \LIP(\X)$, denoting  $L\coloneqq \Lip(f)$, we see that for any $x \in B_{c/2L}(x_n)$ we have $u(x) \ge u(x_n) - L\sfd(x,x_n) \ge c/2$ and therefore
\[
 \int_{B_{c/(2L)}(x_n)}|u|^{2^*}\,\d\mm \ge (c/2)^{2^*} \mm\big(B_{c/(2L)}(x_n)\big) \ge  \omega_N c^{2^*}{\sf AVR}(\X)(c/(2L))^N >0.
\]
However this contradicts $u \in L^{2^*}(\mm)$.

We deduced all the regularity required to invoke Theorem \ref{thm:rigidity PolyaAVR}, so we know that that $\X$ is an $N$-Euclidean metric measure cone with tip $z_0$ and $u$ is radial, i.e. $u(x)= u^*\circ \, \avr(\X)^{\frac1N}\sfd(x,z_0)$. The conclusion follows since $u^*(t)=|u|^*(t) = a(1+bt^2)^{\frac{2-N}{2}}$ for some $a\in\R,b>0$.
\end{proof}

\section{Compactness of extremizing sequences} \label{sec:gen extremals}
A classical result using concentration compactness is that a sequence extremizing functions for the Sobolev inequality in $\R^n$, up to a rescaling, dilation and translation, converges up to a subsequence to an extremal function. In this part, we generalize this method to an extremizing sequence of functions defined on a sequence of $\RCD(0,N)$ spaces (Theorem \ref{thm:CC_Sobextremals}).
\subsection{Density upper bound} 
 We first address a technical density bound that will be needed in the proof of Theorem \ref{thm:CC_Sobextremals} to get pre-compactness in the pmGH-topology. This part is needed only for collapsed $\RCD$-spaces: a reader interested in the case of smooth manifolds can skip this subsection.
\begin{lemma}[Density bound from reverse Sobolev]\label{lem:density bound}
For every $N\in(2,\infty)$, $K \in \R$, there are constants $\lambda_{N,{K}}\in(0,1),\, r_{{K^-},N}>0$  (with $r_{0,N}=+\infty$), $C_{N,K}>0$  such that the following holds. Let $\Xdm$ be an $\RCD(K,N)$ space and $u \in W^{1,2}_{loc}(\X)\cap L^{2^*}(\mm)$, non-constant satisfying 
    \begin{equation}\label{eq:reverse sobolev lemma}
           \|u\|_{L^{2^*}(\mm)}^2\ge A\|\nabla u\|_{L^2(\mm)}^2,
    \end{equation}
    for some $A>0.$ Assume also that for some $\eta \in(0,\lambda_{N,K})$, $\rho \in (0,r_{{K^-},N}\wedge \frac{\lambda_{N,K}}8 \diam(\X))$ and $x\in\X$ it holds
    $$\|u\|^{2^*}_{L^{2^*}(B_{\rho}(x))}\ge (1-\eta) \|u\|_{L^{2^*}(\mm)}^{2^*}.$$
    Then
        \begin{equation}\label{eq:density bound}
        \frac{\mm(B_\rho(x))}{\rho ^N}\le \frac{C_{N,K}}{A^{N/2}}.
    \end{equation}
\end{lemma}
\begin{proof}
We fix a constant $\lambda=\lambda_{N,K}\in(0,1)$ sufficiently small and to be chosen later. We also fix a constant $r_{K^-,N}>0$, with $r_{0,N}=+\infty$ and with $r_{K^-,N}$ small and to be chosen later in the case $K<0$ ($r_{K^-,N}$ will be chosen after $\lambda_{N,K}$). Assume $\rho\le r_{K^-,N}$ and $\eta\le \lambda_{N,K}$ are as in the hypotheses.

Observe that $B_{4\lambda^{-1}\rho}(x)\subsetneq \X$. Up to choosing $r_{K^-,N}$ small enough (when $K<0$) we can assume that $4\lambda^{-1}\rho\le\tilde  r_{K^-,N}$, where $\tilde r_{K^-,N}>0$ is the one given by Lemma \ref{lem:local sobolev embedding}. Set $r\coloneqq 4\lambda^{-1}\rho\ge 4\rho$ and note that $B_{r}(x)\subsetneq \X$.

Fix a cut-off function $\phi \in \LIP_c(B_{r/2}(x))$ such that $\phi=1$ in $B_{r/4}(x)$, $0\le \phi\le 1$ and $\Lip(\phi)\le 10/r.$  Then from \eqref{eq:local sobolev}, since $r\le \tilde r_{K^-,N}$, we have
\begin{align*}
     &\|u\|_{L^{2^*}(B_{\rho}(x))}\le \|u\phi\|_{L^{2^*}(\mm)}\le \frac{C_{N,K}r}{\mm(B_{r}(x))^{1/N}} \|\nabla u\|_{L^2(\mm)}+\frac{10C_{N,K}}{\mm(B_{r}(x))^{1/N}} \| u\|_{L^2(B_{r}(x))}\\
     & \le \frac{C_{N,K}r}{\mm(B_{r}(x))^{1/N}} \|\nabla  u\|_{L^2(\mm)}+\frac{10C_{N,K}}{\mm(B_{r}(x))^{1/N}} 
     (\| u\|_{L^2(B_{\rho }(x))}+\| u\|_{L^2(B_r(x)\setminus B_{\rho }(x))})\\
      &\le \frac{C_{N,K}r}{\mm(B_{r}(x))^{1/N}} \|\nabla u\|_{L^2(\mm)}+\frac{10C_{N,K} \| u\|_{L^{2^*}(\mm)}}{\mm(B_{r }(x))^{1/N}} 
     (\mm(B_{\rho}(x))^{1/N}+\lambda^{1/2^*}\mm(B_{r}(x))^{1/N})
\end{align*}
Substituting \eqref{eq:reverse sobolev lemma}, applying \eqref{eq:reverse doubling} (up to choosing $r_{K^-,N}$ small enough so that $r\le R_{K^-,N}$), using that $1-\lambda < 1-\eta$ and simplifying $\| u\|_{L^{2^*}(\mm)}$, we reach
\[
(1-\lambda)^{1/2^*}\le \frac{C_{N,K}r}{\sqrt A\mm(B_{r}(x))^{1/N}}+10C_{N,K}((\lambda /4)^\gamma+ \lambda^{1/2^*}),
\]
where $\gamma>0$ is a constant depending only on $N.$
Choosing $\lambda$ small enough with respect to $N$ and $K$ gives  
\begin{equation}\label{eq:true density bound}
    \frac{\mm(B_\rho (x))}{r^N}\le \frac{\mm(B_r (x))}{r^N} \le \frac{C_{N,K}}{A^{N/2}}.
\end{equation}
Recalling that $r=4\lambda^{-1}\rho$ proves \eqref{eq:density bound}.
\end{proof}

\subsection{Concentration compactness for Sobolev extremals}\label{sec:CC}
In the following theorem we show that a sequence of extremizing functions defined on a sequence of $\RCD(0,N)$ spaces, after a suitable rescaling of both the function and the space, admits a subsequence converging to a limit extremal function on some limit $\RCD(0,N)$ space.  The idea is similar to the classical Lions' concentration-compactness principle (\cite{Lions84,Lions85}). The first step is a characterization of  the failure of compactness in the critical Sobolev embedding  by specific concentration and splitting of the mass phenomena (see Appendix \ref{sec:cc appendix}). The second step is observing that the extra information that the sequence is extremizing for the Sobolev inequality will prevent these pathological phenomena and ensure compactness. A crucial point will be to exploit the strict concavity property of the Sobolev inequality, and in particular of the function $t\mapsto t^{2/2^*}$, to deduce that splitting the mass is not convenient in an extremizing sequence. 
\begin{theorem}\label{thm:CC_Sobextremals}
   For every $N\in (2,\infty)$,  exists $\eta_N\in(0,1/2)$ such that the following holds. Let $(Y_n,\rho_n,\mu_n,y_n)$ be a sequence of pointed $\RCD(0,N)$   spaces  supporting a Sobolev inequality   \eqref{eq:convention}  with $A_n\to A >0$ and $B_n \to B \in [0,\infty)$  and also satisfying either $\sup_n \mu_n(B_1(y_n))<+\infty$   or $\diam(Y_n)> \eta_N^{-1}$.
%$\mu_n(B_R(y_n))\le C_{N,K} R^N\mu_n(B_1(y_n))$  with $C_{N,K}>0$ uniform in $n\in\N,R\ge \eta^{-1}$ and 
%with $\mu_n(B_1(y_n))\in (v^{-1},v)$ for $v>0$ and let $\eta \in (0,1)$. 

Suppose there exist non-constant functions $u_n \in W^{1,2}(Y_n)$ with $\| u_n\|_{L^{2^*}(\mu_n)} =1$  and
\begin{align}
&\sup_{y \in Y_n} \int_{B_1(y)}|u_n|^{2^*}\,\d\mu_n=\int_{B_1(y_n)}|u_n|^{2^*}\,\d\mu_n = 1-\eta,\label{eq:Levyscalings} \\
 &\| u_n\|^2_{L^{2^*}(\mu_n)} \ge  \tilde A_n \|\nabla u_n\|^2_{L^2(\mu_n)} + B_n\| u_n\|_{L^2(\mu_n)}^2,\label{eq:extremals}
\end{align}
for  $\tilde A_n \to A$,  and  some $\eta \in(0,\eta_N).$ Then, up to a subsequence, it holds:
    \begin{itemize}
    \item[ \rm i)]
    $Y_n$ pmGH-converges to a pointed $\RCD(0,N)$-space $(Y,\rho,\mu,y)$ supporting a Sobolev inequality as in \eqref{eq:convention} with constants $A,B$;
    \item[ \rm ii)]  $u_n$ converges  $L^{2^*}$-strong to some $u \in W^{1,2}_{loc}(Y)$ with $|\nabla u| \in L^2(\mu)$ and
\[
       \int |\nabla u_n|^2\, \d \mu_n \to  \int |\nabla u|^2\,\d\mu, \qquad  \text{as }n\uparrow\infty.
\]
If $B>0$, then the convergence is also $W^{1,2}$-strong.
    \item[ \rm iii)] It holds
    \[ \| u\|^2_{L^{2^*}(\mu)} =  A \|\nabla u\|^2_{L^2(\mu)} + B\| u\|^2_{L^2(\mu)}.\]

\end{itemize}
\end{theorem}
\begin{proof} We subdivide the proof into different steps.

\noindent{\color{blue} {\sc Step 1}}. We take $\eta_N\coloneqq \frac{\lambda_{0,N}}8\wedge  \frac13$, with $\lambda_{0,N}$ as in Lemma \ref{lem:density bound}. In light of Theorem \ref{thm:pmGHstableRCD}, to extract a subsequence converging pmGH it is sufficient to check that $\mu_n(B_1(y_n)) \in (v^{-1},v) $ for some $v>1.$  If $\diam(Y_n)> \eta_N^{-1}\ge 8\lambda_{0,N}^{-1}$,  thanks to the assumptions \eqref{eq:Levyscalings} and \eqref{eq:extremals}, we can apply  Lemma \ref{lem:density bound} to obtain
\[
\limsup_n \mu_n(B_1(y_n))\le \limsup_n \frac{C_N }{ (\tilde A_n)^{N/2}}= \frac{C_N }{ A^{N/2}}<+\infty,
\]
otherwise $\sup_n \mu_n(B_1(y_1))<+\infty$ is directly true by the assumptions.
On the other hand, since by assumption the spaces $Y_n$ satisfy a Sobolev inequality with constants $A_n,B_n$, plugging in functions $\phi_n\in \LIP(Y_n)$ such that $\phi_n=1$ in $B_1(y_n)$ with $\supp \varphi_n \subset B_2(y_n)$, $0\le \phi_n\le 1$ and $\Lip(\phi_n)\le 1$, we get
\[
\mu_n(B_1(y_n))^{2/2^*}\le (A_n+B_n)\mu_n(B_2(y_n))\le 2^N(A_n+B_n)\mu_n(B_1(y_n)),
\]
where we used the Bishop-Gromov inequality. Since $\lim_n(A_n+B_n)=A+B>0$ we also obtain $\liminf_n \mu_n(B_1(y_n))>0.$ Therefore up to a not relabelled subsequence, the spaces $Y_n$ pmGH converge to a pointed $\RCD(0,N)$ space $(Y,\rho,\mu,y)$. Moreover, the stability of the Sobolev inequalities \cite[Lemma 4.1]{NobiliViolo21} ensures that $Y$ supports a Sobolev inequality as in \eqref{eq:convention} with constants $A,B$.
This settles point i).

\noindent{\color{blue} {\sc Step 2}}. From now on we assume to have fixed a realization of the convergence in a proper metric space $(\Z,\sfd)$ (as in Section \ref{sec:conv}). Let $\nu_n \coloneqq  |u_n|^{2^*}\mu_n \in \PP(\Z)$. Moreover   we will denote by $B_r(z)$, $z \in \Z,$ and by $B_r^n(y)$, $y \in \Y_n,$ respectively  the balls in $(\Z,\sfd)$ and in $(\Y_n,\rho_n)$, recalling that we are identifying $(\Y_n,\rho_n)$ as a subset of $(\Z,\sfd).$
From Lemma \ref{lem:ConcComp1} we have that, up to a subsequence, (exactly) one of cases i),ii),iii) in the statement of Lemma \ref{lem:ConcComp1} holds. We claim i) (i.e.\ compactness) occurs. First, notice that vanishing as in case ii) cannot occur:
\[ \limsup_{n\to\infty} \sup_{y \in Y_n }\nu_n(B_R(y)) \ge \limsup_{n\to\infty} \nu_n(B_1(y_n)) \stackrel{\eqref{eq:Levyscalings}}=1-\eta,\qquad  \forall R\ge 1.
\]
Thus, it remains to exclude the dichotomy case iii). Suppose by contradiction that  iii) of Lemma \ref{lem:ConcComp1} holds for some $\lambda \in (0,1)$ (with $\lambda \ge \limsup_{n}\sup_z \nu_n(B_R(z))$ for all $R>0$), sequences $R_n\uparrow \infty$, $(z_n)\subset \Z$ and measures $\nu_n^1,\nu_n^2$ with  $\supp (\nu_n^1) \subset B_{R_n}(z_n)$ and $\supp (\nu_n^2) \subset \Z\setminus B_{10R_n}(z_n)$. We claim first that $\supp (\nu_n^1) \subset B_{3R_n}(y_n)$ and $\supp (\nu_n^2) \subset \Z\setminus B_{4R_n}(y_n)$. Indeed $\lambda \ge \limsup_n \nu_n(B_1(y_n)) = 1-\eta$ and 
$$\liminf_n \nu_n(B_{R_n}(z_n)) \ge\liminf_n  \nu_n^1(B_{R_n}(z_n))=\lim_n \nu_n^1(\Z) = \lambda\ge 1-\eta.$$  Since $\nu_n(B_1(y_n))=1-\eta$ and $\eta<1/2$, this implies that for $n$ large enough $B_{R_n}(z_n) \cap B_1(y_n) \neq 0$, which implies the claim, provided $R_n\ge 1.$

 Let $\varphi_n$ be a Lipschitz cut-off so that $0\le  \varphi_n \le 1, \varphi_n \equiv 1$ on $B^n_{3R_n}(y_n)$, $\supp (\varphi_n)\subset B^n_{4R_n}(y_n)$  and $\Lip(\varphi_n)\le R_n^{-1}$, for every $n \in \N$. Since 
\begin{equation}
1\ge |\varphi_n|^2 + |(1-\varphi_n)|^2,\qquad\text{in $\Z$},\label{eq:phi one minus phi}
\end{equation} 
we can estimate by triangular inequality, the Leibniz rule and  Young inequality
\begin{equation}
\begin{split}
      \|&\nabla u_n\|^2_{L^2(\mu_n)}  \ge \| \varphi_n|\nabla u_n| \|^2_{L^2(\mu_n)} +  \| (1-\varphi_n) |\nabla u_n|  \|^2_{L^2(\mu_n)} \\
       &\ge \| \nabla (u_n\varphi_n)\|_{L^2(\mu_n)}^2+ \|\nabla (u_n(1-\varphi_n)) \|_{L^2(\mu_n)}^2 -\underbrace{2(1+\delta^{-1}) \| u_n|\nabla  \varphi_n|\|_{L^2(\mu_n)}^2 -2\delta \|\nabla u_n\|^2_{L^2(\mu_n)}}_{\coloneqq R_n(\delta)}
\end{split}
  \label{eq:estim 1}
\end{equation}
for every $\delta>0$ and every $n.$ Setting $O_n\coloneqq  B^n_{4R_n}(y_n)\setminus B^n_{3R_n}(y_n)$, we have by the H\"older inequality
\[ \| u_n |\nabla \varphi_n| \|^2_{L^2(\mu_n)} \le {R_n}^{-2}\| u_n\|^2_{L^{2^*}(O_n)} \mu_n(O_n)^{2/N} \le   16 v^{2/N}\| u_n\|^2_{L^{2^*}(O_n)},   \]
having used that $\mu_n(O_n) \le \mu_n(B^n_{4R_n}(y_n))\le (4R_n)^N\mu_n(B_1(y_n))\le    (4R_n)^Nv$, by the Bishop-Gromov inequality. Notice that we also have
\[ \limsup_{n\to\infty} \| u_n\|_{L^{2^*}(O_n)} \le  \limsup_{n\to\infty}\Big| 1 -  \nu^1_n(\Z)  - \nu^2_n(\Z)  \Big|^{1/2^*}  =0,\]
from which we get $ \lim_n \| u_n |\nabla \varphi_n| \|^2_{L^2(\mu_n)} =0$. Therefore, recalling that $\|\nabla u_n\|^2_{L^2(\mu_n)}$ is uniformly bounded by \eqref{eq:extremals}, choosing appropriately $\delta_n\to 0$,  we get
\begin{equation}\label{eq:zero reminder}
    R_n(\delta_n)\to 0.
\end{equation}
Combining \eqref{eq:estim 1} with \eqref{eq:zero reminder}, recalling that $\lim_n A_n = \lim_n \tilde A_n$,   we get
\begin{align*}
       1 &\overset{\eqref{eq:extremals}}{\ge}  \limsup_{n\to \infty}   A_n  \| \nabla (u_n\varphi_n) \|^2_{L^2(\mu_n)}   +  A_n\| \nabla (u_n(1-\varphi_n)) \|^2_{L^2(\mu_n)} + B_n\|u_n\|_{L^2(\mu_n)}^2   \\
      &\overset{\eqref{eq:convention}}{\ge}  \limsup_{n\to \infty}    \| u_n\varphi_n\|^2_{L^{2^*}(\mu_n)} +  \| u_n(1-\varphi_n)\|^2_{L^{2^*}(\mu_n)}  \\
      &\qquad +B_n\Big( \|u_n\|_{L^2(\mu_n)}^2-  \| u_n\varphi_n\|^2_{L^{2}(\mu_n)} - \| u_n(1-\varphi_n)\|^2_{L^{2}(\mu_n)} \Big)  \\
       &\overset{\eqref{eq:phi one minus phi}}{\ge} \limsup_{n\to \infty} \big( \nu^1_n(\Z)\big)^{2/2^*} +\big( \nu^2_n(\Z)\big)^{2/2^*} \\
      &\ge \lambda ^{2/2^*} +(1-\lambda)^{2/2^*} >1,
\end{align*}
having used the strict concavity of $t \mapsto t^{2/2^*}$ and the fact that $\lambda \in (0,1)$. This gives a contradiction, hence dichotomy in iii) cannot happen.

\noindent{\color{blue} {\sc Step 3}}. In the previous step, we proved that case i) in Lemma \ref{lem:ConcComp1} occurs, i.e.\  there exists $(z_n) \subset \Z$ such that for every $\eps>0$ there exists $R\coloneqq R(\eps)$ so that $\int_{B^n_R(z_n)} |u_n|^{2^*}\,\d\mu_n \ge 1-\eps$ for all $ n \in \N$. As soon as $\eps <1/2$, we have $B^n_R(z_n) \cap B^n_1(y_n) \neq \emptyset$ and
\begin{equation}
    \int_{B^n_{2R+1}(y_n)} |u_n|^{2^*}\,\d\mu_n \ge 1-\eps\qquad \forall n \in \N.\label{eq:usigmaCompactness}
\end{equation}
Moreover $y_n\to y$ in $\Z$, hence the sequence of probabilities $|u_n|^{2^*}\mu_n$ is tight ($\Z$ is proper) and, along a not relabelled subsequence, converges in duality with $C_{b}(\Z)$ to some $\nu\in \PP(Y)$. Additionally, up to a further subsequence we have that $u_n$ is $L^{2^*}$-weak convergent to some $u \in L^{2^*}(\mu)$ (\cite{AmbrosioHonda17}) with $\sup_n \|\nabla u_n\|_{L^2(\mu_n)}<\infty$ and also that $|\nabla u_n|^2\,\d\mu_n \weakto \omega$ in duality with $C_{bs}(\Z)$ for some bounded Borel measure $\omega.$
% If $B>0$, we see from \eqref{eq:extremals} that $u_n$ must be uniformly bounded in $L^2$ and thus $L^2$-weak converging to $u \in L^2(\mu)$. By the $\Gamma$-liminf inequality \eqref{eq:Gamma liminf}, we directly deduce that $u \in W^{1,2}(Y)$. If instead $B=0$, 
Applying Lemma \ref{lem:pmGHW12loc}, up to a further subsequence, we also deduce that $u_n$ converges $L^{2}_{loc}$-strong to some $u\in L^2_{loc}(\mu)$, together with the facts $u \in W^{1,2}_{loc}(Y)$ and $|\nabla u|\in L^2(\mu)$.  Note that if $B>0$ then actually $u \in W^{1,2}(Y)$, by  \eqref{eq:extremals} and the lower semicontinuity of the $L^2$-norm \eqref{eq:L2norm is Lp-lsc}.

We are in position to invoke Lemma \ref{lem:conccomp2} to infer the existence of countably many points $\{x_j\}_{j \in J}\subset Y$ and positive weights $(\nu_j),(\omega_j)\subset \R^+$, so that $\nu =|u|^{2^*}\mu +\sum_{j\in J}\nu_j\delta_{x_j}$ and $\omega \ge |\nabla u|^2\mu +\sum_{j\in J}\omega_j\delta_{x_j}$, with $A\omega_j \ge \nu_j^{2/2^*}$  and in particular $\sum_j \nu_j^{2/2^*}<\infty$.  Moreover up  to passing to a subsquence we can, and will, from now on assume that the limits $\lim_n \| \nabla   u_n \|^2_{L^2(\mu_n)}$ and $\lim_n B_n\|u\|^2_{L^2(\mu_n)}$ exist.
Finally, by the lower semicontinuity of the $L^2$-norm  (see \eqref{eq:L2norm is Lp-lsc}) we have  $B\|u\|^2_{L^2(\mu)} \le \lim_n  B_n\|u\|^2_{L^2(\mu_n)}$, where $B\|u\|^2_{L^2(\mu)}$ is taken to be zero when $B=0$ and $\|u\|^2_{L^2(\mu)}=+\infty.$ Also $\lim_n \| \nabla   u_n \|^2_{L^2(\mu_n)} \ge \omega(\Z)$.  Therefore
    \begin{align*}
        1 =  \lim_{n\to\infty} \int |u_n|^{2^*}\, \d \mu_n  &\ge \lim_{n\to\infty}  \tilde  A_n\| \nabla   u_n \|^2_{L^2(\mu_n)} + \lim_{n\to\infty}  B_n\|  u_n \|^2_{L^2(\mu_n)}  \\
               &\ge  A \omega(\Z) + B \|u\|^2_{L^2(\mu)}\\
         &\ge  A\int |\nabla  u|^2\,\d \mu + \sum_{j\in J} \nu_j^{2/2^*} + B \|u\|^2_{L^2(\mu)}\\
          &\overset{\eqref{eq:convention}}{\ge} \Big(\int |u|^{2^*}\, \d \mu\Big)^{2/2^*} + \sum_{j\in J} \nu_j^{2/2^*} \\
           &\ge \Big( \int |u|^{2^*}\, \d \mu  + \sum_{j\in J} \nu_j \Big)^{2/2^*} = \nu(Y)^{2/2^*} = 1,
    \end{align*}
    having used, in the last inequality, the concavity of the function $t^{2/2^*}$. In particular, all the inequalities must be equalities and, since $t^{2/2^*}$ is strictly concave, we infer that every term in the sum $\int |u|^{2^*}\, \d \mu  + \sum_{j\in J} \nu_j^{2/2^*}$ must vanish except one. By the assumption \eqref{eq:Levyscalings} and $|u|^{2^*}\mm_n \rightharpoonup \nu$ in $C_b(\Z),$ we have $\nu_j\le 1-\eta$ for every $j \in J$. Hence $\nu_j=0$ and $\|u\|_{L^{2^*}(\mu)}=1$.
    This means that $u_n$ converges $L^{2^*}$-strong to $u$. Moreover, retracing the equalities in the above we have that $\lim_n \int |\nabla u_n|^2\, \d \mu_n = \int |\nabla u|^2\,\d\mu$ and, when $B>0$, $\lim_n\int |u_n|^2\,\d\mu_n =\int|u|^2\,\d\mu$. This proves point ii). Finally, equality in the fourth inequality is precisely part iii) of the statement. The proof is now concluded.
\end{proof}
\section{Radial functions: technical results}

In this section, we prove results about convergence and approximation of radial functions.

The first one (Lemma \ref{lem:recovery extremals} below) says that, given a sequence of $\RCD$ spaces converging in the pmGH-sense,  a radial function on the limit space is the limit of the same radial functions along the sequence.

We will need  the following simple fact. We omit the proof, which is an easy consequence of Cavalieri's formula and Bishop-Gromov inequality. In this section, we denote $\sfd_{z}(.) \coloneqq \sfd(z,.)$ the distance function from a point $z$.

\begin{lemma}\label{lem:integral dist}
    Let $\Xdm$ be an $\RCD(0,N)$ space for some $N \in (2,\infty)$. Then for every $\alpha>N$, $z \in \X$ and $r>0$ it holds
    \begin{equation}\label{eq:integral distance}
        \int_{B_{r}(z)^c}\sfd_z(\cdot)^{-\alpha}\d \mm \le  \frac{\mm(B_r(z))}{r^N}  C_{N,\alpha} r^{N-\alpha}.
    \end{equation}
\end{lemma}

\begin{lemma}\label{lem:recovery extremals}
    Let $(Y_n,\rho_n,\mu_n,z_n)$ be a sequence of $\RCD(K,N)$ spaces, for some $K\in\R,N\in(2,\infty)$, that is pmGH-converging to $(Y,\rho,\mu,z_0)$. Let $p\in (1,\infty)$ and $f\in C(\R)$ satisfying $|f(t)|^p\le C|t|^{-\alpha}$, for some $\alpha>0$. Suppose also that 
    \begin{equation}\label{eq:growth assumption}
          \lim_{R\to +\infty} \sup_n \int_{B_R(z_n)^c} \rho_{z_n}^{-\alpha} \d \mu_n =0,
    \end{equation}
    where $\rho_{z_n}(\cdot )\coloneqq \rho_n(\cdot ,z_n).$     Then, $f\circ \rho_{z_n}$ converges $L^p$-strong to $f\circ \rho_{z_0}$. In particular, for any $u_n \in L^p(\mu_n)$ that converges $L^p$-strong to $f\circ \rho_{z_0}$, it holds
    \begin{equation}\label{eq:classic l2 conv}
    \|u_n - f\circ \rho_{z_n}\|_{L^p(\mu_n)}\to 0.
    \end{equation}
    \end{lemma}
\begin{proof}
We only need to prove that $f\circ \rho_{z_n}$ converges $L^p$-strong to $f\circ \rho_{z_0}$, then \eqref{eq:classic l2 conv} follows from the linearity of the $L^p$-convergence \eqref{eq:linearity Lpconvergence}.

The assumptions on $f$ imply that $f$ is uniformly continuous and we denote by $\omega: [0,\infty)\to [0,\infty)$ a global modulus of continuity for $f$. Observe that $f$ is also bounded. In the sequel, we fix $(\Z,\sfd)$ a realization of the convergence and recall that $\sfd \restr {Y_n \times Y_n} = \rho_n $. We can estimate
\begin{align*}
     \int |f\circ  \sfd_{z_0} &-f \circ \sfd_{z_n}|^p\d \mu_n \le \int_{B_R(z_n)} |f\circ  \sfd_{z_0} -f \circ \sfd_{z_n}|^p\d \mu_n+2^p\int_{Z\setminus B_R(z_n)} |f|^p\circ  \sfd_{z_0}+|f|^p \circ \sfd_{z_n}\d \mu_n\\
    &\le \mu_n(B_R(z_n))\omega(\sfd(z_0,z_n))^p+ 2^pC\int_{B_R(z_n)^c} \sfd(z_0,\cdot)^{-\alpha}+ \sfd(z_n,\cdot)^{-\alpha}\d \mu_n\\
    &\le\mu_n(B_R(z_n))\omega(\sfd(z_0,z_n))^p+ 2^p C\cdot 2^\alpha \sup_n \int_{B_R(z_n)^c} \sfd(z_n,.)^{-\alpha} \d \mu_n,
\end{align*}
where in the last step we assume that $n$ is big enough so that $\sfd(z_n,z_0)< R/2,$ which ensures $ \sfd^{-1}(z_0,\cdot) \le 2  \sfd^{-1}(z_n,\cdot)$ in $B_R(z_n)^c.$ Since $\sup_n \mu_n(B_R(z_0))<+\infty$ for every $R>0$, by the pmGH-convergence, we can send first $n\uparrow \infty$ and then $R\uparrow \infty$ to obtain $\|f\circ  \sfd_{z_0}-f \circ \sfd_{z_n}\|_{L^p(\mu_n)}\to0.$  Fix $\phi \in C_{bs}(\Z)$ and $R>0$ so that $\supp(\phi) \subset B_{R}(z_0)$, then
\begin{align*}
   & \left |\int \phi  f\circ  \sfd_{z_0} \d \mu- \int \phi  f\circ  \sfd_{z_n} \d \mu_n \right|\le  \\
   &\le  \|\phi\|_\infty \int_{\supp \phi} |f\circ  \sfd_{z_0}-f \circ \sfd_{z_n}| \d \mu_n+ \left|\int \phi f\circ  \sfd_{z_0}  \d \mu_n-\int \phi f\circ  \sfd_{z_0}  \d \mu \right|\\
   &\le \|\phi\|_\infty \mu_n(B_R(z_0))^{1-1/p}\|f\circ  \sfd_{z_0}-f \circ \sfd_{z_n}\|_{L^p(\mu_n)} + \Big|\int \phi f\circ  \sfd_{z_0}  \d \mu_n-\int \phi f\circ  \sfd_{z_0}  \d \mu \Big|.
\end{align*}
Sending $n\uparrow \infty$ we obtain that $ f\circ  \sfd_{z_n} \d \mu_n\rightharpoonup   f\circ  \sfd_{z_0} \mu$ in duality with $C_{bs}(\Z).$ It remains to prove that $\| f\circ  \rho_{z_n}\|_{L^p(\mu_n)}\to \| f\circ  \rho_{z_0}\|_{L^p(\mu)}$.  Since $\|f\circ  \sfd_{z_0}-f \circ \sfd_{z_n}\|_{L^p(\mu_n)}\to0,$ it is enough to show that  $\| f\circ  \sfd_{z_0}\|_{L^p(\mu_n)}\to \| f\circ  \rho_{z_0}\|_{L^p(\mu)}$. Clearly $\| f\circ  \rho_{z_0}\|_{L^p(\mu)} = \| f\circ  \sfd_{z_0}\|_{L^p(\mu)}\le \liminf_n \| f\circ  \sfd_{z_0}\|_{L^p(\mu_n)},$ hence we only need to show  $\| f\circ \sfd_{z_0}\|_{L^p(\mu)}\ge \limsup_n \| f\circ  \sfd_{z_0}\|_{L^p(\mu_n)}.$ We can assume $n$ is big enough so that $\sfd(z_0,z_n)\le 1$. For every $R\ge 4$ fix a cut-off function $\phi_R \in C_{bs}(\Z)$, $0\le \phi_R\le 1$, such that $\phi_R\equiv 1$ in $B_{R}(z_0)$ and with support in $B_{2R}(z_0).$ Then
\[
\int |\phi_R(|f|^p\circ \sfd_{z_0})-|f|^p\circ \sfd_{z_0}| \d \mu_n\le \int_{B_{R}(z_0)^c} 
|f|^p\circ \sfd_{z_0}\d \mu_n \le 2\cdot 2^{\alpha} C\sup_n \int_{B_{R/2}(z_n)^c} \sfd_{z_n}^{-\alpha}\d \mu_n,
\]
where we have used that $B_{R}(z_0)^c\subset B_{R/2}(z_n)^c$ and $\sfd_{z_0}^{-1}\le 2 \sfd_{z_n}^{-1}$ in $B_{R/2}(z_n)^c.$ This shows that 
\[
\left |\int \phi_R(|f|^p\circ \sfd_{z_0})-|f|^p\circ \sfd_{z_0} \d \mu_n\right|\le \eps_R \to 0,\qquad \text{as }R\uparrow\infty,
\]
where $\eps_R$ is independent of $n.$ Therefore
\[
-\eps_R+\limsup_n \int |f|^p \circ \sfd_{z_0} \d \mu_n \le \limsup_n \int \phi_R(|f|^p \circ \sfd_{z_0}) \d \mu_n =\int  \phi_R(|f|^p \circ \sfd_{z_0}) \d \mu\le \int| f|^p \circ \rho_{z_0} \d \mu.
\]
Sending $R$ to infinity, we conclude the proof.
\end{proof} 
The second result of this section is a technical fact that will play a key role in the proof of our main theorem. It states that a Euclidean bubble which is strongly concentrated around a  point is close to a spherical bubble. 
\begin{lemma}\label{lem:stability when concentrating}
   For every $N\in(2,\infty),$ there are constants $C_N,\alpha=\alpha(N)>0$  such that the following holds. Given $\sigma\ge 1$, set $2^*=2N/(N-2)$ and
\[
f_{\sf eu}(t) \coloneqq \frac{\sigma^{\frac{N-2}2}}{\big(  1+\sigma^2t^2\big)^{\frac{N-2}2}}, \quad  f_{\sf sphere}(t) \coloneqq \frac{\sigma^{\frac{N-2}2}}{\big(  1+2\sigma^2(1-\cos(t )\big)^{\frac{N-2}2}}, \qquad t \in [0,\pi].
\]
Let $\Xdm$ be $\RCD(N-1,N)$, $z \in \X$, $\sfd_z(.)\coloneqq \sfd(z,.)$ and $v\coloneqq \sigma^N\mm(B_{\sigma^{-1}}(z))$. Then 
    \[
   \| (f_{\sf eu} - f_{\sf sphere})(\sfd_z)\|_{L^{2^*}(\mm)}+ \| \nabla (f_{\sf eu} - f_{\sf sphere})(\sfd_z)\|_{L^{2}(\mm)}\le C_N \sigma^{-\alpha}(\sqrt v +1).
    \]
\end{lemma}
\begin{proof}
We fix $\eta\in(0,1)$  to be chosen later.   Denote $B\coloneqq B_{\frac1{\eta\sigma}}(z)$.   In what follows $C_N>0$ is a constant depending only on $N$, its value may vary from line to line without notice and without being relabelled.  By Bishop-Gromov and the assumptions, we get
\begin{equation}
\mm(B) \le v(\eta\sigma)^{-N}.\label{eq:growth zn} 
\end{equation}
We divide the proof into two steps, one for the $L^{2^*}$-norm and one for the $L^2$-norm of the gradient. 

\noindent{\color{blue}\textsc{ Step 1}}. We start estimating
\begin{align*}
    \| (f_{\sf eu} - f_{\sf sphere})&(\sfd_z)\|_{L^{2^*}(\mm)}\le \| (f_{\sf eu}-f_{\sf sphere})(\sfd_z)\|_{L^{2^*}(B)} \\
    &+ \|f_{\sf sphere}(\sfd_z)\|_{L^{2^*}(B^c)}+\| f_{\sf eu}(\sfd_z)\|_{L^{2^*}(B^c)}\eqqcolon \I +\II +\III.
\end{align*}
We analyze each term separately. We start with $\I $. Recall that
\[
| 2(1-\cos(t) )-t^2 |\le ct^4 , \qquad 1-\cos(t) \le c t^2,\qquad\forall t\ge 0,
\]
for some numerical constant $c>0.$
Using $||x|^p-|y|^p|\le C_p|x-y|(|x|^{p-1}+|y|^{p-1})$ with $p=(N-2)/2$ and the above  estimates we have for all $ t \in [0,(\eta\sigma)^{-1})$  the following:
 \[
 \begin{aligned}
   |f_{\sf eu}-f_{\sf sphere}|(t) & \le  \frac{C_N\sigma\big|2(1-\cos(t)) - t^2 \big|  \big( \big | \frac 1\sigma +  2\sigma(1-\cos(t)) \big|^{\frac{N-2}{2}-1}   +\big|     \frac 1\sigma +\sigma t^2\big|^{\frac{N-2}{2}-1} \big)}{\big(\frac 1\sigma + \sigma t^2\big)^{\frac{N-2}{2}}\big(\frac 1\sigma+2\sigma(1-\cos(t))\big)^{\frac{N-2}{2}}} \\
   &\le  C_N\frac{\sigma \frac1{(\eta \sigma)^4} \cdot (\sigma^{-1}+(\eta^2\sigma)^{-1})^{{\frac{N-2}{2}-1}}}{\sigma^{2-N} }\le C_N \eta^{-N} \sigma^{\frac{N-2}{2}-2}.
 \end{aligned}
 \]
 This and \eqref{eq:growth zn} directly implies that
 \[
(\I )^{2^*}  =\int_{B_{\frac{1}{\eta \sigma}}(z_n)} |f_{\sf eu} - f_{\sf sphere}|^{2^*}(\sfd_z)\,\d\mm \le  C_N v \eta^{-N(2^*+1)}\sigma^{-2\cdot 2^*}.
 \]
We pass to $\II $. Note that $|f_{\sf sphere}(t)|^{2^*},|f_{\sf eu}(t)|^{2^*}\le C_N \sigma^{-N}t^{-2N}$,  having used that $1-\cos(t)\ge c t^2$ in $[0,\pi]$ for some numerical constant $c>0.$ Hence applying Lemma \ref{lem:integral dist} and using \eqref{eq:growth zn}
\[ 
(\II )^{2^*}+(\III )^{2^*}  \le  C_N v\sigma^{-N}(\sigma \eta)^{N}\le vC_N \eta^N.
\]

\noindent{\color{blue}\textsc{Step 2}}. From the chain rule for the gradient and the fact that $|\nabla \sfd(z,.)|=1$ $\mm$-a.e., we have
\begin{align*}
    \| \nabla (f_{\sf eu} - f_{\sf sphere})&(\sfd_z)\|_{L^{2}(\mm)}\le \| (f_{\sf eu}'-f_{\sf sphere}')(\sfd_z)\|_{L^{2}(B)} \\
    &+ \|f_{\sf sphere}'(\sfd_z)\|_{L^{2}(B^c)}+\| f_{\sf eu}'(\sfd_z) \|_{L^{2}(B^c)}\eqqcolon \I '+\II '+\III ',
\end{align*}

We start with $\I'_n$.  Reasoning similarly to Step 1, we can estimate for all $ t \in [0,(\eta\sigma)^{-1})$
 \[
 \begin{aligned}
   |f_{\sf eu}'-f_{\sf sphere}'|(t) &= (N-2)\sigma \Big| \frac{t \big(\frac 1\sigma+2\sigma(1-\cos(t))\big)^{\frac N2} -\sin(t)\big(\frac 1\sigma + \sigma t^2\big)^{\frac N2}}{\big(\frac 1\sigma + \sigma t^2\big)^{\frac N2}\big(\frac 1\sigma+2\sigma(1-\cos(t))\big)^{\frac N2}} \Big| \\
   &\le C_N \sigma^{N+2} t\,\big|2(1-\cos(t)) - t^2\big|  \big( \big ( \frac 1\sigma +  2\sigma(1-\cos(t)) \big)^{\frac{N}{2}-1}   +\big(     \frac 1\sigma +\sigma t^2\big)^{\frac{N}{2}-1} \big)\\
   & \quad +  C_N \sigma^{N+1} |\sin(t)-t|\big(     \frac 1\sigma +\sigma t^2\big)^{\frac{N}{2}}\\
   &\le   C_N \sigma^{N+2} t^5  \big(     \frac 1\sigma + \frac 1{\sigma\eta^ 2}\big)^{\frac{N}{2}-1}+ C_N \sigma^{N+1} t^3  \big(     \frac 1\sigma + \frac 1{\sigma\eta^ 2}\big)^{\frac{N}{2}-1}\le  C_N \sigma^{\frac N2-2} \eta^{-N-3}.
 \end{aligned}
 \]
Therefore, again using \eqref{eq:growth zn} we deduce $(\II'_n)^2 \le C v \eta^{-3N-6} \sigma^{-4}$. As above we can directly estimate 
$$|f_{\sf eu}'|,|f_{\sf sphere}'|^2 \le C_N\sigma^{-N+2}t^{2-2N}, \quad t \in [0,\pi],$$
having used $|\sin(t)|\le c t$ and $1-\cos(t)\ge c t^2$ in $[0,\pi]$.
Hence  by Lemma \ref{lem:integral dist} and using \eqref{eq:growth zn}
\[
(\II'_n)^2+(\III'_n)^2 \le C_N v\sigma^{-N+2}(\sigma \eta)^{N-2}\le vC_N \eta^{N-2}.
\]
Combining all cases and taking $\eta\coloneqq \sigma^{-\beta}$ with $\beta>0$ small enough depending on $N$ we conclude, using also that $v^{1/2^*}+v^{1/2}\le 2+2\sqrt v$.
\end{proof}

\section{Proof of the main results}
\subsection{Stability in the compact case}
In this part, we prove the main qualitative stability result of this note. Note that this proves our main Theorem \ref{thm:qualitative SobCompact intro}. We will also provide a proof of Corollary \ref{cor:yamabe} at the end.

Given $N>2$ the family of spherical bubbles in a metric space $(\X,\sfd)$ is denoted by
\[
\mathcal{M}_{\sf sphere}(\X) \coloneqq \{  a(1 - b \cos \sfd(x,z_0) )^{\frac{2-N}{2}} \colon  a\in\R, b\in(0,1),z_0 \in \X \} \cup \{ u\equiv a \colon a \in \R\}.
\]

\begin{theorem}\label{thm:qualitative Optimal Sobolev RCD}
For every $\eps>0$ and $N\in(2,\infty)$ there exists $\delta\coloneqq \delta(\eps,N)>0$ such that the following holds. Let $\Xdm$ be an $\RCD(N-1,N)$ space for some $N\in(2,\infty)$ with $\mm(\X)=1$, set $2^*=2N/(N-2)$ and suppose that there exists $u\in W^{1,2}(\X)$ non-constant satisfying 
\begin{equation}\label{eq:almost extremal comact rcd}
    \frac{ \| u \|_{L^{2^*}(\mm)}^2 - \|u\|^2_{L^2(\mm)} }{\|\nabla u\|^2_{L^2(\mm)} } > \frac{2^*-2}{N} -  \delta.
\end{equation}
Then there exists $w \in \mathcal{M}_{\sf sphere}(\X)$ such that
\begin{equation}\label{eq:close to bubble rcd}
    \frac{\| \nabla(u- w )\|_{L^2(\mm)} + \|u-w\|_{L^{2^*}(\mm)}}{\|u\|_{L^{2^*}(\mm)}} \le \eps.
\end{equation}
Moreover if $w\equiv a \in \R$, then $a\in\R$ can be chosen  so that the reminder
$$R\coloneqq u-a$$ 
satisfies for some $x\in\X$
\begin{equation}
    \| R\cdot \|R\|_{L^2}^{-1}- \sqrt{N+1}\cos(\sfd(\cdot,x))\|_{L^2}\le C_N (\eps^\alpha+\delta)^{\beta},
\end{equation}
for some positive constants $\alpha, \beta, C_N$ depending only on $N$.
\end{theorem}
\begin{proof} By scaling invariance, it is not restrictive to assume $\|u\|_{L^{2^*}(\mm)}=1$. We only need to prove the first part, as the second follows from Proposition \ref{prop:reminder cosine} below.

We argue by contradiction and suppose that there exist $\eps>0$, a sequence of $\RCD(N-1,N)$ spaces $(\X_n,\sfd_n,\mm_n)$ and non-constant functions $u_n \in W^{1,2}(\X_n)$ with $\|u_n\|_{L^{2^*}(\mm_m)}=1$ so that
\begin{equation}
    \|u_n\|^2_{L^{2^*}(\mm_n)} \ge \tilde A_n\| \nabla  u_n \|_{L^2(\mm_n)} + \|u_n\|^2_{L^2(\mm_n)}, \label{eq:reverse Sobolev un}
\end{equation}
with $\tilde A_n\to \frac{2^*-2}{N}$ and satisfying
\begin{equation}
   \inf_{w \in \mathcal{M}_{\sf sphere}(\X_n)} \| \nabla(u_n- w)\|_{L^2(\mm_n)} + \| u_n- w\|_{L^{2^*}(\mm_n)}  >\eps,\qquad \forall  n \in \N,\label{eq:eps distance extremal}
\end{equation}
Let us fix $\eta <(\eta_N\wedge \frac13)$, where $\eta_N$ is as in Theorem \ref{thm:CC_Sobextremals}. For every $n$ there exist  $y_n \in \X_n$ and  $t_n< \diam(\X_n)$   such that
\begin{equation} 
1-\eta =  \int_{B_{t_n}(y_n)} |u_n|^{2^*}\, \d \mm_n = \sup_{y \in \X_n}  \int_{B_{t_n}(y)} |u_n|^{2^*}\, \d \mm_n,\qquad \forall n \in\N.\label{eq:yn maximizing}
\end{equation}
This follows directly by Bishop Gromov inequality and the properness of the space.  Define now $\sigma_n \coloneqq  t_n^{-1}$ and consider the sequence  $(Y_n,\rho_n,\mu_n,y_n) \coloneqq  (\X_{\sigma_n}, \sfd_{\sigma_n},\mm_{\sigma_n},y_n)$, where $\sfd_{\sigma_n}\coloneqq \sigma_n \sfd_n$, $\mm_{\sigma_n}\coloneqq \sigma_n^N \mm_n$ and $u_{\sigma_n}\coloneqq \sigma_n^{-N/2^*}u_n \in W^{1,2}(Y_n)$. In particular, by scaling, it holds that
\begin{equation}
        1-\eta =  \int_{B_1(y_n)} |u_{\sigma_n}|^{2^*}\, \d \mu_n =\sup_{y \in Y_n}  \int_{B_{1}(y)} |u_{\sigma_n}|^{2^*}\, \d \mm_n,\label{eq:normalization un}
\end{equation}
and also
\begin{equation}
      1 = \| u_{\sigma_n}\|^2_{L^{2^*}(\mu_n)}  \ge \tilde A_n \| \nabla  u_{\sigma_n}\|^2_{L^2(\mu_n)} + \sigma_n^{-2} \| u_{\sigma_n} \|^2_{L^2(\mu_n)}, 
\end{equation}
for all $n \in \N$. Moreover, $Y_n$ supports a Sobolev inequality with constants $A_n=(2^*-2)/N, B_n= \sigma_n^{-2}$. Since $t_n \le \diam(\X_n)\le \pi$, up to a subsequence we  have that $\lim_n\sigma_n = \sigma \in [\pi^{-1},+\infty]$ and, consequently, that $B_n \to \sigma^{-2} \in [0,\pi^2]$. Thanks to \cite[Theorem 1.10]{NobiliViolo21} and up to passing to a subsequence, we can assume that $\diam(\X_n)\ge \pi/2$ and in particular that $\diam(Y_n)\ge \sigma_n\pi/2$. Moreover $\mu_n(B_1(y_n))\le \mu_n(Y_n)=\sigma_n^N$. Hence, up to a subsequence and no matter the value of $\sigma $, the hypotheses of Theorem \ref{thm:CC_Sobextremals} are satisfied.
Applying Theorem \ref{thm:CC_Sobextremals} we get that, up to a further subsequence,  $Y_n$ pmGH-converge to a pointed metric measure space $(Y,\rho,\mu,\bar y)$ and that $u_{\sigma_n}$ converges $L^{2^*}$-strong  to some $u \in W^{1,2}(Y)$ satisfying
\[
\| u\|^2_{L^{2^*}(\mu)} =  \frac{2^*-2}{N} \|\nabla u\|^2_{L^2(\mu)} + \sigma^{-2}\| u\|^2_{L^2(\mu)},
\]
(where it is intended that $\sigma^{-2}\| u\|^2_{L^2(\mu)} =0$ if $\sigma = \infty$ and $u \notin L^2(\mu)$). We distinguish now two cases, depending on the value of $\sigma$.

\noindent{{\color{blue} \sc Case 1:}} $\sigma <\infty $. In this case, $B_n \to B\coloneqq \sigma^{-2}>0$ and $Y_n$ are compact of uniformly bounded diameter. Therefore, $Y_n$ mGH-converges to $Y$ and $u_{\sigma_n}$ converges also $W^{1,2}$-strong to $u \in W^{1,2}(Y)$ (recall $ii)$ in Theorem \ref{thm:CC_Sobextremals} when $B>0$). Define $(\X_\infty,\sfd_\infty,\mm_\infty) \coloneqq (Y,\rho/\sigma,\mu/\sigma^N)$ so that $\X_\infty$ is a $\RCD(N-1,N)$ space with $\mm_\infty(\X_\infty)=1$. By $iii)$ in  Theorem \ref{thm:CC_Sobextremals} the function $v \coloneqq \sigma^{N/2^*}u \in W^{1,2}(\X_\infty)$ satisfies
\[
     \| v\|^2_{L^{2^*}(\mm_\infty)} =  \frac{2^*-2}{N} \| \nabla  v \|^2_{L^2(\mm_\infty)} + \| v\|^2_{L^2(\mm_\infty)}.
\]
Here, we distinguish two situations: $v$ is constant, or not.
If $v$ is constant, then $v\equiv1$ and $u \equiv \sigma^{-N/2^*}$. By linearity of convergence \eqref{eq:linearity Lpconvergence}, $u_{\sigma_n}- \sigma^{-N/2^*}$ converges $W^{1,2}$-strong and $L^{2^*}$-strong to zero so, by scaling, we reach
\[ 
\begin{aligned}
 0 &=\lim_{n\to\infty} \| \nabla( u_{\sigma_n}- \sigma^{-N/2^*}) \|_{L^2(\mm_{\sigma_n})}^2 + \| u_{\sigma_n}- \sigma^{-N/2^*} \|^2_{L^{2^*}(\mm_{\sigma_n})} \\
 &= \lim_{n\to\infty} \|\nabla( u_n- (\sigma_n /\sigma)^{N/2^*} )\|^2_{L^2(\mm_n)} + \| u_n- (\sigma_n /\sigma)^{N/2^*} \|^2_{L^{2^*}(\mm_n)}.
 \end{aligned}
\]
This yields a contradiction with \eqref{eq:eps distance extremal}.

If $v$ is not constant, by Theorem \ref{thm:Aubin extremals}, there  exist $a\in \R,b\in(0,1),z_0 \in \X_\infty$ so that
\[
    v(x) = \frac{a}{\big(1 - b \cos (\sfd_\infty(x,z_0)) \big)^{\frac{N-2}{2}}},\qquad\forall x\in\X_\infty.
\]
Denoting $f(t) \coloneqq a \big(1 - b \cos (t) \big)^{\frac{2-N}{2}}$ for $t \in [0,\pi]$, it is clear that $u =  \tilde f\circ \rho(\cdot, z_0))$, where   $\tilde f(s)\coloneqq \sigma^{-N/2^*}f(\sigma^{-1}s)$, $s \in \R$. Take now a sequence $z_n \to z_0$ GH-converging and invoke Lemma \ref{lem:recovery extremals} (here \eqref{eq:growth assumption} is trivially satisfied by equi-boundedness of the diameters) to get that $\tilde f\circ \sfd_{\sigma_n}(\cdot , z_n)$ converges $L^{2^*}$-strong to $u$ and
\begin{equation}
    \lim_n \| u_{\sigma_n} -\tilde f\circ \sfd_{\sigma_n}(\cdot , z_n) \|_{L^{2^*}(\mm_{\sigma_n})}  =0. \label{eq:2* close}
\end{equation}
We want to scale back this information to the original sequence $u_n$. Simple estimates and triangular inequalities give 
\[
\begin{aligned}
     \limsup_{n\to\infty}& \|  u_n - f\circ \sfd_n (\cdot , z_n)\|_{L^{2^*}(\mm_n)} =  \limsup_{n\to\infty}  \| \sigma_n^{-N/2^*} \big( u_n - f\circ  \sfd_n (\cdot , z_n) \big)\|_{L^{2^*}(\mm_{\sigma_n})} \\
     &\le   \limsup_{n\to\infty} \| u_{\sigma_n} - \sigma^{-N/2^*}f\circ \sfd_n (\cdot , z_n) \|_{L^{2^*}(\mm_{\sigma_n})} 
      + |\sigma^{-N/2^*}- \sigma_n^{-N/2^*} | \ \| f\circ  \sfd_n (\cdot , z_n)\|_{L^{2^*}(\mm_{\sigma_n})}  \\
 & \le \limsup_{n\to\infty} \| u_{\sigma_n} -\tilde f \circ\sfd_{\sigma_n}(\cdot , z_n) )\|_{L^{2^*}(\mm_{\sigma_n})}  +C_{f,\sigma}\limsup_{n\to\infty} |\sigma^{-1} - \sigma_n^{-1}| \overset{\eqref{eq:2* close}}{=}0,
\end{aligned}
\]
using that $\tilde f$ is bounded and  that $\sigma_n$ is away from zero.
We pass now to the gradient norm. From the chain rule of weak gradients and the fact that $|\nabla \rho(\cdot,z_0)|=1$ $\mu$-a.e., we have $|\nabla u| = |\tilde f'|\circ \rho(\cdot,z_0)$  $\mu$-a.e.\ and similarly  $|\nabla (\tilde f\circ \sfd_{\sigma_n}(\cdot,z_n))|=|\tilde f'|\circ \sfd_{\sigma_n}(\cdot,z_n)$ at $\mm_{\sigma_n}$-a.e.\ point. In particular again by Lemma \ref{lem:recovery extremals} we have that $|\tilde f'|\circ \sfd_{\sigma_n}(\cdot,z_n)$ converges $L^2$-strong to $|\nabla u|$. This means that the convergence of $\tilde f\circ \sfd_{\sigma_n}(\cdot,z_0)$ to $u$ is $W^{1,2}$-strong. Moreover, as we said above, also $u_{\sigma_n}$ $W^{1,2}$-strong converges to $u$. This together with Lemma \ref{lem:coupling nablas} and \eqref{eq:cosine rule} gives 
\begin{equation}
    \lim_{n\to\infty} \|\nabla \big(u_{\sigma_n}- \tilde f\circ\sfd_{\sigma_n}(\cdot,z_n) \big) \|^2_{L^2(\mm_{\sigma_n})} =0.\label{eq:gradient contradiction}
\end{equation}
Arguing as above for the $2^*$-norm we can scale back the above information to obtain
\[
\lim_{n\to\infty} \| \nabla\big(  u_n - f\circ \sfd_n (\cdot , z_n)\big)\|_{L^{2}(\mm_n)}=0.
\]
 We omit the computation since it is analogous.
Since $ f\circ \sfd_n (\cdot , z_n)\in \mathcal{M}_{{\sf sphere}}(\X_n)$, we again reached a contradiction with \eqref{eq:eps distance extremal}.

\noindent{\color{blue} \textsc{Case 2}: }$\sigma = \infty$. Here $B_n \to B \coloneqq 0$ and we know that $(Y,\rho,\mu,\bar y)$ supports a Sobolev inequality \eqref{eq:convention} with constants $A=(2^*-2)/N,B=0$. In particular, ${\sf AVR}(Y)>0$ thanks to \cite[Theorem 4.6]{NobiliViolo21} and $\sqrt{A}\ge   \eucl(N,2){\sf AVR}(Y)^{-1/N}$ by sharpness in \eqref{eq:sharp sobolev growth}. The sequence $u_{\sigma_n}$ (that we recall is $L^{2^*}$-strong converging to some $u \in L^{2^*}(\mu)$) is so that  $\| \nabla u_{\sigma_n}\|_{L^2( \mm_{\sigma_n})} \to \| \nabla u\|_{L^2(\mu)}$, hence
\[ \eucl(N,2){\sf AVR}(Y)^{-1/N} \|\nabla u\|_{L^2(\mu)}\ge  \| u\|_{L^{2^*}(\mu)} =  \sqrt{A} \|\nabla u\|_{L^2(\mu)}.\]
Therefore $\sqrt A = \eucl(N,2) {\sf AVR}(Y)^{-1/N}$ (recall that $u$ is non-zero) and in particular ${\sf AVR}(Y)$ depends only on $N$. Recalling the rigidity in Theorem \ref{thm:rigidity sharp Sob} we get that $Y$ is isomorphic to an $N$-Euclidean metric measure cone with tip $z_0$ and $u$ is radial of the following form
\[
u(y) = \frac{a}{\big(1+b\rho^2(y,z_0)\big)^{\frac{N-2}2}},\qquad y \in Y,
\]
for some $a \in \R,b>0$. 

Pick now a sequence $z_n\in Y_n$ with $z_n \to z_0$ in $\Z$.  Note that, since $z_n \to z_0$ and $z_0$ is a tip of $Y$, by pmGH convergence we have
\begin{equation*}
 \lim_n  \sigma_n^{N}\mm_{_n}(B_{1/\sigma_n}(z_n))=\lim_n \mm_{\sigma_n}(B_1(z_n))= \mu (B_1(z_0))=\avr(Y)\omega_N.
\end{equation*}
Hence up to a subsequence, since $\avr(Y)$ depends only on $N$, for every $n$ it holds
\begin{equation}\label{eq:density zn}
   \mm_{\sigma_n}(B_1(z_n))= \mm_{n}(B_{\sigma_n^{-1}}(z_n))\sigma_n^{N}\le  C_N.
\end{equation}
Denote 
$$f(t) \coloneqq \frac {a}{(1+bt^2)^{\frac{N-2}{2}}}, \qquad t \ge0.$$  
Note that $|f|^{2^*}, |f'|^2\le Ct^{-2N+2}$ and for every $R\ge 1,$
\begin{equation}
\int_{B_R(z_n)^c}\sfd_{\sigma_n}(\cdot, z_n)^{-2N+2} \, \d \mm_{\sigma_n} \overset{ \eqref{eq:integral distance}}{ \le} C_N \mm_{\sigma_n}(B_1(z_n)) R^{-N+2}\overset{ \eqref{eq:density zn}}{ \le} C_N R^{-N+2}.
\label{eq:growth Talenti}
\end{equation}
Hence assumption \eqref{eq:growth assumption} in Lemma \ref{lem:recovery extremals} is satisfied for $Y_n$ and both $f',f$  and we can apply the result twice to get that $f\circ \sfd_{\sigma_n}(\cdot,z_n)$ converges $L^{2^*}$-strong to $u$, that
\begin{equation}
   \lim_{n\to\infty}\| u_{\sigma_n} - f\circ \sfd_{\sigma_n}(\cdot, z_n)\|_{L^{2^*}(\mm_{\sigma_n})} \overset{\eqref{eq:classic l2 conv}}{=} 0,  \label{eq:un 2* close feu}
\end{equation}
and that $|f'|\circ\sfd_{\sigma_n}(\cdot,z_n) \text{ converges }L^2\text{-strong to }|\nabla u|=|f'|\circ \rho(\cdot,z_0)$. By Lemma \ref{lem:coupling nablas} and the convergence of the gradient norms, we immediately get from the parallelogram identity 
\begin{equation}
   \lim_{n\to\infty}\int | \nabla (u_{\sigma_n} - f\circ \sfd_{\sigma_n}(\cdot,z_n))|^2\,\d\mm_{\sigma_n} = 0. \label{eq:un gradient close feu}
\end{equation}
Scaling all back to $\X_n$ we can rewrite the above convergences as
\begin{equation*}
     \lim_{n\to\infty} \| u_n - f_n\circ \sfd_{n}(\cdot, z_n)\|_{L^{2^*}(\mm_{n})}+\| \nabla (u_n - f_n\circ \sfd_{n}(\cdot, z_n))\|_{L^{2^*}(\mm_{n})}=0,
\end{equation*}
where $$f_n\coloneqq \frac{ ab^{\frac{2-N}{4}} (\sqrt b\sigma_n)^{\frac{N-2}{2}}}{(1+(\sqrt b\sigma_n)^2 t^2)^{\frac{N-2}{2}}}.$$
 Using \eqref{eq:density zn}  we deduce 
$$\mm_{n}(B_{(\sqrt b \sigma_n)^{-1}}(z_n))(\sqrt b \sigma_n)^{N}\le C_N (\sqrt b \vee 1)^N.$$
This is obvious if $b\ge 1$, while for $b\le 1$  it follows by the Bishop-Gromov inequality. Having this density bound, we can  now apply Lemma \ref{lem:stability when concentrating} to get
\begin{equation*}
     \lim_{n\to\infty} \| u_n - g_n\circ \sfd_{n}(\cdot, z_n)\|_{L^{2^*}(\mm_{n})}+\| \nabla (u_n - g_n\circ \sfd_{n}(\cdot, z_n))\|_{L^{2^*}(\mm_{n})}=0,
\end{equation*}
where $$g_n\coloneqq \frac{ ab^{\frac{2-N}{4}}(\sqrt b\sigma_n)^{\frac{N-2}{2}}}{(1+(\sqrt b\sigma_n)^2-(\sqrt b\sigma_n)^2\cos(t))^{\frac{N-2}{2}}}.$$ 
Multiplying and dividing by $1+(\sqrt b\sigma_n)^2$ shows that $g_n\circ \sfd_{n}(\cdot, z_n) \in \mathcal{M}_{{\sf sphere}}(\X_n)$ and gives a contradiction with \eqref{eq:eps distance extremal}. Having examined all the possible cases, the proof is now concluded.
\end{proof}
\begin{remark}\label{rmk:weaker assumptino}
\rm
It is evident from the proof that \eqref{eq:close to bubble rcd} holds true assuming only that 
\[
\| u \|_{L^{2^*}(\mm)}^2 \ge  A\|\nabla u\|^2_{L^2(\mm)}+B\|u\|^2_{L^2(\mm)},
\]
with $|A- \frac{2^*-2}{N}|+|B-1| <  \delta,$ which is a weaker assumption than \eqref{eq:almost extremal comact rcd}.  Indeed,  the starting point of the argument is the reverse Sobolev inequality \eqref{eq:reverse Sobolev un} (for $u_n$) and adding here a sequence $B_n\to 1$ in front of $\|u_n\|^2_{L^2(\mm_n)}$ does not influence the subsequent steps.  
\fr
\end{remark}

\begin{proposition}\label{prop:reminder cosine}
Let $\Xdm$ be an $\RCD(N-1,N)$ space, $N>2$, with $\mm(\X)=1$ and set $2^*=2N/(N-2)$. Let $u \in W^{1,2}(\X)$ be non-constant and set 
 \[   \delta\coloneqq \frac{2^*-2}{N}-\frac{ \|u\|_{L^{2^*}(\mm)}^2 - \|u\|^2_{L^2(\mm)} }{\|\nabla u\|^2_{L^2(\mm)} } \ge 0.\]
Then setting $g\coloneqq {u-\int u}$ we have for some $x \in \X$
\begin{equation}\label{eq:linearization}
   \big\| g\|g\|_{L^2(\mm)}^{-1} - \sqrt{N+1}\cos(\sfd(\cdot,x))\big\|_{L^2(\mm)}\le C_N ((\|\nabla u\|_{L^{2}(\mm)}\|u\|_{L^{2^*}(\mm)}^{-1})^\alpha+\delta)^{\beta},
\end{equation}
for some positive constants $\alpha,\beta$ depending only on $N$.
\end{proposition}
\begin{proof}
We can clearly assume that $\int u=1$. Moreover we can assume that $\|\nabla u\|_{L^{2}(\mm)}\le  \eps_N\|u\|_{L^{2^*}(\mm)} $ for some small constant $\eps_N>0$, otherwise the statement is trivial. Analogously we can assume that $\delta$ is small with respect to $N.$  By the Sobolev and  the Poincar\'e inequalities, provided $\eps_N$ is small enough,  we have $\|u\|_{L^{2^*}}\le 2$.  Set $g\coloneqq u-\int u$. Then by \cite[Lemma 6.7]{NobiliViolo21} and the Poincar\'e inequality we have, provided $\delta$ and $\eps_N$ are small enough,
\[
\left|  N- \frac{\int |\nabla g|^2 \d \mm} {\int g^2\d \mm }\right|\le C_N (\|\nabla u\|_{L^2(\mm)}^\alpha+\delta)\le \tilde C_N ((\|\nabla u\|_{L^2(\mm)}\|u\|_{L^{2^*}(\mm)}^{-1})^\alpha+\delta),
\]
for some $\alpha>0$ depending only on $N.$
Now \eqref{eq:linearization} follows directly from the  quantitative Obata theorem in \cite{CavalettiMondinoSemola19} (there, written for Lipschitz functions but by density in $W^{1,2}$, the statement directly extend to Sobolev functions recalling \eqref{eq:wug is lip}).
\end{proof}
We conclude this part with the proof of the stability result for the Yamabe minimizers in the smooth setting.
\begin{proof}[Proof of Corollary \ref{cor:yamabe}]
Take as in the hypotheses $(M,g)$ so that ${\rm Ric}_g \ge n-1$ and $\sfd_{GH}(M,{\mathbb{S}^n})\le \delta$. Let $u \in W^{1,2}(M)$ non-zero satisfying $|\mathcal{E}(u)-Y(M,g)|\le \delta$.
Set $\nu$ the renormalized volume measure. Since ${\rm Scal}_g\ge n(n-1)$, we have by the Sobolev inequality \eqref{eq:Sob main intro} that
\[
  1\le \frac{\frac{2^*-2}{n}  \|\nabla u\|_{L^{2}(\nu)}^{2}  + \|u\|_{L^{2}(\nu)}^{2} }{\|u\|_{L^{2^*}(\nu)}^{2}}\le  \frac{\mathcal{E}(u)}{n(n-1)\vol_g(M)^{2/n}}\le  \frac{(Y(M,g)+\delta)}{n(n-1)\vol_g(M)^{2/n}},
\]
where the norms are computed using the renormalized volume measure. Recall also that by \cite{Co96} we have that $\vol_g(M)\ge (1-\eps')\vol({\mathbb{S}^n})$, where $\eps'=\eps'(\delta,n)$ goes to zero as $\delta \to 0.$ This in particular gives that $Y(M,g)\ge c(n)>0$ if $\delta$ is chosen small enough (depending on $n$). Therefore, combining the above with the  inequality  (see \cite{Aubin76-2})
$$Y(M,g)\le Y({\mathbb{S}^n})=n(n-1)\vol({\mathbb{S}^n})^{2/n}$$
gives
\begin{align*}
    \frac{\frac{2^*-2}{n}  \|\nabla u\|_{L^{2}(\nu)}^{2}  + \|u\|_{L^{2}(\nu)}^{2} }{\|u\|_{L^{2^*}(\nu)}^{2}}\le  (1+\delta c(n)^{-1})(1-\eps')^{-2/n}.
\end{align*}
The conclusion now follows applying Theorem \ref{thm:qualitative SobCompact intro} (in the stronger version given by \eqref{eq:stronger}).
\end{proof}

\subsection{Stability in the non-compact case}
We now prove the qualitative stability result for the sharp Euclidean-type Sobolev inequality. Note that this proves also Theorem \ref{thm:qualitative SobAVR intro}. Given $N>2$, the family of Euclidean bubbles in a metric space $(\X,\sfd)$ is denoted by
\[
\mathcal{M}_{\sf eu}(\X) \coloneqq \{  a(1+ b \sfd^2(x,z_0))^{\frac{2-N}{2}} \colon  a\in\R, b>0,z_0 \in \X \}.
\]
\begin{theorem}\label{thm:qualitative SobAVR RCD}
For every $\eps>0,V\in(0,1)$ and $N\in(2,\infty),$ there exists $\delta\coloneqq \delta(\eps,N,V)>0$ such that the following holds.
Let  $\Xdm$  be an $\RCD(0,N)$ space with ${\sf AVR}(\X) \in( V,V^{-1})$ and, setting $2^*=2N/(N-2)$, assume there exists $u\in W^{1,2}_{loc}(\X)\cap L^{2^*}(\mm)$ non constant with $\mm(|u|>t)<\infty$ for every $t>0$  satisfying
\[
  \frac{\| u\|_{L^{2^*}(\mm)}}{\| \nabla u\|_{L^2(\mm)}} > {\sf AVR}(\X)^{-\frac 1N}\eucl(N,2)- \delta.
\]
Then, there exists $v \in \mathcal{M}_{\sf eu}(\X)$ so that
\[
    \frac{\| \nabla( u - v)\|_{L^2(\mm)}}{\|\nabla u\|_{L^2(\mm)}} \le \eps.
\] 
\end{theorem}
\begin{proof} We can clearly assume that $\| u\|_{L^{2^*}(\mm)}=1$. Moreover by approximation it is also sufficient to prove the statement for $u \in W^{1,2}(\X)$ (see  Lemma \ref{lem:el trick}).

We proceed by contradiction and suppose that there exist $\eps>0$, a sequence $(\X_n,\sfd_n,\mm_n)$ of $\RCD(0,N)$ spaces with ${\sf AVR}(\X_n) \in (V,V^{-1})$ and a sequence  $u_n\in W^{1,2}(\X_n)\cap L^{2^*}(\mm_n)$ of non-constant functions satisfying
\begin{equation}
    \| u_n\|_{L^{2^*}(\mm_n)} \ge (A_n-1/n) \| \nabla u_n\|_{L^2(\mm_n)}, \label{eq:AVR sobolev proof}
\end{equation}
where $A_n\coloneqq  {\sf AVR}(\X_n)^{-\frac 1N}\eucl(N,2)$, and 
\begin{equation}
        \inf_{v\in\mathcal{M}_{\sf eu}(\X_n)} \frac{\| \nabla ( u_n  -v  )\|_{L^2(\mm_n)}}{\|\nabla u_n\|_{L^2(\mm_n)}} > \eps,\qquad \forall n \in \N.\label{eq:contradiction AVR}
\end{equation}
For every $\eta \in (0,1)$, let  $y_n \in \X_n$ and $t_n>0$ so that (arguing as for \eqref{eq:yn maximizing})
\[ 
1-\eta = \int_{B_{t_n}(y_n)} |u_n|^{2^*}\, \d \mm_n =  \sup_{y \in \X_n} \int_{B_{t_n}(y)} |u_n|^{2^*}\, \d \mm_n  ,\qquad n \in \N.
\]
Define now $\sigma_n \coloneqq  t_n^{-1}$ and  $(Y_n,\rho_n,\mu_n,y_n) \coloneqq  (\X_{\sigma_n}, \sfd_{\sigma_n},\mm_{\sigma_n},y_n)$, where $\sfd_{\sigma_n}\coloneqq \sigma_n \sfd_n$, $\mm_{\sigma_n}\coloneqq \sigma_n^N \mm_n$ and $u_{\sigma_n}\coloneqq \sigma_n^{-N/2^*}u_n \in W^{1,2}(Y_n)$. In particular, by scaling, for every $n\in\N$ we have
\[
1-\eta =  \int_{B_1(y_n)} |u_{\sigma_n}|^{2^*}\, \d \mu_n\qquad \text{and}\qquad  \| u_{\sigma_n}\|_{L^{2^*}(\mu_n)}  \ge (A_n-1/n) \| \nabla u_{\sigma_n}\|_{L^2(\mu_n)}.
    \]
By the assumption, we have the uniform bounds $2 V^{\frac 1N}\eucl(N,2) \le A_n  \le 2 V^{-\frac 1N}\eucl(N,2)$. Thus, up to subsequences, we can clearly suppose that $A_n\to A$, for some $A>0$ finite. We can now invoke Theorem \ref{thm:CC_Sobextremals} (the assumptions are satisfied as $\diam(Y_n)=+\infty$) with $\eta \coloneqq \eta_N/2$ and get that up to a subsequence   $(Y_n,\rho_n,\mu_n,y_n)$ pmGH-converges to some $\RCD(0,N)$ space $(Y,\rho,\mu,\bar y)$ 
supporting a Sobolev inequality \eqref{eq:convention} with constant $A>0,B=0$. Moreover we have $L^{2^*}$-strong convergence of $u_{\sigma_n}$ to a function $u \in W^{1,2}_{loc}(Y)$ attaining equality in this said Sobolev inequality and $\|\nabla u_{\sigma_n}\|_{L^2(\mm_{\sigma_n})}\to \|\nabla u\|_{L^2(\mu)}$. From \cite[Theorem 4.6]{NobiliViolo21} we have ${\sf AVR}(Y)=(\eucl(N,2)/A)^N $ and in particular $u$ satisfies the assumptions of Theorem \ref{thm:rigidity sharp Sob}, which gives that $Y$ is isomorphic to a $N$-Euclidean metric measure cone with tip  $z_0$ and
\[
u(y) = \frac{a}{(1+b\rho^2(y,z_0))^{\frac{N-2}{2}}}, \qquad y \in Y,
\]
for suitable $a \in \R, b>0$.

Take any $z_n \to z_0$. Then up to subsequence we can assume that $\mm_{\sigma_n}(B_1(z_n))\le C_N {\sf AVR}(Y)$ hold for every $n$. Writing $f(t) \coloneqq a(1+bt^2)^{\frac{2-N}{2}}$ for every $t \in \R^+$, recalling $|f|^{2^*}, |f'|^2\le C t^{-2N+2}$ and arguing as for \eqref{eq:growth Talenti}, we see that all the hypotheses of Lemma \ref{lem:recovery extremals}  are fulfilled both for $f\circ \rho(\cdot,z_0)$ and for $f'\circ \rho(\cdot,z_0)$. We therefore apply Lemma \ref{lem:recovery extremals} twice to get that  $f\circ \sfd_{\sigma_n}(\cdot,z_n)$ converges $L^{2^*}$-strong to $u$ 
and that $|f'|\circ\sfd_{\sigma_n}(\cdot,z_n) $ converges $L^2$-strong to $|\nabla u|$.
We can thus combine Lemma \ref{lem:coupling nablas} with the convergence of the gradient norms to deduce, from the parallelogram identity, that
\begin{equation}
   \lim_{n\to\infty} \| \nabla\big( u_{\sigma_n} - f\circ \sfd_{\sigma_n}(\cdot, z_n)\big)\|_{L^{2}(\mm_{\sigma_n})} = 0. \label{eq:un AVR gradient recovery}
\end{equation}
Scaling back, \eqref{eq:un AVR gradient recovery} becomes
\[
   \lim_{n\to\infty} \| \nabla\big( u_n - (\sigma_n^{N/2^*}f)\circ (\sigma_n\sfd_n(\cdot, z_n) ) \big)\|_{L^{2}(\mm_n)} = 0.
\]
This means that the sequence  $v_n \coloneqq   a\sigma_n^{N/2^*} (1+  b\sigma_n^2\sfd_n(\cdot,z_n)^2)^{\frac{2-N}{2}} \in \mathcal{M}_{\sf eu}(\X_n)$, satisfies
\[
   \limsup_{n\to\infty} \frac{\| \nabla ( u_n - v_n )\|_{L^2(\mm_n)}}{\|\nabla u_n\|_{L^2(\mm_n)}} =0,
\] 
having used that $\|\nabla u_n\|_{L^2(\mm_n)}\ge C_N\avr(\X_n)^{1/N} \|u_n\|_{L^{2^*}}\ge C_N V^{1/N}.$ This is a contradiction with \eqref{eq:contradiction AVR} and concludes the proof.
\end{proof}
From the above stability, the next corollary directly follows (proving also Corollary \ref{cor:introcorollary}).
\begin{corollary}\label{cor:sobconstant}
Let $\Xdm$  be an $\RCD(0,N)$ space with $N\in (2,\infty),{\sf AVR}(\X)>0$. Then
\[
 {\sf AVR}(\X)^{\frac 1N}\eucl^{-1}(N,2)  =\inf_{v \in \mathcal{M}_{\sf eu}(\X)} \frac{\| \nabla  v \|_{L^2(\mm)}}{\|  v \|_{L^{2^*}(\mm)}}.
\]
\end{corollary}

\appendix

\section{Concentration compactness: non-compact case}

Here we extend the concentration compactness tools for a sequence of converging $\RCD$ spaces (developed in \cite{NobiliViolo21} in compact setting) to the non-compact case. 
The main difference is that here mass can also escape to infinity and so we need an additional result (see Lemma \ref{lem:ConcComp1}).  Some additional technical convergence results will be also needed and proved in Section \ref{sec:technical conv}.

\subsection{Technical convergence lemmas}\label{sec:technical conv}
Throughout this part we fix a sequence $(\X_n,\sfd_n,\mm_n,x_n)$ of pointed $\RCD(K,N)$ spaces, $n \in \N\cup\{\infty\}$, for some $K\in\R, N \in (1,\infty)$ with $\X_n\overset{pmGH}{\to}\X_\infty$. We also fix a proper metric space $(\Z,\sfd)$ realizing the convergence via extrinsic approach \cite{GigliMondinoSavare13} (see Section \ref{sec:conv}). We start with a version of the Brezis-Lieb Lemma \cite{BrezisLieb83}.
\begin{lemma}[Brezis-Lieb type Lemma] \label{lem:BreLieb} 
	 Let $q,q'\in(1,\infty)$  and suppose that $u_n \in L^q(\mm_n)$ satisfy $\sup_n \|u_n\|_{L^q(\mm_n)}<+\infty$ and that $u_n$ converges in $L^{q'}$-strong to some $u_\infty\in L^{q'}\cap L^{q}(\mm_\infty)$. Then, for any sequence $v_n \in L^q(\mm_n)$ such that $v_n \to u_\infty$ strongly both in $L^{q'}$ and $L^{q}$, it holds 
	\begin{equation}\label{eq:brezislieb}
	    \lim_{n\to \infty }\int |u_n|^q\, \d \mm_n-\int |u_n-v_n|^q\, \d \mm_n=\int |u_\infty|^q\,\d \mm_\infty.
	\end{equation}
\end{lemma} 
\begin{proof}
The proof is the same as in \cite[Prop. 6.2]{NobiliViolo21}. Even if the argument there is done assuming finite reference measure, it is used only at the end when applying the H\"older inequality. In that step here is enough to multiply by an arbitrary $\phi \in C_{bs}(Z)$ and argue in the same way.
(Note also that the assumptions  $q\in[2,\infty)$ and $q'\in(1,q)$, even if present in the statement of \cite[Prop. 6.2]{NobiliViolo21} are actually not used in its proof).
\end{proof}
 We shall need an alternative version of the semicontinuity result \eqref{eq:Gamma liminf} to deal with locally Sobolev functions; we include a proof since we could not find it in the literature.
\begin{lemma}\label{lem:Ch is Lp-lsc}
Let $p \in (1,\infty)$ and suppose $(f_n)\subset W^{1,2}_{loc}(\X_n)$ is $L^p$-strong converging to $f_\infty$. Then 
\begin{equation}
      \|\nabla f_\infty \|^2_{L^2(\mm_\infty)} \le \liminf_{n\to\infty} \| \nabla f_n \|^2_{L^2(\mm_n)},
\label{eq:Ch is Lp-lsc}  
\end{equation}
(meaning that, if the right hand side is finite, then $f_\infty\in W^{1,2}_{loc}(\X_\infty)$ and \eqref{eq:Ch is Lp-lsc} holds).
\end{lemma}
\begin{proof}
Since $|f_n|\to |f_\infty|$ $L^p$-strongly (see \cite[a) in Prop. 3.3]{AmbrosioHonda17}) and $|\nabla f_n|=|\nabla |f_n||$ $\mm$-a.e.\ for every $f_n$, without loss of generality we can suppose $f_n,f_\infty$ nonnegative. If the liminf in \eqref{eq:Ch is Lp-lsc} is infinite, there is nothing to prove. So, let us assume that it is finite. For every $k \in \N$, we  consider $\varphi^k\in \LIP([0,\infty)$  with $\Lip(\varphi^k)\le 1$, $\phi^k(0)=0,$ converging point-wise to the identity as $k\uparrow \infty$ and such that that $\{\varphi^k(f_{n})\}_n$ is $L^2$-bounded. For instance we can take $\varphi^k(t) \coloneqq (t-1/k)^+ \wedge k$, indeed
\[\|\varphi^k(f_{n})\|^2_{L^2(\mm_{n})} \le k^2 \mm_{n}(\{ f_{n}  >1/k\})\le k^{2+p} \|f_{n}\|_{L^p(\mm_{n})}^p,\]
for every $n\in\N$. Again by \cite[a) in Prop. 3.3]{AmbrosioHonda17}, we have $\varphi^k(f_{n})$ is $L^p$-strong convergent to $\varphi^k(f_\infty)$. Moreover is also $L^2$-bounded, thus it is also  $L^2$-weak convergent to $\varphi^k(f_\infty)$. Then, by \eqref{eq:Gamma liminf} we have $\varphi^k(f_\infty) \in W^{1,2}(\X_\infty)$ and
\[ \| \nabla (\varphi^k(f_\infty)) \|_{L^2(\mm_\infty)}^2 \le \liminf_{n\to\infty} \| \nabla (\varphi^k(f_{n}))\|_{L^2(\mm_{n})}^2 \le \liminf_{n\to\infty} \| \nabla f_{n} \|_{L^2(\mm_{n})}^2<\infty,
\]
having used the fact that $\varphi^k$ is $1$-Lipschitz. By arbitrariness of $k >0$ and since $\varphi^k(f_\infty) \to f_\infty$ pointwise, we see by semicontinuity \eqref{eq:W12loc lsc} that \eqref{eq:Ch is Lp-lsc} follows.
\end{proof}

The following lemma allows extracting $L^2_{loc}$-converging subsequences from $W^{1,2}$-boundedness.
\begin{lemma}\label{lem:pmGHW12loc}
Let $p\ge 2$ and suppose $u_n \in W^{1,2}_{loc}(\X_n)$ converges $L^p$-weak to $u_\infty \in L^p(\mm_\infty)$ and $\sup_n \|\nabla u_n\|_{L^2(\mm_n)} <\infty$. Then, up to  a subsequence $u_n$ converges $L^2_{loc}$-strong to $u_\infty \in W^{1,2}_{loc}(\X_\infty)$ with $|\nabla u_\infty|\in L^2(\mm_\infty)$.
\end{lemma}
\begin{proof}
We first prove the $L^2_{loc}$ convergence.
Consider $\varphi \in \Lip_{bs}(\Z)$ (recall that $(\Z,\sfd)$ is a space realizing the convergence). Since $\sup_n\mm_n(B_R(x_n))<+\infty$, for every $R>0$, by H\"older inequality we have  $\sup_n \| \varphi u_n\|_{L^2(\mm_n)} <+\infty$. Analogously using the Leibniz rule, $\varphi u_n \in W^{1,2}(\X_n)$ with $\sup_n \| \nabla ( \varphi u_n) \|_{L^2(\mm_n)} <\infty$. Thus there exists a subsequence $(n_k)$ (see \cite[Theorem 6.3]{GigliMondinoSavare13}) such that  $\varphi u_{n_k}$ converges $L^2$-strong to some $v$, which must be equal to $\varphi u_\infty$ by uniqueness of weak limits. Hence the whole sequence $\phi u_n$ is $L^2$-strongly convergent to $\varphi u_\infty$. The fact that $u_\infty \in W^{1,2}_{loc}(\X_\infty)$ follows by the Mosco convergence of the Cheeger energies (see \eqref{eq:Gamma liminf}), indeed for every $\phi \in \LIP_{bs}(Z)$, $\rmCh(\varphi u_\infty)\le \liminf_n \rmCh( \varphi u_{n}) <\infty $. It remains to prove that $|\nabla u_\infty|\in L^2(\mm_\infty)$. Fix a ball $B\subset Z$  and take $\varphi \in \Lip_{bs}(\Z)$ equal to $1$ on $B$. Using \cite[Lemma 5.8]{AmbrosioHonda17}, we have 
$  \int_B  |\nabla u_\infty |^2 \, \d \mm_\infty = \int_B |\nabla (\varphi u_\infty) | \, \d \mm_\infty  \le \liminf_n \int_B  |\nabla (\varphi u_n)|^2 \, \d \mm_n \le \sup_n \|\nabla u_n\|_{L^2(\mm_n)}<\infty$. Where in the first and last step we used the locality of the gradient. By the arbitrariness of $B$ this implies $|\nabla u_\infty| \in L^2(\mm_\infty)$. 
\end{proof}

\begin{lemma}\label{lem:strongrecov}
Let $p\ge 2$ and  $u_\infty \in W^{1,2}_{loc}(\X_\infty)\cap L^{p}(\mm_\infty)$ with  $|\nabla u_\infty|\in L^2(\mm_\infty)$. Then, there exists a  sequence $u_n \in W^{1,2}_{loc}(\X_n)\cap L^{p}(\mm_n)$ that converges $L^{p}$ and $L^2_{loc}$-strong to $u_\infty$ and so that $|\nabla u_n|$ converges $L^2$-strong to $|\nabla u_\infty|$.
\end{lemma}
\begin{proof}
By Lemma \ref{lem:el trick} there exists a sequence $u_n \in W^{1,2}(\X_\infty)\cap L^p(\mm_\infty)$ such that $u_n \to u_\infty$ in $L^p(\mm_\infty)$ and $|\nabla u_n|\to |\nabla u_\infty|$ in $L^2(\mm_\infty).$ From \cite[Lemma 6.4]{NobiliViolo21}  (there written for compact spaces, but the same proof works in the present setting)  there exists a sequence $u_n^k \in W^{1,2}(\X_n)$ that converges $L^p$ and $W^{1,2}$-strong to $u_n$.
 By \cite[Theorem 5.7]{AmbrosioHonda17} this implies that $|\nabla u^k_n|$ converges $L^2$-strong to $|\nabla(\eta_k u_n)|$.  The conclusion then follows via diagonal argument. Finally the $L^2_{loc}$-strong convergence  follows from Lemma \ref{lem:pmGHW12loc}.
 \end{proof}
 We prove a convergence result for pairings (the case $p=2$ follows from \cite[Theorem 5.4]{AmbrosioHonda17}).
 \begin{lemma}\label{lem:coupling nablas}
     Let $p \in [2,\infty)$ and $u_n,v_n \in L^p(\mm_n)\cap W^{1,2}_{loc}(\X_n)$ be converging $L^p$-strong to $u_\infty,v_\infty$ respectively. Suppose that $u_\infty\in W^{1,2}_{loc}(\X_\infty)$, that $\| \nabla u_n\|_{L^2(\mm_n)} \to \| \nabla u_\infty \|_{L^2(\mm_\infty)}<+\infty$ and $\limsup_n \| \nabla v_n\|_{L^2(\mm_n)}<+\infty$. Then $v_\infty\in W^{1,2}_{loc}(\X_\infty)$, $|\nabla v_\infty| \in L^2(\mm_\infty)$ and
     \[ \lim_{n\to\infty}\int \la \nabla u_n , \nabla v_n \ra\,\d\mm_n = \int \la \nabla u_\infty , \nabla v_\infty \ra\,\d\mm_\infty.
     \]
 \end{lemma}
 \begin{proof}
 The fact that $v_\infty\in W^{1,2}_{loc}(\X_\infty)$ with $|\nabla v_\infty| \in L^2(\mm_\infty)$ follows from Lemma \ref{lem:Ch is Lp-lsc}.
 In particular by Cauchy-Schwarz $\la \nabla u_\infty , \nabla v_\infty \ra \in L^1(\mm_\infty)$. Let $t>0$ and notice that $u_n + tv_n$ converges $L^p$-strong to $u_\infty + t v_\infty$ by \eqref{eq:linearity Lpconvergence}. Applying again Lemma \ref{lem:Ch is Lp-lsc} we have $u_\infty+tv_\infty\in W^{1,2}_{loc}(\X_\infty)$ and
 \[
 \begin{aligned}
 \int 2t\la \nabla u_\infty , \nabla v_\infty \ra +& |\nabla u_\infty|^2 + t^2|\nabla v_\infty|^2 \,\d\mm_\infty =\int |\nabla (u_\infty+tv_\infty)|^2\,\d\mm_\infty \\
 &\overset{\eqref{eq:Ch is Lp-lsc}}{\le}\liminf_{n\to\infty} \int |\nabla (u_n+tv_n)|^2\,\d\mm_n \\
 & \le  2t\liminf_{n\to\infty}\int \la \nabla u_n , \nabla v_n \ra\,\d\mm_n +  ^2\limsup_n\int|\nabla v_n|^2 \,\d\mm_n +\int |\nabla u_\infty|^2 \d \mm_\infty.
 \end{aligned}
 \]
 Simplifying $\int |\nabla u_\infty|^2 \d \mm_\infty$, dividing by $t$ and sending $t \downarrow 0$ we obtain $\int \la \nabla u_\infty , \nabla v_\infty \ra \d \mm_\infty \le\liminf_{n\to\infty}\int \la \nabla u_n , \nabla v_n \ra\,\d\mm_n.$ Arguing analogously for $t<0$, we conclude.
 \end{proof}

\subsection{Concentration compactness principles}\label{sec:cc appendix}
Here we briefly extend two concentration compactness principles from \cite{Lions84,Lions85} (see also \cite{Struwe08}) for general sequences of probabilities on metric measure spaces.

The first deal with an arbitrary sequence of probability measures on varying ambient space. Compare also with the version \cite[Lemma 2.1]{AntonelliNardulliPozzetta22}. 
\begin{lemma}\label{lem:ConcComp1}
Let $(\Z,\sfd)$ be a complete and separable metric spaces and let $\nu_n\in \PP(\Z)$, for $n \in \N$. Then, up to a subsequence, one of the following holds:
\begin{itemize}
    \item[{\color{blue}  \rm i)}] {\sc Compactness}. There exists $(z_n) \subset \Z$ such that for all $\eps>0$, there exists $R>0$ satisfying \[ \nu_n(B_R(z_n)) \ge 1-\eps, \qquad \forall n \in \N.\]

    \item[{\color{blue}   \rm ii)}] {\sc Vanishing}. \[ \lim_{n\to\infty} \sup_{z \in \Z} \nu_n(B_R(z)) =0,\qquad \forall R>0.\]
    
    \item[{\color{blue}   \rm iii)}] {\sc Dichotomy}. There exists $\lambda \in (0,1)$ with $\lambda \ge \limsup_n  \sup_{z \in \Z} \nu_n(B_R(z))$, for all $R>0,$ so that: there exists $R_n\uparrow \infty$, $(z_n) \subset \Z$ and there are $\nu_n^1,\nu_n^2$ two non-negative Borel measures satisfying
    \[
    \begin{split}
        &0 \le \nu_n^1+\nu_n^2 \le \nu_n, \\
        &\supp( \nu_n^1) \subset B_{R_n}(z_n), \quad \supp( \nu_n^2) \subset \Z \setminus B_{10R_n}(z_n), \\
        & \limsup_{n\to\infty}   \ \big| \lambda -  \nu_n^1( \Z) \big|+ \big| (1-\lambda ) -  \nu_n^2( \Z) \big|  =0 .
    \end{split}
    \]
\end{itemize}
\end{lemma}
The above can be obtained arguing exactly as in \cite[Lemma I in Section 4.3]{Struwe08} and therefore its proof is omitted.  We briefly comment on the difference in case iii) with respect to \cite{Struwe08}:  our formulation of  case iii) using a sequence $R_n$  follows from the one used in \cite{Struwe08} (where $R$ is fixed depending on a parameter $\eps>0$)  with a diagonal argument (this is observed also in the proof of \cite[Theorem 4.9]{Struwe08}); the condition $\lambda \ge \limsup_n  \sup_{z \in \Z} \nu_n(B_R(z))$ (not present in \cite{Struwe08}) instead can be directly checked to hold by the way $\lambda$ is chosen in the proof.

The second principle is a concentration compactness result for the Sobolev embedding stating that concentration may occur only at countably-many points. With respect to \cite[Lemma 6.6]{NobiliViolo21}, here we extend the principle to deal with varying pmGH-convergent $\RCD$ spaces (hence, the difference arises when considering noncompact limit spaces). 
\begin{lemma}\label{lem:conccomp2}
Let $(\X_n,\sfd_n,\mm_n,x_n)$,  $n \in\N\cup\{\infty\}$, be  pointed $\RCD(K,N)$ spaces, $K \in \R$, $N \in(1,\infty)$  with $\X_n\overset{pmGH}{\to}\X_\infty$ and assume that $\X_n$ supports a Sobolev inequality \eqref{eq:convention} with uniformly bounded constants $A_n>0,B_n\ge 0$.

Suppose further that $u_n \in W_{loc}^{1,2}(\X_n) \cap  L^{2^*}(\mm_n)$ with $\sup_n \|\nabla u_n\|_{L^2(\mm_n)}<\infty $ is $L^2_{loc}$-strong converging to $u_\infty \in L^{2^*}(\mm_\infty)$ and suppose that $|\nabla u_n|^2 \mm_n \weakto \omega,$ $|u_n|^{2^*}\mm_n\weakto \nu$ in duality with $C_{bs}(\Z)$ and $C_b(\Z)$, respectively (where $(\Z,\sfd)$ is a fixed realization of the convergence). 
    
    Then, $u_\infty \in W^{1,2}_{loc}(\X_\infty)$ with $|\nabla u_\infty|\in L^2(\mm_\infty)$ and:
    \begin{itemize}
        \item[{\color{blue}  \rm i)}] there exists a countable set of indices $J$, points $(x_j)_{j\in J}\subset \X_\infty$ and  weights $(\nu_j)_{j\in J}\subset\R^+$ so that
        \[ \nu =|u_\infty|^{2^*}\mm_\infty +\sum_{j\in J}\nu_j\delta_{x_j};\]
        \item[{\color{blue}  \rm ii)}] there exists $(\omega_j)_{j\in J} \subset \R^+$ satisfying $\nu_j^{2/2^*} \le (\limsup_{n}A_n)\omega_j$ and such that
        \[\omega \ge |\nabla u_\infty|^2\mm_\infty +\sum_{j\in J}\omega_j\delta_{x_j}.\]
        In particular, we have $\sum_j \nu_j^{2/2^*}<\infty$.
    \end{itemize}
\end{lemma}
\begin{proof}
We subdivide the proof into two steps.

\noindent{\color{blue} \sc Step 1}. Suppose first that $u_\infty=0$. Then, the conclusion follows arguing as in Step 1  of \cite[Lemma 6.6]{NobiliViolo21} taking here $\varphi$ a Lipschitz and boundedly supported (instead of only Lipschitz) cut-off and using the assumed $L^2_{loc}$-strong convergence.

\noindent{\color{blue} \sc Step 2}.  For general $u_\infty$, the idea is to apply the above to `$u_\infty -u_n$' and then use a Brezis-Lieb lemma to recover the information for $u_\infty$. Take $\tilde u_n$  a recovery sequence given by Lemma \ref{lem:strongrecov} for $u_\infty$.   Thus, for every $\varphi \in \Lip_{bs}(\Z)^+$, we have $\varphi u_n$ is $L^2$-strong to $\phi u_\infty$ and $L^{2^*}$-bounded and  $\varphi \tilde u_n $ is $L^2$ and $L^{2^*}$-strong convergent to $\phi u_\infty$. Therefore Lemma \ref{lem:BreLieb} ensures
\begin{equation}
    \lim_{n\to\infty} \int |\varphi|^{2^*} |u_n|^{2^*}\, \d \mm_n - \int |\varphi|^{2^*}|u_n-\tilde u_n|^{2^*}\, \d \mm_n = \int |\varphi|^{2^*}|u_\infty|^{2^*}\, \d\mm_\infty,\label{BrezisLieb}
\end{equation}
Now define $v_n\coloneqq  u_n -\tilde u_n$ and notice that all the assumptions ensures that $v_n $ is $L^2_{loc}$-strong and $L^{2^*}$-weak convergent to zero.
From the bounds $|v_n|^{2^*} \le 2^{2^*} (|u_n|^{2^*} + |\tilde u_n|^{2^*})$ and $|\nabla v_n|^2 \le 2(|\nabla u_n|^2 + |\nabla \tilde u_n|^2)$ by tightness we can extract a not relabelled subsequence where $|v_n|^{2^*}\mm_n $ converge in duality with $C_b(\Z)$ to $\bar \nu$ and $|\nabla v_n|^2\mm_n$ converge in duality with $C_{bs}(\Z)$ to a finite Borel measure  $\bar \omega$. Then from Step 1, i),ii) hold true for $(v_n)$, for suitable weights $(\nu_j),(\omega_j)\subset \R^+$ and points $(x_j)\subset \X_\infty $. Then passing to the limit in \eqref{BrezisLieb} 
\[ 
    \int \varphi^{2^*} \d \nu- \int \varphi^{2^*}\d \bar \nu= \int \varphi^{2^*}|u_\infty|^{2^*}\, \d \mm_\infty,\qquad \forall \varphi \in \Lip_{bs}(\Z)^+.
    \]
This in turn implies $\nu = |u_\infty|^{2^*}\mm_\infty + \bar \nu  = |u_\infty|^{2^*}\mm_\infty + \sum_j \nu_j\delta_{x_j}$ that is point i). We pass  to prove ii) and therefore we need to show separately  that
\[
\begin{aligned}
&\omega(\{x_j\})=\bar \omega(\{x_j\})\ge \omega_j, \quad \forall \, j \in J, \\
&\omega\ge |\nabla u_\infty|^2\mm_\infty.
\end{aligned}
\]
The first can be verified arguing exactly as in Step 2 of \cite[Lemma 6.6]{NobiliViolo21} replacing the usage of  \cite[Theorem 5.7]{AmbrosioHonda17} with Lemma \ref{lem:strongrecov} above. For the second, we fix $\phi \in C_{bs}(\Z)$, $\phi \ge 0$, and $\nchi \in \LIP_{bs}(\Z)$ be such that $\nchi=1$ in $\supp (\phi)$. It is easy to check that $\nchi u_n$ is $W^{1,2}$-weak converging to $\nchi u_\infty$ (recall that $u_n\to u_\infty$ in $L^2_{loc}$).   Then, \cite[Lemma 5.8]{AmbrosioHonda17} ensures that 
\[
     \int \phi |\nabla u_\infty|^2\, \d \mm_\infty=\int \phi |\nabla (\nchi u_\infty)|^2\, \d \mm_\infty \le \liminf_{n\to\infty}\int \phi |\nabla (\nchi u_n)|^2\, \d \mm_n= \liminf_{n\to\infty}\int \phi |\nabla  u_n|^2\, \d \mm_n
     \]
By arbitrariness of $\varphi$, we showed ii) and the proof is now concluded.
\end{proof}

\section{Technical results}
In this appendix, we collect basic results about Sobolev inequalities and a version of the chain rule for the weak upper gradient. 
\begin{lemma}\label{lem:sobolev loc}
    	Let $\Xdm$ be an $\RCD(K,N)$ space, $N \in (2,\infty),K\in\R$, satisfying for $A>0$
	\begin{equation}\label{eq:euclidean type sobolev}
		\|u\|_{L^{2^*}(\mm)}\le A\| \nabla  u\|_{L^2(\mm)}, \qquad \forall u \in \LIP_c(\X),
	\end{equation}
where $2^*\coloneqq \frac{2N}{N-2}.$
Then \eqref{eq:euclidean type sobolev} holds also for all $u \in W^{1,2}_{loc}(\X)$ satisfying $\mm(\{|u|>t\})<+\infty$ for all $t>0.$
\end{lemma}
\begin{proof}
It is enough to prove \eqref{eq:euclidean type sobolev} for non-negative functions. First note that \eqref{eq:euclidean type sobolev} holds for every $u\in W^{1,2}(\X)$, by  density in energy of Lipschitz functions \cite{AmbrosioGigliSavare11-3} and by the lower semicontinuity of the $L^{2^*}$-norm with respect to $L^2$-convergence.  For a general $u\ge0$ as in the hypotheses, if $\int |\nabla u|^2 \d \mm=+\infty$ there is nothing to prove, otherwise take $u_n\coloneqq((u-1/n)^+)\wedge n\in W^{1,2}(\X)$ (since $u_n,|\nabla u_n| \in L^2(\mm)$) and then send $n \to + \infty)$. 
\end{proof}

\begin{lemma}[Local Sobolev embedding]\label{lem:local sobolev embedding}
	Let $\Xdm$ be an
	$\RCD(K,N)$ space for some $K\in \R, N\in(2,\infty)$ and set $2^*\coloneqq 2N/(N-2)$. Then  exists $\tilde r_{{K^-},N}>0$ (with $\tilde r_{0,N}=+\infty$) such that for every $B_R(x)\subsetneq \X$, $R\le \tilde r_{{K^-},N}$ it holds
	\begin{equation}\label{eq:local sobolev}
		\|u\|_{L^{2^*}(\mm)}\le \frac{C_{N,K}R}{\mm(B_R(x))^{1/N}} \|\nabla u\|_{L^2(\mm)}, \qquad \forall u\in W^{1,2}_0(B_{R/2}(x)).
	\end{equation}
\end{lemma}
\begin{proof}
It is enough to prove the statement for $u \in \LIP_c(B_{R/2}(x)).$ Thanks to the uniformly locally doubling property of $\Xdm$ and the validity of a local $(1,1)$-Poincar\'e inequality (\cite{Rajala12}), from the results in \cite{HajlaszKoskela00} the following Sobolev-Poincar\'e inequality holds 
	\begin{equation}\label{eq:improved poincaret}
		\Big(\fint_{B_R(x)} |f-f_{B_R(x)}|^{2^*} \, \d \mm \Big)^\frac{1}{2^*} \le C(N,K,R_0)R \Big(\fint_{B_{2R}(x)} |\nabla  f|^{2}\, \d \mm \Big)^\frac{1}{2}, \qquad \forall \, f \in \LIP(\X),
	\end{equation}
	for every $R\le R_0$ and
	where  $f_{B_R(x)}\coloneqq\fint_{B_R(x)}f\, \d \mm$ (see also \cite{BB13}). Moreover if $K\ge 0$, the constant $C(N,K,R_0)$ can be taken independent of $R_0.$
	
	Hence applying \eqref{eq:improved poincaret} to $u \in \LIP_c(B_{R/2}(x))$ we can write
	\begin{align*}
		&\Big(\int_{B_R(x)}|u|^{2^*} \, \d \mm \Big)^\frac{1}{2^*} \le C_{N,K}R\frac{\mm(B_R(x))^{1/2^*}}{\mm(B_{2R}(x))^{1/2}} \Big(\int_{B_{2R}(x)} |\nabla  u|^2\Big)^\frac{1}{2}+\mm(B_R(x))^{1/2^*-1} \int_{B_{R/2}(x)} |u|\d \mm\\
		&\le C_{N,K}R\mm(B_R(x))^{-1/N} \Big(\int_{B_{2R}(x)} | \nabla u|^2\Big)^\frac{1}{2}+
		\frac{\mm(B_{R/2}(x))^{1-1/2^*}}{\mm(B_{R}(x))^{1-1/2^*}}\Big(\int_{B_{R/2}(x)}|u|^{2^*} \, \d \mm \Big)^\frac{1}{2^*},
        \end{align*}
        where we have used that $\supp(u)\subset B_{R/2}(x).$ Thanks to the reverse doubling inequality (recall \eqref{eq:reverse doubling}), assuming $R\le R_{K^-,N},$ we can absorb the rightmost term inside the left-hand side of the above to obtain \eqref{eq:local sobolev} as desired.
\end{proof}

A technical result needed in this note is a chain rule for the composition with an absolutely continuous function $\phi$, which we could not find in the literature (see \cite{Gigli14} or \cite{GP20} for the classical one with $\phi$ Lipschitz).
\begin{lemma}[Chain rule for composition with ${\sf AC}$-functions]\label{lem:chain rule}
    Let $\Xdm$ be a proper metric measure space and $u \in \LIP_{loc}(\Omega)$ with $\Omega \subset \X$ open. Let $\phi \in {\sf AC}_{loc}(I)$ with $I$ open interval such that $u(\Omega')\subset\subset I$ for every $\Omega'\subset\subset \Omega$. Suppose also that $|\phi'(u)||\nabla u| \in L^2_{loc}(\Omega)$.
    
    Then $\phi(u)\in W^{1,2}_{loc}(\Omega)$ and $|\nabla \phi(u)|= |\phi'(u)||\nabla u|$ $\mm$-a.e..
\end{lemma}
\begin{proof}
Up to subtracting a constant, we can assume that $0 \in I$ and $\phi(0)=0$. Then with a cut-off argument we can reduce to the case when $u \in \LIP_{c}(\X)$ and $ \phi \in {\sf AC}(\R)$ with compact support and $\phi(0)=0$.  We argue by approximation and define functions $\phi_n\in \LIP(\R)$ by
    \[
    \phi_n(t)\coloneqq \int_0^t -n\vee \phi'(s) \wedge n\, ds.
    \]
    Clearly $\phi_n\to \phi$ pointwise in $\R.$ By the usual chain rule for Lipschitz composition we have that $\phi_n(u)\in W^{1,2}(\X)$ with $|\nabla \phi_n(u)|=|\phi_n'(u)||\nabla u|\le |\phi'(u)||\nabla u|$, $\mm$-a.e., where we have used that $|\phi_n'|\le |\phi'|$ a.e.. 
    In particular  the sequence $|\nabla \phi_n(u)|$ is bounded in $L^2(\mm)$. Moreover $\phi_n(u)\to \phi(u)$ pointwise and from the lower semicontinuity of the minimal weak upper gradient (see, e.g.,  \cite[Prop. 2.1.13]{GP20}) we deduce that $\phi(u)\in W^{1,2}(\X)$ and 
    \begin{equation}\label{eq:chain ineq}
        |\nabla  \phi(u)|\le  |\phi'(u)||\nabla u|,\qquad  \text{$\mm$-a.e..}
    \end{equation}
    The  equality in \eqref{eq:chain ineq} then follows with a standard argument (see e.g.\ \cite[Theorem 2.1.28]{GP20}).
\end{proof}

\medskip

\textbf{Acknowledgements.} F.N. is supported by the Academy of Finland
Grant No. 314789. I.Y.V. is supported by the Academy of Finland
projects \emph{Incidences on Fractals}, Grant No. 321896 and
\emph{Singular integrals, harmonic functions, and boundary regularity in Heisenberg groups}, Grant No. 328846.

    \parskip0pt 
\itemsep0pt
 %   \bibliographystyle{siam}
  % \bibliography{biblio} 
\def\cprime{$'$} \def\cprime{$'$}

\end{document}